\documentclass[reqno,10pt]{amsart}
\usepackage[dvipsnames,usenames]{color} 
\usepackage[toc,page]{appendix}

\usepackage{tikz} 
\usepackage{mathcomp,wasysym}  

\usepackage{hyperref}
\usepackage[mathscr]{euscript}  
        
\usepackage{cite}      
\usepackage{graphicx}  
\usepackage[all]{xy} \xyoption{arc} \xyoption{color}

\usepackage{amsmath} 
\usepackage{amsthm} 
\usepackage{amsfonts}  
\usepackage{amssymb} 

\usepackage{verbatim}

\DeclareMathAlphabet{\mathpzc}{OT1}{pzc}{m}{it}

\newcommand{\vertiii}[1]{{\left\vert\kern-0.25ex\left\vert\kern-0.25ex\left\vert #1 
    \right\vert\kern-0.25ex\right\vert\kern-0.25ex\right\vert}}

\numberwithin{equation}{section} 

\newtheorem{theorem}{Theorem}[section]
\newtheorem{proposition}[theorem]{Proposition}
\newtheorem{lemma}[theorem]{Lemma}

\newtheorem{remark}[theorem]{Remark}

\newcommand{\jeps}{\mathcal{J}_\kappa}
\newcommand{\jkap}{\mathcal{J}_\epsilon}
\newcommand{\jdel}{\mathcal{J}_\delta}
\newcommand{\pat}{\partial_t}

\def\bdy #1{{\partial #1\hspace{1pt}}}
\def\cls #1{\overline {#1}}

\def\bbR{{\mathbb R}}

\def\bbT{{\mathbb T}}

\def\bbZ{{\mathbb Z}}

\def\div{{\operatorname{div}}}

\def\curl{{\operatorname{curl}}}

\def\id{{\text{Id}}}

\def\V{{\mathcal V}}

\def\rN{{\rm N}}

\def\rT{{\rm T}}

\def\p{{\partial\hspace{1pt}}}

\def\jump#1{{[\hspace{-2pt}[#1]\hspace{-2pt}]}}
\def\bigjump#1{{\big[\hspace{-3pt}\big[#1\big]\hspace{-3pt}\big]}}
\def\Bigjump#1{{\Big[\hspace{-4.5pt}\Big[#1\Big]\hspace{-4.5pt}\Big]}}
\def\({{(\hspace{-2pt}(}}
\def\){{)\hspace{-2pt})}}

\def\smallexp#1{{\text{\small #1}}}

\def\XXint#1#2#3{{\setbox0=\hbox{$#1{#2#3}{\int}$}
\vcenter{\hbox{$#2#3$}}\kern-.5\wd0}}

\title[The well-posedness of the Muskat equations]{Well-posedness of the Muskat problem with  $H^2$ initial data}

\author[C.H. A. Cheng]{C.H. Arthur Cheng}
\email{cchsiao@math.ncu.edu.tw}
\address{Department of Mathematics, National Central University, Jhongli City, Taoyuan County, 32001, Taiwan ROC}

\author[R.Granero-Belinch\'{o}n]{Rafael Granero-Belinch\'{o}n}
\email{rgranero@math.ucdavis.edu}
\address{Department of Mathematics, University of California, Davis, CA 95616, USA}

\author[S. Shkoller]{Steve Shkoller}
\email{shkoller@math.ucdavis.edu}
\address{Department of Mathematics, University of California, Davis, CA 95616, USA}

\date{December 23, 2014}

\begin{document}

\begin{abstract}
We study the dynamics of the interface between two incompressible fluids in a two-dimensional porous medium whose flow is
modeled by the Muskat equations. For the two-phase Muskat problem, we establish global well-posedness and {\it decay to equilibrium} for small 
$H^2$ perturbations of the rest state. For the one-phase Muskat problem, we prove local well-posedness for $H^2$ initial data of arbitrary size.  
Finally,  we show that  solutions to the Muskat equations  instantaneously   become infinitely smooth.
\end{abstract}

\maketitle
{\small
\tableofcontents}

\section{Introduction}
We consider the  two-phase Muskat moving free-boundary problem:
\begin{subequations}\label{laplacian}
\begin{alignat}{2}
\Delta P^\pm &= 0 \qquad&&\text{in}\quad\Omega^\pm(t)\,,\\
\jump{P} &= \Gamma(t)\cdot e_2 &&\text{on }\Gamma(t),\\
\jump{\nabla P\cdot n} &= 0 &&\text{on }\Gamma(t),\\
\partial\Omega^+(t)\cap \partial\Omega^-(t) &= \Gamma(t) && \forall\,t\geq0\,,\\
\V(\Gamma(t)) &= -\nabla P^\pm \cdot n \qquad&&\text{on}\quad\Gamma(t)\,,
\end{alignat}
\end{subequations}
where  $\Omega^+(t)$ and $\Omega^-(t)$ denote the time-dependent fluid domains associated with the two phases, $\Gamma(t)$ denotes the free 
boundary, $\Gamma(t)\cdot e_2$ is the second component of its parametrization, and $\V(\Gamma(t))$ is its normal velocity. We use the notation $\jump{f} = f^+ - f^-$ to  denote the jump of a function $f$ across $\Gamma(t)$. The problem (\ref{laplacian}) arises in the literature as the Hele-Shaw cell (with gravity) or the Muskat problem.

Many recent results on the Muskat problem rely on the fact that equations (\ref{laplacian}a-e) can be rewritten as a system of equations for the interface
$$
\Gamma(t) = (\psi_1(t,x_1), \psi_2(t,x_1)), \qquad x_1 \in \mathbb{R}  \,, \ \ t \in [0,T]\,,
$$
taking  the form
$$
\pat \psi=T[\psi],
$$
where $T[\psi]$ is a  highly nonlinear  singular integral operator, whose linearization (about a flat interface) behaves like  $\sqrt{-\Delta}$.
In order to establish existence theorems for the system (\ref{laplacian}),
this singular-integral-operator approach makes  extensive use of the explicit integral kernel representations for the operator $T$ for 
 the following fluid domains (or geometries):
 \begin{alignat*}{2}
&\text{\bf(a)} \ \     \overline{\Omega^+(t)}\cup \Omega^-(t)  && =   \bbR^2 \,, \\
&\text{\bf(b)} \ \   \overline{\Omega^+(t)}\cup \Omega^-(t)   && =   \bbT\times\bbR\,,\\
&\text{\bf(c)} \ \   \overline{\Omega^+(t)}\cup \overline{\Omega^-(t)}   && =   \bbR\times[-l,l] \,.
\end{alignat*} 
In the case of  general domain geometries, we are not aware of any existence and regularity theories.

The classical problem (\ref{laplacian}a-e) is related to both the (two-phase) Stefan problem
\begin{subequations}\label{stefan}
\begin{alignat}{2}
\partial_tP^\pm-\Delta P^\pm &= 0 \qquad&&\text{in}\quad\Omega^\pm(t)\,,\\
P^\pm&= 0 &&\text{on }\Gamma(t),\\
\partial\Omega^+(t)\cap \partial\Omega^-(t) &= \Gamma(t) && \forall\,t\geq0\,,\\
\V(\Gamma(t)) &= \jump{\nabla P\cdot n} \qquad&&\text{on}\quad\Gamma(t)\,,
\end{alignat}
\end{subequations}
and also with the Muskat problem with variable permeability $\beta(x)$,
\begin{subequations}\label{laplacianpermeability}
\begin{alignat}{2}
\text{div}\,\left(\beta(x)\nabla P^\pm\right) &= 0 \qquad&&\text{in}\quad\Omega^\pm(t)\,,\\
\jump{P} &= \Gamma(t)\cdot e_2 &&\text{on }\Gamma(t),\\
\jump{\nabla P\cdot n} &= 0 &&\text{on }\Gamma(t),\\
\partial\Omega^+(t)\cap \partial\Omega^-(t) &= \Gamma(t) && \forall\,t\geq0\,,\\
\V(\Gamma(t)) &= -\nabla P^\pm \cdot n \qquad&&\text{on}\quad\Gamma(t)\,.
\end{alignat}
\end{subequations}

Herein,  we introduce a new method to analyze the system  (\ref{laplacian}a-e), which is based on the analysis of the partial differential 
equations rather than any associated integral kernel.   Our methodology can treat the two-phase 
Muskat problem  with two different viscosities or with a non-constant permeability.   Our  method can also be applied to the Stefan 
problem \cite{HaSh2014}, to the free-boundary problem for the incompressible Euler equations \cite{CoSh2007,CoSh2014}, as well as
to the compressible Euler equations \cite{CoSh2012, CoHoSh2013} . 
One of the main interests of this new method is that it can be adapted to several space dimensions and arbitrary domain geometries 
$
\Omega^+(t)\cup \Omega^-(t).
$

\subsection{Darcy's law}
The Muskat problem, introduced in \cite{Muskat}, models the evolution of two fluids of varying density in a two-dimensional porous medium. 
The presence of the solid matrix inside the porous medium has an important consequence:   the usual fluid equations for the conservation 
of momentum are replaced with the {\it empirical}  Darcy's Law (see \cite{bear, NB}) given by
\begin{equation}\label{basic-model}
\frac{\mu}{\beta}u=-\nabla p-(0,g\rho)^T,
\end{equation} 
where $\mu,\rho$ are the viscosity and the density of the fluid, respectively,  $\beta$ is the permeability of the medium, $p$ is the pressure, and $g$ is the 
acceleration due to gravity.   
 As (\ref{basic-model}) is a model of aquifers, oil wells or geothermal reservoirs, this problem is of practical importance in geoscience (see, for 
 example, \cite{CF,Parseval-Pillai-Advani:model-variation-permeability} and the references therein);  moreover, it has also been considered as a model for the velocity of cells in tumor growth (see \cite{F, P}).   

The movement of a fluid trapped between two parallel vertical plates, which are separated by a very narrow distance, is known as the Hele-Shaw  cell problem (see \cite{HeleShaw:motion-viscous-fluid-parallel-plates}). The equations of motion in a Hele-Shaw cell are
$$
\frac{12\mu}{d^2}u=-\nabla p-(0,g\rho)^T,
$$
where $d$ is the distance between the plates. The similarity of both problems is obvious and, in fact, the Muskat problem is equivalent to the two-phase Hele-Shaw problem with gravity.

\subsection{The Muskat problem set in various geometries}
We shall consider various domain geometries in this paper, and we begin with the case of a domain with infinite depth.
\subsubsection{The infinitely-deep case} Let $(u^\pm,p^\pm)$ denote the velocity and the pressure in the fluid domains $\Omega^\pm(t)$, and let $\Gamma(t)$ denote the material interface between $\Omega^+(t)$ and $\Omega^-(t)$; that is, $\Gamma(t) = \cls{\Omega^+(t)}\cap \cls{\Omega^-(t)}$. Setting,  the permeability $\beta\equiv1$, the two-phase Muskat problem has the following Eulerian description:
\begin{subequations}\label{HS_Eulerian}
\begin{alignat}{2}
\mu^\pm u^\pm + \nabla p^\pm &= - \rho^\pm e_2 \qquad&&\text{in}\quad\Omega^\pm(t)\,,\\
\operatorname{div}  u &= 0 &&\text{in}\quad\Omega^\pm(t)\,,\\
\V(\Gamma(t)) &= u^\pm \cdot n \qquad&&\text{on}\quad\Gamma(t)\,,\\
\Omega^\pm(0) &= \Omega^\pm &&\text{on}\quad\{t=0\}\,,\\
\overline{\Omega^+(t)}\cup\Omega^-(t) &= \bbR^2 &&\text{for every }\quad t\geq 0\,,
\end{alignat}
\end{subequations}
where $e_2 = (0,1)$, $n(\cdot ,t)$ is the outward pointing unit normal on $\partial\Omega^-(t)$. In particular, we consider the case that
$$\Gamma(t) = (x_1, h(x_1,t))$$
is the graph of the height function $h(x_1,t)$, and  we assume that either $x_1 \in \mathbb{T}  ^1$, or that $x_1 \in \mathbb{R}  ^1$ and that $h(x_1,t)$
vanishes  at infinity.  It follows that the two time-dependent fluid domains $\Omega^\pm(t)$ are given by
$$
\Omega^+(t) = \big\{(x_1,x_2)\,\big|\, x_2 > h(x_1,t)\big\}\,,\quad \Omega^-(t) = \big\{(x_1,x_2)\,\big|\, x_2 < h(x_1,t)\big\}\,. 
$$
Since $ \operatorname{div} u^\pm=0$, we must have that $\jump{u\cdot n}=0$ on $\Gamma(t)$;  furthermore, as we assume that the effect of surface tension is negligible\footnote{Our methodology can treat the Muskat problem with surface tension in the same way.},  
we set 
$$\jump{p}= 0 \text{ on  } \Gamma(t)\,.$$

\subsubsection{The finitely-deep case with general geometry}
We shall additionally consider  geometries  which generalize the \emph{infinitely-deep}  case that $\Omega^+(t)\cup\Omega^-(t)=\bbR^2$ or
the \emph{confined} case that
 $\overline{\Omega^+(t)}\cup\overline{\Omega^-(t)}=\bbR\times[-l,l]$ (and $\|h\|_{L^\infty}<l$). 
  
Let $\tilde{t}(x_1)$ and $\tilde{b}(x_1)$ be two smooth functions. Given two constants $c_t>0$, $c_b<0$, we write 
$$
b(x_1)=c_b+\tilde{b}(x_1),\;t(x_1)=c_t+\tilde{t}(x_1).
$$ 
We assume that the the two fluids flow in  bounded domains of the type 
$$
\Omega^+(t)\cup\Omega^-(t)=\{(x_1,x_2),\,b(x_1)<x_2<t(x_1)\},\text{ for every } t\geq0; \eqno{\rm(\ref{HS_Eulerian}e')}
$$
thus,  each phase is given by
$$
\Omega^+(t)=\{(x_1,x_2),\, x_1\in\bbR,\,h(x,t)<x_2<t(x)\},
$$
and
$$
\Omega^-(t)=\{(x_1,x_2),\, x_1\in\bbR,\,b(x)<x_2<h(x,t)\}.
$$
Note that additional impervious boundary conditions must  be added to the system \eqref{HS_Eulerian} on the {\it fixed} bottom and top boundaries.   
These are given by
$$
u\cdot n=0\text{ at }\partial(\overline{\Omega^+(t)}\cup\overline{\Omega^-(t)}). \eqno{\rm(\ref{HS_Eulerian}f)}
$$ 
Finally, we assume that the initial height function $h_0$ satisfies
$$
b(x_1)<h_0(x_1)<t(x_1).
$$
\subsubsection{The one-phase Muskat problem} We shall also  consider  the one-phase Muskat problem, corresponding to the case that
 $(\mu^+,\rho^+)=(0,0)$. In other words, only one fluid flows through the  porous medium,  and the ``top'' phase corresponds to vacuum. 
 Furthermore, we consider the case that the interface is periodic (so $x_1\in\bbT$). Then,  our time-dependent domain is given by
$$
\Omega(t) = \bbT\times\big\{c_b<x_2 < h(x_1,t)\big\}\text{ for every }\quad t\geq 0\,,
$$
with moving boundary
$$
\Gamma(t) = \bbT\times\big\{x_2 = h(x_1,t)\big\}\text{ for every }\quad t\geq 0\,.
$$
To simplify notation for the one-phase problem, we set $(\mu^-,\rho^-)=(1,1)$. 
We  again use $(u,p)$ to denote  the velocity and the pressure of this fluid in the fluid domain 
$\Omega(t)$ with 
free boundary $\Gamma(t)$.  The one-phase Muskat  problem is written as

\begin{subequations}\label{HS_Eulerian_Onephase}
\begin{alignat}{2}
u+ \nabla p &= - e_2 \qquad&&\text{in}\quad\Omega(t)\,,\\
\div u &= 0 &&\text{in}\quad\Omega(t)\,,\\
\V(\Gamma(t)) &= u \cdot n \qquad&&\text{on}\quad\Gamma(t)\,,\\
u \cdot e_2 &= 0\qquad&&\text{on}\quad\{x_2=c_b\}\,,\\
p &= 0 \qquad&&\text{on}\quad\Gamma(t)\,,
\end{alignat}
\end{subequations}
where $e_2 = (0,1)$, $n(\cdot ,t)$ is the outward pointing unit normal on $\Gamma(t)$, and $\V(\Gamma(t))$ is the normal velocity of $\Gamma(t)$. 
 As we only have one phase,  (\ref{HS_Eulerian_Onephase}e) expresses the continuity of the pressure on $\Gamma(t)$.  Note, also,  that we have added
the  impermeable boundary condition on the fixed bottom boundary in(\ref{HS_Eulerian_Onephase}d).

\subsection{The Rayleigh-Taylor stability condition}
The Rayleigh-Taylor stability (or sign) condition  is defined as
\begin{equation*}
RT(t)=\Bigjump{\frac{\partial p}{\partial n}}=-(\nabla p^-(\Gamma(t))-\nabla p^+(\Gamma(t)))\cdot n>0.
\end{equation*}
Due to the incompressibility of the fluids, and using Darcy's law together with the fact that the curve can be parametrized as a graph, the Rayleigh-Taylor stability condition reduces to the following expression:
\begin{equation}\label{RTstable}
RT(t)=(\mu^--\mu^+)u\cdot n+\frac{\rho^--\rho^+}{\sqrt{1+(h'(x))^2}}=-\jump{\mu}u\cdot n-\frac{\jump{\rho}}{{\sqrt{1+(h'(x))^2}}} >0.
\end{equation}
In particular, for the case of two equal viscosities $\mu^-=\mu^+$, the fluids are in the stable regime if the lighter fluid is above the heavier fluid.
 Our research focuses on the stable case, so, henceforth, we shall assume that $\jump{\rho}<0$.

Note  that in the one-phase Muskat problem, the Rayleigh-Taylor stability condition reduces to
\begin{equation}\label{RTonephase}
RT(t)=-\nabla p^-(\Gamma(t))\cdot n>0.
\end{equation}
This stability condition is ubiquitous in free boundary problems;  it  also appears  in the  Stefan problem, the water waves problem, the incompressible Euler
equations, the compressible Euler equations with physical-vacuum boundary, and the MHD equations.  
When the initial data  does not verify the Rayleigh-Taylor stability condition, then the Muskat problem is ill-posed (see, for instance, \cite{c-g07,CGO}).   
It has also been shown  for the Muskat problem that there exists initial data such that the Rayleigh-Taylor stability condition can  break-down in finite time \cite{ccfgl,CGO, GG}.

We note that if the height function $h( \cdot , t) $ (which represents the moving interface $\Gamma(t)$) is small in certain norms, and if we assume that $\jump{\rho}<0$, 
then the Rayleigh-Taylor stability condition is achieved 
without any other hypothesis on the initial data. In the case of the unbounded, one-phase Muskat problem,  it is known that the Rayleigh-Taylor stability 
condition is automatically satisfied due to Hopf's Lemma and Darcy's Law (see \cite{ccfgonephase});  however, in the one-phase case with a  flat, bounded
 domain, it is not clear that the Rayleigh-Taylor stability condition is automatically satisfied, because of  a non-zero Neumann boundary  condition on the
 fixed bottom boundary.

\subsection{Prior results on the Muskat problem and related models}
Free-boundary problems for incompressible fluids in a porous medium have been extensively studied in recent years. 

For the Muskat problem with  fluids having the same viscosities ($\jump{\mu}=0$), the qualitative behavior for arbitraraly large initial data is well understood. In particular, for the infinitely-deep case, C\'ordoba \&  Gancedo proved the local existence of solutions for $H^3(\bbR)$ initial data in the stable Rayleigh-Taylor regime and the ill-posed character of the Muskat problem in the unstable Rayleigh-Taylor regime in \cite{c-g07}, a maximum principle for $\|h(t)\|_{L^\infty}$ in \cite{c-g09}, and local existence in the case with more than two phases in \cite{c-g10}. In a remarkable paper, Castro, C\'ordoba, Fefferman, Gancedo \& L\'opez-Fern\'andez \cite{ccfgl} proved the existence of turning waves, \emph{i.e.} interfaces such that there exists $T_{1}$ such that
$$
\limsup_{t\rightarrow T_{1}} \|h'(t)\|_{L^\infty}=\infty.
$$
Later, Castro, C\'ordoba, Fefferman \& Gancedo obtained in \cite{castro2012breakdown} the existence of curves showing finite-time singularities. These curves correspond to analytic initial data in the Rayleigh-Taylor stable regime such that there exists $T_{1}$ and $T_2$ such that, at $t=T_1$, the solution 
enters  the Rayleigh-Taylor unstable regime and later, at $t=T_2$, is no longer  $C^4$.

The confined case when the two viscosities are the same ($\jump{\mu}=0$) has been treated by C\'ordoba, Granero-Belinch\'on \& Orive \cite{CGO}. When the porous medium is inhomogeneous, the evolution of the interface has been studied by Berselli, C\'ordoba \& Granero-Belinch\'on \cite{BCG} and G\'omez-Serrano \& Granero-Belinch\'on \cite{GG}. Ambrose \cite{AmbroseST} studied the limit of zero surface tension  for initial data which satisfies (\ref{RTstable}). For further results, see also the review by Castro, C\'ordoba \& Gancedo \cite{Castro-Cordoba-Gancedo:recent-results-muskat}.

For the related Hele-Shaw cell problem,  Constantin \& Pugh \cite{Peter}, using complex analysis tools, proved the stability and exponential decay of solution. Chen \cite{chen1993hele} studied the two-phase Hele-Shaw problem with surface tension and proved global well-posedness for small enough initial interfaces. Elliot \& Ockendon \cite{elliott1982weak} proved the existence of weak solutions, while Escher \& Simonett \cite{ES} obtained local, classical solutions in multiple space dimensions. Escher \& Simonett \cite{escher1998center} proved global existence and stability near spherical shapes using center manifold theory.
The global existence and decay for solutions of the one-phase Hele-Shaw  problem with
various fluid injection-rates was studied by Cheng, Coutand \& Shkoller \cite{cheng2012global}. 

Returning to the Muskat problem,
when the initial data is assumed to be small in certain lower-order norms and the two fluid viscosities are equal, there are several available results for global-in-time solutions. In \cite{ccgs-10},  C\'ordoba,  Constantin, Gancedo \& Strain proved the global existence of $H^3$ Sobolev class solutions for initial data with \emph{small} derivative in the Wiener algebra
$A(\bbR)$, and global existence of Lipschitz (weak) solutions for initial data with 
\begin{equation}\label{sizecon}
\|h'_0\|_{L^\infty}<1.
\end{equation}
Therein, the authors also proved  an $L^2$ energy balance.   The global weak solution  of \cite{ccgs-10} was later 
extended to the confined case by Granero-Belinch\'on in \cite{G}.
 It is worth noting that, due to the effect of the impervious boundaries, the size restrictions on the data are not as clear as \eqref{sizecon} and for
the confined setting, involve $\|h_0\|_{L^\infty}, \|h'_0\|_{L^\infty},$ and the depth.

Very recently, in 
\cite{ccgs-13},  C\'ordoba,  Constantin,  Gancedo,  Rodr\'iguez-Piazza \& Strain obtained global existence for small data in the case of a two-dimensional interface;   furthermore, among other results, they proved the existence of a global solution in $H^2$ for data with \emph{small}  derivative in the Wiener algebra $A(\bbR)$,  and the existence of a global solution in 
$H^{1.5}$ if the initial data is also in the Wiener algebra $A(\bbR)$ and satisfies a smallness assumption. 
We remark that these global-in-time existence results are for initial data of {\it medium}-size, in the sense that initial data must be bounded
by constants of  $O(1)$.

In the case of two fluids with different viscosities, there are fewer results. The local existence for arbitrary $\mu^\pm$, $\rho^\pm$ and $H^3$ data was 
proven by  C\'ordoba, C\'ordoba \& Gancedo in \cite{c-c-g10}.   In the case of surface tension,  Escher \&  Matioc \cite{e-m10} and  Escher, Matioc \& Matioc \cite{escher2011generalized} established local and global existence, and stability,  in the little H\"{o}lder spaces.

The singularity  formation for the one-phase case (when $\mu^+=\rho^+=0$) has been studied by Castro, C\'ordoba, Fefferman \& Gancedo in \cite{ccfgonephase} where they proved the existence of the so-called interface ``splash'' singularity  wherein a locally smooth interface self-intersects at a point.  C\'ordoba \& Pern\'as-Casta\~no in \cite{cponephase} proved the non-existence of ``splat'' singularity, in which a locally smooth interface self-intersects on a curve. Gancedo \& Strain \cite{gancedo2013splasnqgmsukat} proved that the Muskat problem with three different fluids cannot develop a ``splash'' singularity in finite time.    In related work, Fefferman, Ionescu \& Lie \cite{FeIoLi2013} and Coutand \& Shkoller \cite{CoSh20142} have shown that a finite-time splash singularity  cannot occur for the two-fluid Euler equations.

Very closely related to the Hele-Shaw and Muskat models, the  Stefan problem (\ref{stefan}a-d) is a model of phase transition, and serves as yet another
example of a classical free-boundary problem.  One fundamental difference, however, with the Muskat problem is that there does not exist a contour dynamics description of the free-boundary evolution; on the other hand, it has been widely studied using a variety of parabolic PDE methods.
 For instance, the existence of classical solutions \emph{with derivative loss} was obtained by Meirmanov \cite{meirmanov1992stefan}, while the regularity of the free boundary was treated by Kinderlehrer \& Nirenberg in a series of papers \cite{kinderlehrer1977regularity, kinderlehrer1978smoothness},
 wherein  they showed  that if the free boundary is $C^1$ and the temperature $P$ satisfy certain conditions, the interface is analytic in space and of Gevrey class in time. More recently, Had\v{z}i\'{c} \& Shkoller \cite{hadzic5817well, HaSh2014} proved the local and global existence  \emph{without derivative loss},
 as well as the decay of solutions to equilibrium states. 

%

\subsection{Well-posedness for $H^s$ data with $s\leq 2.5$}

Mathematically, an $H^s$ well-posedness result,  with $s\leq 2.5$, for \eqref{contour} and \eqref{contour2} is challenging because the usual energy estimates indicate that $\|h\|_{C^{2+\delta}}$ is the quantity in the \emph{available} continuation criterion (see \cite{c-g07, CGO}).

As we have already noted,  most prior existence theorems have  relied upon the contour equations for the interface, which, in the case of  the infinitely-deep, unconfined 
Muskat problem is given as
\begin{equation}\label{contour}
\partial_t h=\text{p.v.}\int_\bbR\frac{(h'(x_1)-h'(x_1-y))y}{y^2+(h(x_1)-h(x_1-y))^2}dy,
\end{equation}
 and for  the finitely-deep medium, confined Muskat problem (with domain $\bbR\times[-l,l]$) as 
\begin{align}\label{contour2}
\partial_t h& =\text{p.v.}\int_\bbR\frac{(h'(x_1)-h'(x_1-y))\sinh(y)}{\cosh(y)-\cos(h(x_1)-h(x_1-y))}dy 
+\text{p.v.}\int_\bbR\frac{(h'(x_1)+h'(x_1-y))\sinh(y)}{\cosh(y)+\cos(h(x_1)+h(x_1-y))}dy \,.
\end{align} 
 These contour equations are obtained from the Birkhoff-Rott integral   together with the following expression for the vorticity:
$$
\omega(x_1, x_2,t)=\varpi(x_1,t)\delta_{\Gamma(t)},
$$
where $x_1 \in \mathbb{R}  $  parametrizes $\Gamma(t)$, 
$\varpi(x_1,t)$ is the amplitude of vorticity, and $\delta_{\Gamma(t)}$ is the Dirac  delta-distribution which is a function of $(x_1,x_2)$ on the moving
interface $\Gamma(t) \subset  \mathbb{R}^2  $.
 In particular, as the contour equations use the kernel for the operator $\nabla^\perp \Delta^{-1}$, there have been no prior existence theorems for arbitrary domain geometries.

In the case that $\jump{\mu}=0$, the contour equations  have a significant simplification with respect to the case of two different viscosities. 
This is due to the fact that, if $\jump{\mu}=0$, the amplitude of the vorticity is $\varpi=\jump{\rho} h'$; however, in the case with two different 
viscosities, the amplitude for the vorticity $\varpi$ verifies the integral equation
$$
-(\rho^2-\rho^1)h'(x_1)=\left(\mu^2-\mu^1\right)\text{p.v.}\int_\bbR \varpi(\beta) \mathcal{B} (x_1,h(x_1),\beta,h(\beta))d\beta\cdot (1,h'(x_1))+\left(\frac{\mu^2+\mu^1}{2}\right)\varpi,
$$
where $\mathcal{B} $ denotes the kernel of $\nabla^\perp\Delta^{-1}$ (which depends on the domain). For instance, if the union of the two fluid domains is
$ \mathbb{R}^2  $,  then
$$
\mathcal{B} (x_1,x_2,y_1,y_2)=\left(-\frac{x_2-y_2}{(x_2-y_2)^2+(x_1-y_1)^2}, \frac{x_1-y_1}{(x_2-y_2)^2+(x_1-y_1)^2}\right).
$$
Thus, to write the amplitude of the vorticity in terms of the interface, one needs to invert an operator as in C{\'o}rdoba,  C{\'o}rdoba, \& Gancedo \cite{c-c-g10}. This is a difficult issue, and with our method, we are able to avoid it entirely.

\section{Statement of the main theorems}

Our first result is
 
\begin{theorem}[$H^2$ local well-posedness for the two-phase problem]\label{localsmall}
Let $h_0\in H^2(\bbR)$ be the initial height function and  let $\mu^\pm,\,\rho^\pm>0,$ be fixed constants. 
Then for every arbitrarily small $s>0$ there exist small enough constants $\sigma_s$, $\tilde{\sigma}$, $T(h_0)>0$, such that if either
\begin{enumerate}
\item (for the infinitely-deep Muskat problem (\ref{HS_Eulerian}a-e)) if
\begin{equation}\label{cgs1}
\|h_0\|_{H^{1.5+s}(\bbR)}< \sigma_s 
\end{equation} 
or
\item (for the confined Muskat problem (\ref{HS_Eulerian}a-d,e',f)) if
$$
\|h_0\|_{H^{1.5+s}(\bbR)}< \sigma_s 
$$
$$
\max\{|\tilde{t}|_2,|\tilde{b}|_2\}\leq\tilde{\sigma},
$$ 
\end{enumerate}
then
there exists a unique local-in-time solution
$$
h\in C([0,T(h_0)];H^2(\mathbb{R}))\cap L^2(0,T(h_0);H^{2.5}(\mathbb{R})).
$$
Moreover, this solution verifies
$$
\|h(t)\|_{L^2(\bbR)}^2+\int_0^t\|\sqrt{\mu^+} u^+(\mathfrak{t})\|_{L^2(\Omega^+(\mathfrak{t}))}^2d \mathfrak{t}  
+\int_0^t\|\sqrt{\mu^-} u^-(\mathfrak{t}  )\|_{L^2(\Omega^-(\mathfrak{t}  ))}^2d\mathfrak{t}  =\|h_0\|_{L^2(\bbR)}^2,
$$
and
$$
\max_{0\leq s\leq T(h_0)}\{\|h(s)\|_{H^2(\bbR)}^2\}+\int_0^{T(h_0)}\|h(s)\|_{H^{2.5}(\bbR)}^2ds\leq C_1\|h_0\|_{H^2(\bbR)}^2,
$$
for a fixed constant $C_1$.
\end{theorem}
We remark that the constants appearing in this theorem depend on the physical parameters $\mu^\pm,\rho^\pm>0$.

The proof of this result in the infinitely-deep case has been split into several steps in  Section \ref{sec2}. For the sake of simplicity, the proof is given for the
 case that $s=0.25$ in (\ref{cgs1}), but the general case is obtained in a  straightforward manner.   This proof also covers the confined problem with flat
 {\it top} and {\it bottom} boundaries.  Observe that the solution gains an extra half-derivative in  space, when integrated in time.  
As we shall explain, this {\it parabolic-regularity}  property is obtained by using the jump condition related to the expression for the amplitude of the vorticity. 
In Section \ref{sec3}, we provide the proof for the case of general domain geometries.

Next, we address the question of  global existence  and decay to equilibrium of classical solutions for small data.   Indeed, if the initial data is periodic, Theorem \ref{localsmall} can be strengthened,  and we obtain
\begin{theorem}[$H^2$ global well-posedness and decay to equilibrium]\label{globalsmall}
Let $h_0\in H^2(\bbT)$ be the periodic, zero-mean initial height function for the infinitely-deep Muskat problem (\ref{HS_Eulerian}a-d) with 
$\mu^\pm,\rho^\pm>0$. Then there exists a small enough constant $\sigma_2=\sigma_2(\mu^\pm,\rho^\pm)$, such that if
 $\|h_0\|_{H^2(\bbT)}\leq \sigma_2$, there exists a unique global-in-time solution
$$
h\in C([0,\infty];H^2(\mathbb{T}))\cap L^2(0,\infty;H^{2.5}(\mathbb{T})).
$$
Moreover, this solution verifies
$$
\max_{0\leq s\leq \infty}\{\|h(s)\|_{H^2(\bbT)}^2\}+\int_0^\infty\|h(s)\|_{H^{2.5}(\bbT)}^2ds\leq C\|h_0\|_{H^2(\bbT)}^2,
$$
together with the decay estimate
$$
\|h(t)\|_{L^2(\bbT)}^2\leq c(h_0)e^{-\alpha t},\text{ and, more generally, } \|h(t)\|_{H^r(\bbT)}^2\leq c(h_0,r)e^{-\left(1-\frac{r}{2}\right)\alpha t} 
$$
for every $0\leq \alpha<2,$ $0\leq r<2$.
\end{theorem}
The proof of this result is given in Section \ref{sec4}. Notice that the decay of the linear problem ($\alpha=2$) is not reached and appears to be critical.

\begin{remark} 
We can compare the global existence result given by our Theorem \ref{globalsmall} with the global existence  results   in \cite{ccgs-10,ccgs-13} 
for the case that $\jump{\mu}=0$.   On the one hand, because of the embedding inequality
$$
  \| u\|_{A(\bbR)} \le C \| u\|_{H^{0.5+s}(\bbR) }  \,,\,s>0
$$
we see that we must impose more severe size constraints our initial data than the results of \cite{ccgs-10,ccgs-13}; on the other hand, our result can
also handle the case that $\mu^+ \neq \mu^-$, and we find the exponential decay rate back to the equilibrium configuration.   
\end{remark}

For the one-phase Muskat problem (the case where $\mu^+=\rho^+=0$), our previous result is improved:

\begin{theorem}[Local well-posedness for the one-phase problem]\label{localonephase}
Fix $\mu^+=\rho^+=0$ $\mu^-,\rho^->0$, $\tilde{b}(x_1)=0$. Let $h_0\in H^{2}(\bbT)$ such that $\min_{x_1}h_0(x_1)>c_b$, be the initial height function for the confined, one-phase Muskat problem (\ref{HS_Eulerian_Onephase}a-e) satisfying the Rayleigh-Taylor stability condition \eqref{RTonephase}.
Then there exists $T(h_0)$ and a unique local-in-time solution
$$
h\in C([0,T(h_0)];H^2(\mathbb{T}))\cap L^2(0,T(h_0);H^{2.5}(\mathbb{T}))
$$
for the confined Muskat problem (\ref{HS_Eulerian_Onephase}a-e). Moreover, this solution verifies
$$
\|h(t)\|_{L^2(\bbT)}^2+\int_0^t\|\sqrt{\mu^-} u(\mathfrak{t})\|_{L^2(\Omega^-(\mathfrak{t}))}^2d\mathfrak{t}=\|h_0\|_{L^2(\bbT)}^2,
$$
and
$$
\max_{0\leq s\leq T(h_0)}\{\|h(s)\|_{H^2(\bbT)}^2\}+\int_0^{T(h_0)}\|h(s)\|_{H^{2.5}(\bbT)}^2ds\leq C_1\|h_0\|_{H^2(\bbT)}^2,
$$
for a fixed constant $C_1$.
\end{theorem}

\begin{remark}
Note that in Theorem \ref{localonephase},  the initial data can be  arbitrarily large; in particular, we place no smallness condition on the data.
\end{remark} 
The proof of Theorem \ref{localonephase} is given in Section \ref{sec5}.
Finally, as a consequence of our half-derivative gain in space, $L^2$-in-time, we have the following
\begin{theorem}[Instantaneous parabolic smoothing]\label{Cinftyonephase}
Given $\Gamma$ and a solution $h$ to the Muskat problem satisfying
$$
h\in C([0,T(h_0)];H^2(\Gamma))\cap L^2(0,T(h_0);H^{2.5}(\Gamma))
$$
and either
\begin{enumerate}
\item $\Gamma=\bbR$ and $h$ is the solution to for the infinitely-deep Muskat problem (\ref{HS_Eulerian}a-e) obtained under the hypotheses of Theorem \ref{localsmall},
\item $\Gamma=\bbR$ and $h$ is the solution to for the confined Muskat problem (\ref{HS_Eulerian}a-d,e',f) obtained under the hypotheses of Theorem \ref{localsmall}, 
\item $\Gamma=\bbT$ and $h$ is the solution to for the one-phase Muskat problem (\ref{HS_Eulerian_Onephase}a-e) obtained under the hypotheses of Theorem \ref{localonephase},
\end{enumerate}
then, in fact, 
$$
h(\cdot ,t)\in C^\infty(\Gamma)\text{ if }\delta\leq t\leq T(h_0),\,\,\forall\,\delta>0.
$$
\end{theorem}
The proof of this result is given in Section \ref{sec6}.

\subsection{Notation}
\subsubsection{Matrix notation} Let $A$ be a matrix,  and $b$ be a column vector. Then, we write $A^i_j$ for the component of $A$, located on row $i$ 
and column $j$;  consequently, using the Einstein summation convention, we write 
$$
(Ab)^k=A^k_ib^i\text{ and }(A^Tb)^k=A^i_k b^i.
$$ 
\subsubsection{Sobolev norms}
For $s\ge 0$, we let
$$
\|u\|_{s,+} = \|u^+\|_{H^s(\Omega^+)}\,,\ \|u\|_{s,-} = \|u^-\|_{H^s(\Omega^-)}\,,\ \|u\|_{s,\pm} = \|u^+\|_{s,+} + \|u^-\|_{s,-}
$$
and
$$
|h|_s = \|h\|_{H^s(\Gamma)}\,.
$$
Let $\bbR^2_+$ and $\bbR^2_-$ denote the upper and lower half plane, respectively. Then, abusing notation, we write
$$
\|v\|_{s,+} = \|v^+\|_{H^s(\bbR^2_+)}\,,\ \|v\|_{s,-} = \|v^-\|_{H^s(\bbR^2_-)}\,,\ \|v\|_{s,\pm} = \|v^+\|_{s,+} + \|v^-\|_{s,-}
$$
and
$$
|h|_s = \|h\|_{H^s(\bbR)}\,.
$$
\subsubsection{The derivatives}
We let $f'$ denote the (tangential) derivative of $f$ with respect to $x_1$; that is,
$$
f' = \frac{\p f}{\p x_1}.
$$
For $k=1,2$, we write 
$$
f_{,k}=\frac{\partial f}{\partial x_k}.
$$
For  a diffeomorphism $\psi$, we let
$\text{curl}_\psi u=\text{curl} u\circ\psi$ and $\text{div}_\psi u=\text{div} u\circ\psi$.

\subsubsection{Mollifiers}
We consider $\mathcal{J}$ a symmetric, positive mollifier with total integral equal to 1. For $\kappa>0$, we define
\begin{equation*}
\mathcal{J}_\kappa(x_1)=\frac{1}{\kappa}\mathcal{J}\left(\frac{x_1}{\kappa}\right)
\end{equation*}
and we denote
$$
f^\kappa=\jeps f=\mathcal{J}_\kappa*f \ \text{ and } \ f^{\kappa\kappa}=\jeps\jeps f\,.
$$

\subsubsection{Dependence on space and time}
For a function $f(x,t)$, we shall often write $f(t)$ to denote $f(\cdot ,t)$.  We associate to the pair of functions $u^\pm:\Omega^\pm(t) \to \mathbb{R}  $, 
the function $u: \mathbb{R}^2 \to \mathbb{R}  $ as follows:
$$
u=u^+\textbf{1}_{\Omega^+(t)}+u^-\textbf{1}_{\Omega^-(t)}.
$$
When we write $\int_ {\Omega^+(t)} u ( \cdot ,t) dx$, this is understood to mean $\int_ {\Omega^+(t)} u^+ ( \cdot ,t) dx$.

\section{The ALE and semi-ALE formulations of the Muskat problem}
\subsection{The ALE and semi-ALE formulation}\label{localsmall1}
\subsubsection{The ALE formulation}
We let $\delta \psi^+$ denote the harmonic extension of $h$ to the upper half plane:
\begin{subequations}\label{deltapsi}
\begin{alignat}{2}
\Delta \delta \psi^{+} &= 0 &&\text{in}\quad \bbR^2_+ \,,\\
\delta \psi^{+} &= h \qquad&&\text{on}\quad \{x_2 = 0\}\,.
\end{alignat}
\end{subequations}
We define $\delta \psi^{-}(x_1,x_2)=\delta \psi^{+}(x_1,-x_2)$. We write $e$ for the identity map given by $e(x) = x$ and define $\psi^\pm=e + \delta\psi^\pm e_2.$ Then, $\psi^\pm( \cdot , t) :\bbR^2_\pm\mapsto \Omega^\pm(t)$ is a solution to
\begin{subequations}\label{psi_eq}
\begin{alignat}{2}
\Delta \psi^\pm &= 0 &&\text{in}\quad \bbR^2_\pm\,,\\
\psi^\pm &= e + h e_2 \qquad&&\text{on}\quad \{x_2 = 0\}\,,
\end{alignat}
\end{subequations}
We note that (\ref{psi_eq}b) is the same as
$$
\psi(x_1,0,t) = \big(x_1, h(x_1,t)\big)\,. \eqno{\rm(\ref{psi_eq}b')}
$$

Setting $J^\pm= \det(\nabla\psi^\pm)$, we see that
$$
A^\pm = (\nabla \psi^\pm)^{-1} = (J^\pm)^{-1} \left[\begin{array}{cc}
(\psi^\pm)^2,_2 & - (\psi^\pm)^1,_2 \\
-(\psi^\pm)^2,_1 & (\psi^\pm)^1,_1
\end{array}
\right]=\frac{1}{1+\delta\psi^\pm_{,2}} \left[\begin{array}{cc}
1+\delta\psi^\pm_{,2} & 0 \\
-\delta\psi^\pm_{,1} & 1
\end{array}
\right]\,.
$$
For a fixed $s>0$, using classical elliptic theory, we have $\|\nabla\delta\psi^\pm\|_{1+s,\pm}\leq C|h|_{1.5+s}$, and
\begin{equation}\label{diffeo}
J^\pm=1+\delta\psi^\pm_{,2}>1-\|\delta\psi^\pm_{,2}\|_{L^\infty(\bbR^2)}>1-C\|\nabla\delta\psi^\pm\|_{1+s,\pm}>1-C|h|_{1.5+s}.
\end{equation}
Consequently, if $|h( \cdot , t)|_{1.5+s}$ is sufficiently small, then $\psi(t)$ is a diffeomorphism.   For example, 
$|h( \cdot , t)|_{1.5+s}$ is small whenever
 the initial data $h_0 \in H^{1.5+s}(\bbR)$ and  $t$ are sufficiently smal.
 
Letting 
$$v^\pm = u^\pm \circ \psi \ \text{ and } \ q^\pm = p^\pm\circ\psi\,,$$
the chain-rule shows that  (\ref{HS_Eulerian}) can be written on the fixed domains as
\begin{subequations}\label{HS_ALE}
\begin{alignat}{2}
\mu^\pm v^\pm+ (A^\pm)^T  \nabla q^\pm &= - \rho^\pm \delta^i_2 \qquad&&\text{in}\quad \bbR^2_\pm\,,\\
(A^\pm)^j_i (v^\pm)^i,_j &= 0 &&\text{in}\quad \bbR^2_\pm\,,
\end{alignat}
\end{subequations}
where $\delta^j_i$ is the Kronecker delta. 

\subsubsection{The evolution equation for $h$}
We derive the evolution equation for $h$ to complete the system (\ref{HS_ALE}). 
We first note that
$$
J^\pm (A^\pm)^\rT e_2 = (-(\psi^\pm)^2,_1, (\psi^\pm)^1,_1) = (\psi^{\pm\prime})^\perp\,,
$$
where $f^\perp = (-f_2,f_1)$. 
Since $\psi^{\pm\prime}( \cdot ,t)$ is tangent to $\Gamma(t)$, we must have $\psi^{\pm\prime}( \cdot ,t)^\perp$ is a normal vector field to $\Gamma(t)$; moreover, 
by (\ref{psi_eq}b) we must have
\begin{equation}\label{JAtN_id}
J A^\rT e_2 = (-h',1)\quad\text{on}\quad\{x_2 = 0\}\,.
\end{equation}
The identity above also suggests that
\begin{equation}\label{defn:n}
n\circ \psi = \frac{(-h',1)}{\sqrt{1+h^{\prime 2}}} \quad\text{on}\quad\{x_2 = 0\}\,.
\end{equation}
On the other hand, differentiating (\ref{psi_eq}b') in $t$, we find that
\begin{equation}
\psi_t \cdot (n\circ \psi) = h_t \big(e_2 \cdot (n\circ \psi)\big)\qquad\text{on}\quad \{x_2 = 0\}\,. \label{h_eq_temp}
\end{equation}
By (\ref{HS_Eulerian}c) (or the interface moves along with the fluid velocity), $\psi_t \cdot (n\circ \psi) = (u\cdot n)\circ \psi$; thus (\ref{JAtN_id}, (\ref{defn:n}) and (\ref{h_eq_temp}) imply that
$$
v \cdot (n\circ \psi) = \frac{h_t}{\sqrt{1+h^{\prime 2}}}
$$
or equivalently,
$$
h_t = v \cdot (-h',1) = v \cdot (J A^\rT e_2) = (JAv) \cdot e_2\,\text{ on } \{x_2=0\}. \eqno{\rm(\ref{HS_ALE}c)}
$$
The coupled equations (\ref{psi_eq}a,b) and (\ref{HS_ALE}a,b,c), together with the initial condition
$$
h = h_0 \qquad\text{on}\quad\{t=0\} \eqno{\rm(\ref{HS_ALE}d)}
$$
is the ALE formulation of (\ref{HS_Eulerian}).

\subsection{The semi-ALE formulation}\label{sec2.2} For the purposes of reinstating a linear divergence-free constraint on the velocity field,
we let
\begin{equation}\label{w-vel}
w^\pm = J^\pm A^\pm v^\pm 
\end{equation} 
or componentwise, $w^\pm\cdot e_k = J^\pm (A^\pm)^k_i (v^\pm)\cdot e_i$.  Then, by the Piola identity, $(J^\pm(A^\pm)^i_{j})_{,i}=0,$  and 
(\ref{HS_ALE}b) implies that
$$
\div w^\pm = 0 \qquad\text{in}\quad\bbR^2_\pm\,.
$$
Therefore, $(w^\pm,q^\pm,h)$ satisfies
\begin{subequations}\label{HS_semiALE}
\begin{alignat}{2}
\mu w^\pm\cdot e_k + J^\pm (A^\pm)^k_i (A^\pm)^j_i q^\pm,_j &= - \rho^\pm J^\pm (A^\pm)^k_2 \qquad&&\text{in}\quad \bbR^2_\pm\,,\\
\div\, w^\pm &= 0 &&\text{in}\quad \bbR^2_\pm\,,\\
\jump{w\cdot e_2} = \jump{q} &= 0 &&\text{on}\quad \{x_2 = 0\}\,,\\
\Delta \psi^\pm &= 0 &&\text{in}\quad \bbR^2_\pm\,,\\
\psi &= e + h e_2 \qquad&&\text{on}\quad \{x_2 = 0\}\,,\\
h_t &= w \cdot e_2 \qquad &&\text{on}\quad \{x_2 = 0\}\,,\\
h &= h_0 &&\text{on}\quad \bbR \times \{t=0\}\,.
\end{alignat}
\end{subequations}
Equation (\ref{HS_semiALE}) is the semi-ALE formulation of (\ref{HS_Eulerian}). Since $A\, \nabla\psi=\text{Id}$ we have
$$
Ae_2=A\left(A^T\left(\nabla\psi^T\cdot e_2\right)\right),
$$

and (\ref{HS_semiALE}a) can also be written as
$$
\mu^\pm w^\pm\cdot e_k + J^\pm (A^\pm)^k_i (A^\pm)^j_i (q^\pm + \rho^\pm \psi^\pm\cdot e_2),_j = 0 \qquad\text{in}\quad \bbR^2_\pm\,. \eqno{\rm(\ref{HS_semiALE}a')}
$$
Let $Q^\pm = q^\pm + \rho^\pm x_2$. Since $A^\pm= [ \nabla \psi^\pm] ^{-1} $, it follows that
$$
\mu^\pm \frac{(\nabla\psi^\pm)^T\nabla\psi^\pm w^\pm}{J^\pm}+\nabla(Q^\pm+\rho^\pm\delta\psi^\pm)=0 \,.
$$

Using $Q$ rather than $q$, we write the system (\ref{HS_semiALE}) as
\begin{subequations}\label{HS_semiALE1}
\begin{alignat}{2}
\mu^\pm w^\pm + \nabla (Q^\pm + \rho^\pm \delta\psi^\pm) &= \left(\text{Id}-\frac{(\nabla\psi^\pm)^T\nabla\psi^\pm }{J^\pm}\right)\mu^\pm w^\pm  \qquad&&\text{in}\quad \bbR^2_\pm\,,\\
\div\, w^\pm &= 0 &&\text{in}\quad \bbR^2_\pm\,,\\
\jump{w\cdot e_2} = \jump{Q} &= 0 &&\text{on}\quad \{x_2 = 0\}\,,\\
\Delta \delta \psi^\pm &= 0 &&\text{in}\quad \bbR^2_\pm\,,\\
\delta\psi^\pm &= h \qquad&&\text{on}\quad \{x_2 = 0\}\,,\\
h_t &= w \cdot e_2 \qquad &&\text{on}\quad \{x_2 = 0\}\,,\\
h &= h_0 &&\text{on}\quad \bbR \times \{t=0\}\,.
\end{alignat}
\end{subequations}
The advantage of the formulation \eqref{HS_semiALE1} is that the nonlinear terms are on the right-hand side, keeping the left-hand side linear. Indeed, using 
$$
\nabla\psi^\pm=\nabla(x+\delta\psi^\pm e_2)=\text{Id}+\nabla\delta\psi^\pm e_2\,,
$$
we have that
\begin{equation}\label{eq1}
\left(\text{Id}-\frac{(\nabla\psi)^T\nabla\psi^\pm }{J^\pm}\right)\mu^\pm w^\pm=\left(\begin{array}{cc}\delta\psi^\pm_{,2}-(\delta\psi^\pm_{,1})^2 & -\delta\psi^\pm_{,1}(1+\delta\psi^\pm_{,2})\\
-\delta\psi^\pm_{,1}(1+\delta\psi^\pm_{,2}) & -\delta\psi^\pm_{,2}(1+\delta\psi^\pm_{,2})\end{array}\right)\frac{\mu^\pm w^\pm}{J^\pm}
\end{equation}

\section{The approximate $ \kappa $-problem}
\subsection{An approximation of the semi-ALE formulation: the $\kappa $-problem}
Letting $A^\pm_ \kappa = ( \nabla \psi^\pm_ \kappa ) ^{-1} $, 
we define the following approximation of (\ref{HS_semiALE1}) which we term the $ \kappa $-problem:
\begin{subequations}\label{HS_semiALE_reg}
\begin{alignat}{2}
\mu^\pm w^\pm+J^\pm_\kappa A^\pm_\kappa (A^\pm_\kappa)^T\nabla(Q^\pm+\rho^\pm\delta\psi^\pm_\kappa)&=0 && \text{in}\quad \bbR^2_\pm\times [0,T_\kappa ]\,,\\
\jump{w\cdot e_2}=\jump{Q}&= 0 && \text{on}\quad \Gamma\times [0,T_\kappa ]\,,\\
\div\, w^\pm &= 0 &&\text{in}\quad \bbR^2_\pm\times [0,T_\kappa ]\,,\\
\Delta \delta\psi^{+}_\kappa &= 0 &&\text{in}\quad \bbR^2_+\times [0,T_\kappa ]\,,\\
\delta\psi^{+}_\kappa &= \jeps\jeps h_\kappa  \qquad&&\text{on}\quad \Gamma\times [0,T_\kappa ]\,,\\
\delta\psi^{-}_\kappa(x_1,x_2) &= \delta\psi^{+}_\kappa(x_1,-x_2) \qquad&&\text{on}\quad \bbR^2_{-}\times [0,T_\kappa ]\,,\\
\psi^{\pm}_\kappa(x_1,x_2) &= (x_1,x_2+\delta\psi^{\pm}_\kappa(x_1,x_2)) \qquad&&\text{in}\quad \bbR^2_{-}\times [0,T_\kappa ]\,,\\
h_{\kappa t} &= w \cdot e_2 \qquad &&\text{on}\quad \Gamma\times(0,T_\kappa ]\,,\\
h_\kappa &= \jeps h_0 &&\text{on}\quad \bbR \times \{t=0\}\,.
\end{alignat}
\end{subequations}
This approximation relies on the following two operations:
\begin{enumerate}
\item the initial data $h_0$ is regularized in (\ref{HS_semiALE_reg}i), and 
\item in order to have smooth ALE maps $\psi^\pm$ via elliptic extension, we (symmetrically) mollify the height function on $\Gamma$ in 
(\ref{HS_semiALE_reg}e), thus producing a smooth evolving interface.
\end{enumerate}  
Note that $w$ and $Q$ depend implicitly on $\kappa$.

\subsection{The ALE formulation of the $\kappa$-problem}  The $ \kappa $-approximation becomes very clear when we return to the original
ALE formulation given in (\ref{HS_ALE}).  Indeed, we use $A_ \kappa ^\pm $ in place of $A^\pm$ and $\psi_ \kappa ^\pm$ in place of $\psi^\pm$ and
write (\ref{HS_semiALE_reg}a-b) equivalently as 
\begin{subequations}\label{HS_ALEk}
\begin{alignat}{2}
\mu^\pm \mathcal{V} ^\pm+ (A_ \kappa ^\pm)^T  \nabla \mathcal{Q} ^\pm &= - \rho^\pm \delta^i_2 \qquad&&\text{in}\quad \bbR^2_\pm\times [0,T_\kappa ]\,,\\
(A_\kappa ^\pm)^j_i (\mathcal{V} ^\pm)^i,_j &= 0 &&\text{in}\quad \bbR^2_\pm\times [0,T_\kappa ]\,,
\end{alignat}
\end{subequations}
where 
$$
\frac{\nabla\psi_\kappa^\pm}{J^\pm_\kappa} w ^\pm = \mathcal{V} ^\pm \text{ and } \mathcal{Q} ^\pm=Q^\pm-\rho^\pm x_2
$$

\subsection{The Eulerian formulation of the $\kappa$-problem}
Pulling back (\ref{HS_semiALE_reg}) using the diffeomorphisms $(\psi_ \kappa ^\pm)^{-1}$ defined in (\ref{HS_semiALE_reg}d-g), we obtain the Eulerian form of the
$ \kappa $-problem:
\begin{subequations}\label{HS_Eulerian_kappa}
\begin{alignat}{2}
\mu^\pm \mathcal{U} ^\pm + \nabla \mathcal{P} ^\pm &= - \rho^\pm e_2 \qquad&&\text{in}\quad\Omega_\kappa^\pm(t)\times [0,T_\kappa ]\,,\\
\operatorname{div}  \mathcal{U}  &= 0 &&\text{in}\quad\Omega_\kappa^\pm(t)\times [0,T_\kappa ]\,,\\
\V(\Gamma_\kappa(t)) &= \mathcal{U} ^\pm \cdot n_\kappa \qquad&&\text{on}\quad\Gamma_\kappa(t)\times [0,T_\kappa ]\,,\\
\overline{\Omega_\kappa^+(t)}\cup\Omega_\kappa^-(t) &= \bbR^2 &&\text{for every }\quad t\in [0,T_\kappa ]\,,
\end{alignat}
\end{subequations}
where
\begin{align*} 
\mathcal{U}^\pm & = \mathcal{V}^\pm  \circ (\psi_ \kappa ^\pm)^{-1} \,,\\
\mathcal{P}^\pm & = \mathcal{Q}^\pm  \circ (\psi_ \kappa ^\pm)^{-1}  \,,
\end{align*} 
and
\begin{align*} 
\Gamma _ \kappa (t)& =\{(x_1,h^{\kappa\kappa}(x_1,t)),\,x_1\in\bbR\} \,, \\
n_\kappa (x_1,t) & =(-h^{\kappa\kappa\prime}(x_1,t),1) \,, \\
\Omega^+_ \kappa (t)& =\{(x_1,x_2),x_2>h^{\kappa\kappa}(x_1,t)),\,x_1\in\bbR\} \,, \\
\Omega^-_ \kappa (t)&=\{(x_1,x_2),x_2<h^{\kappa\kappa}(x_1,t)),\,x_1\in\bbR\} \,.
\end{align*} 
 we obtain a solution to (\ref{HS_semiALE_reg}).   
 
 \subsection{An alternative semi-ALE formulation of the $\kappa$-problem}  In order to construct solutions to the $\kappa $-problem for initial height
 functions in $H^2(\Gamma)$ of arbitrary size, we use a different family of
 diffeomorphisms which have the property that the
 Jacobian determinant  is equal to one.   For this purpose, we introduce the diffeomorphisms
$$
{\Psi_ \kappa }^{\pm} = (x_1, x_2 +  {h}^{\kappa\kappa}) \,.\\
$$
Because of the mollifiers present in the definition of $ {h}^{\kappa\kappa}$, we see that the maps $\Psi_ \kappa^\pm ( \cdot ,t): \Gamma \to
 \Omega^\pm_\kappa (t)$ are $C^ \infty $ diffeomorphisms, and that $ \det \nabla \Psi_\kappa ^\pm =1$.

Letting 
\begin{align*} 
\mathscr{V}^\pm& =  \mathcal{U}  \circ \Psi_ \kappa ^\pm \,,\\
\mathscr{Q}^\pm& =  \mathcal{P}  \circ \Psi_ \kappa ^\pm + \rho^\pm x_2 \,,
\end{align*} 
and defining
\begin{align*} 
 \mathcal{A}_ \kappa ^\pm& = [ \nabla \Psi_\kappa ^\pm] ^{-1} \,, \\
\mathscr{W} ^\pm &=   \mathcal{A}_ \kappa  ^\pm \mathscr{V} ^\pm \,,
\end{align*} 
we have our alternative semi-ALE description of the $ \kappa $-problem:
\begin{subequations}\label{HS_semiALE_reg2}
\begin{alignat}{2}
\mu^\pm \mathscr{W} ^\pm+ \mathcal{A} ^\pm_\kappa (\mathcal{A} ^\pm_\kappa)^T\nabla(\mathscr{Q}^\pm+\rho^\pm h^{\kappa\kappa})&=0 && \text{in}\quad \bbR^2_\pm \times [0,T_\kappa ]\,,\\
\jump{\mathscr{W} \cdot e_2}=\jump{\mathscr{Q}}&= 0 && \text{on}\quad \Gamma \times [0,T_ \kappa ]\,,\\
\div\,  \mathscr{W} ^\pm&= 0 &&\text{in}\quad \bbR^2_\pm\times [0,T_ \kappa ]\,,\\
h_t &= \mathscr{W} \cdot e_2 \qquad &&\text{on}\quad \Gamma\times (0,T_ \kappa ] \,,\\
h &= \jeps h_0 &&\text{on}\quad\Gamma  \times \{t=0\}\,.
\end{alignat}
\end{subequations}
Note well that a solution to (\ref{HS_semiALE_reg2}) give a solution to (\ref{HS_Eulerian_kappa}) and hence a solution to the original semi-ALE
formulation (\ref{HS_semiALE_reg}).   

\subsection{The construction of solutions to the $ \kappa $-problem (\ref{HS_semiALE_reg})}
In this section we prove the following result:
\begin{proposition}\label{approxsol}
For $h_0\in H^2$, there exist a time $T_\kappa$ and a unique solution $h\in C([0,T_\kappa],H^2(\bbR))$ to the approximate $ \kappa $-problem (\ref{HS_semiALE_reg}a-i).
\end{proposition}
Given
$$\bar{h}\in C([0,T];H^2(\bbR)) \ \text{ and } \ \bar{h}_t\in L^2(0,T;L^2(\bbR))\,,$$
we consider the following  linear problem:
\begin{subequations}\label{HS_semiALE_reg_fix}
\begin{alignat}{2}
\mu^\pm w^\pm+\bar{J}^\pm \bar{A}^\pm (\bar{A}^\pm)^T\nabla(Q^\pm+\rho^\pm\bar h^{ \kappa \kappa })&=0 && \text{in}\quad \{\bbR^2_{\pm}\}\,,\\
\jump{w\cdot e_2}=\jump{Q}&= 0 && \text{in}\quad \{x_2 \ne 0\}\,,\\
\div\, w^\pm &= 0 &&\text{in}\quad \bbR^2_\pm\,,\\
\overline{\Psi}^{\pm } &= (x_1, x_2 +  \bar{h}^{\kappa\kappa}) \ \  &&\text{in}\quad \bbR^2_\pm\,,\\
h_{ t} &= w \cdot e_2 \qquad &&\text{on}\quad \{x_2 = 0\}\,,\\
h &= \jeps h_0 &&\text{on}\quad \bbR \times \{t=0\}\,.
\end{alignat}
\end{subequations}
To simplify notation, we have dropped the $\kappa$-subscript used to indicate implicit dependence on $ \kappa $, but we have kept the 
$ \kappa $-superscript to indicate 
an explicit mollification operation; in particular,  
$$\bar{h}^{\kappa}=\jeps \bar{h} \  \text{ and } \  \bar{h}^{\kappa\kappa}=\jeps\jeps \bar{h} \,.$$
Note that
 $$
\|\nabla\overline{\Psi}^\pm - \operatorname{Id} \|_{s-1,\pm}\leq C|\bar{h}^{\kappa\kappa}|_{s} \leq C(\kappa,s)|\bar{h}^\kappa|_0.
$$ 
We shall also (temporarily)  drop the  $(\cdot)^\pm$ notation on $A$, $\psi$, $\rho$, and $\mu$, as it will be clear from the context which phase we are 
analyzing.
\subsubsection{The existence of $\nabla Q^\pm$}  
Taking the divergence of (\ref{HS_semiALE_reg_fix}a) we obtain the elliptic equation for $Q^\pm$ 
\begin{equation}\label{eqQapprox}
-\text{div}\left(\frac{1}{\mu} \overline{A}\, \overline{A}^T\nabla Q^\pm \right)=\text{div}\left( \frac{\rho}{\mu}  \overline{A}\, \overline{A}^T  \nabla \bar{h}^{ \kappa \kappa } \right)   \text{ in } 
\mathbb{R}^2  _{\pm}\,,
\end{equation}
where $\nabla \bar{h}^{ \kappa \kappa }= (\bar{h}_{,1}^{\kappa \kappa } , 0)$.
Due to the fact that the domain is unbounded, we consider a constant $ \gamma$ satisfying
$0<\gamma<{\frac{1}{2}} $, and define the following  elliptic equation in $ \mathbb{R}^2_\pm$ for the modified pressure functions $Q_\gamma^\pm $:
\begin{equation}\label{cgsQ}
\gamma Q_\gamma^\pm -\text{div}\left(\frac{1}{\mu} \overline{A}\, \overline{A}^T\nabla Q^\pm \right)=\text{div}\left( \frac{\rho}{\mu}  \overline{A}\, \overline{A}^T  \nabla \bar{h}^{ \kappa \kappa } \right)
 \text{ in } 
\mathbb{R}^2  _{\pm}\,.
\end{equation} 
Using (\ref{HS_semiALE_reg_fix}a) and (\ref{HS_semiALE_reg_fix}b), we 
supplement (\ref{cgsQ}) with  the following  jump conditions across $\{x_2=0\}$:
\begin{equation}\label{jump0}
\jump{Q_\gamma }=0
\end{equation} 
and
\begin{equation}\label{jump1}
\bigjump{\left((1/\mu) \overline{A}\, \overline{A}^T\nabla Q_\gamma \right)\cdot e_2}=-\jump{(\rho/\mu))\left(\overline{A}\,\overline{A}^T\nabla \bar{h}^{ \kappa \kappa }\right)\cdot e_2}.
\end{equation}
Recall that a function $Q_\gamma = {Q_\gamma}^+ {\bf 1}_{\bbR^2_+} +  {Q_\gamma}^- {\bf 1}_{\bbR^2_-} \in H^1 ( \mathbb{R}^2  )$ is said to be a weak solution of
(\ref{cgsQ})--(\ref{jump1}) if
\begin{equation}\label{cgs_weak}
\gamma\int_{\bbR^2_+ \cup \bbR^2_-} Q_\gamma P\, dx+\int_{\bbR^2_+ \cup \bbR^2_-}  {\frac{1}{\mu}}  \overline{A}\, \overline{A}^T \, \nabla Q_\gamma \, \nabla P dx = 
\int_{\{x_2=0\}}g\,  Pdx_1 +\int_{\bbR^2_+ \cup \bbR^2_-} f \, P  dx
\end{equation} 
for all  $P\in H^1(\bbR^2)$, 
where $g = \jump{(\rho/\mu)) \overline{A}\,\overline{A}^T\nabla \bar{h}^{\kappa\kappa}}  \cdot e_2$ and $f =\text{div}\left((\rho/\mu)  \overline{A}\,\overline{A}^T\nabla \bar{h}^{\kappa\kappa}\right)$.

This problem can be written as
$$
B(Q_\gamma,P)=\mathcal{L}_1(P)+\mathcal{L}_2(P)\text{ for all }P\in H^1(\bbR^2)\,,
$$
where
\begin{align*} 
B(Q_\gamma,P) &=\gamma\int_{\bbR^2_+ \cup \bbR^2_-} Q_\gamma Pdx+\int_{\bbR^2_+ \cup \bbR^2_-}\nabla P\left((1/\mu) \overline{A}\, \overline{A}^T\nabla Q_\gamma\right) dx,
\\
\mathcal{L}_1(P)&=\int_{\{x_2=0\}}\jump{(\rho/\mu)\left( \overline{A}\, \overline{A}^T\nabla \bar{h}^{ \kappa \kappa } \right)\cdot e_2} Pdx_1, \\
\mathcal{L}_2(P)&=\int_{\bbR^2_+ \cup \bbR^2_-}P\text{div}\left((\rho/\mu)  \overline{A}\,\overline{A}^T\nabla \bar{h}^{ \kappa \kappa }\right)dx.
\end{align*} 
The existence of $Q_\gamma \in H^1( {\bbR^2})$ will  follow from the Lax-Milgram theorem, once we verify the necessary hypotheses.
From the fundamental theorem of calculus, we have that
$$
\|\overline{A}_0\overline{A}_0^T-\overline{A}( \cdot , t) \overline{A}^T(\cdot ,t)\|_{L^\infty}\leq  C_ \kappa \sqrt{t} \int_0^t|\bar h_t(s)|^2_0ds \le C_ \kappa \sqrt{t}  \,,
$$
where $C_ \kappa $ is a constant which depends on $ \kappa $. Since $[\overline{A}_0 \overline{A}_0^T]^i_j   \xi _i\xi _j \ge \lambda | \xi |^2$, we see that for $t$ sufficiently small, 
$$
 {\frac{\lambda }{2}} |\xi|^2\leq [\overline{A} (\cdot, t ) \overline{A}^T(\cdot , t)]^i_j\xi^i\xi^j\leq 2 \lambda |\xi|^2,
$$
The bilinear form is bounded, as
$$
|B(Q_\gamma,P)|\leq C(\bar{h}^{\kappa\kappa})\|Q_\gamma\|_{1,\pm}\|P\|_{1,\pm},
$$
and it is also coercive, since
$$
|B(Q_\gamma,Q_\gamma)|
\geq c(\gamma,\lambda) \|Q_\gamma\|^2_{1,\pm}.
$$
Thus, we need to prove that $\mathcal{L}_i(P)$ are continuous functionals on $H^1(\bbR^2)$. We have that
$$
|\mathcal{L}_1(P)|\leq C(\bar{h}^{\kappa\kappa})|P|_0\leq C(\bar{h}^{\kappa\kappa})\|P\|_{1,\pm},
$$
and using the divergence theorem,
\begin{align*} 
|\mathcal{L}_2(P)| 
&\leq\left|\int_{\bbR^2_+ \cup \bbR^2_-}\nabla P(\rho/\mu)  \overline{A}\,\overline{A}^T\nabla \overline{A} dx\right|+\left|\int_{\{x_2=0\}} P\jump{(\rho/\mu) \overline{J}\, \overline{A}\, \overline{A}^T\nabla \bar{h}^{\kappa\kappa}}e_2dx_1\right|\\
& \leq C(\bar{h}^{\kappa\kappa})\left(\|\nabla P\|_{0,\pm}+|P|_0\right)\leq C(\bar{h}^{\kappa\kappa})\|P\|_{1,\pm}.
\end{align*} 
We have thus verified the hypotheses of the Lax-Milgram theorem.
 
To obtain estimates which are uniform in $\gamma$,  we test (\ref{cgsQ}) with $Q_\gamma$, and integrate by parts. Since $Q_\gamma \in H^1( \mathbb{R}^2  )$, ${Q_\gamma}^+ = {Q_\gamma}^-$ on $\{x_2=0\}$ (in the sense of trace); hence, 
we have that
\begin{align*} 
{\frac{1}{2}} \|\nabla Q_\gamma\|_{0,\pm}^2 & \leq -\int_{\{x_2=0\}} Q_\gamma \jump{\left((\overline{J}/\mu) \overline{A} \,\overline{A}^T\nabla Q_\gamma\right)\cdot e_2} dx_1
+\int_{\bbR^2_+ \cup \bbR^2_-}\nabla Q_\gamma(\rho/\mu) \overline{A}\, \overline{A}^T\nabla \bar{h}^{\kappa\kappa} dx \\
& \qquad
-\int_{\{x_2=0\}} Q_\gamma\jump{\left((\rho/\mu)  \overline{A}\, \overline{A}^T\nabla  \bar{h}^{ \kappa \kappa } \right)\cdot e_2}dx_1.
\end{align*} 
In particular, using the jump condition \eqref{jump1}, we find that
\begin{equation}\label{cgs1001}
\|\nabla Q_\gamma\|_{0,\pm}\leq C|\bar{h}^{\kappa\kappa}|_{0.5},
\end{equation} 
where the constant in the right-hand side is independent of $\gamma$.  As such, 
we obtain the existence of a weak limit $ \nabla Q_  \gamma \rightharpoonup F\in L^2(\bbR^2)$;  moreover, the weak limit is a gradient:  $F=\nabla Q$. Indeed, if $U\subset \bbR^2$, by means of the Poincar\'e inequality, we have that
$$
\|Q_\gamma-\text{mean}(Q_\gamma)\|_{L^2(U)}\leq C(U)\|\nabla Q_\gamma\|_{0,\pm}.
$$
In particular,  we obtain that $Q_\gamma$ converges weakly in $L^2(U)$. We write $Q$ for this limit and note that $ \nabla Q$ also satisfies
(\ref{cgs1001}). 
 Thus, considering a test function $\phi$  with compact support within $U$,  as $ \gamma \to 0$, we have that
\begin{align*} 
\int_U \phi\nabla Q_\gamma dx=-\int_U \text{div}\phi Q_\gamma dx & \rightharpoonup -\int_U \text{div}\phi Q dx=\int_U \phi \nabla Q dx, \\
\int_U \phi\nabla Q_\gamma dx & \rightharpoonup \int_U \phi F dx.
\end{align*} 
Using the uniqueness of the weak limit, we conclude the claim.  We then easily obtain that $Q\in L^2_{loc}(\bbR^2)\cap \dot{H}^1(\bbR^2_\pm)$ is a distributional solution to 
$$
-\text{div}\left((1/\mu) \overline{A}\, \overline{A}^T\nabla Q\right)=\text{div}\left((\rho/\mu) \overline{A}\, \overline{A}^T\nabla \bar{h}^{ \kappa \kappa } \right).
$$
\subsubsection{The existence of $w$ and $h$}
We consider the Banach space
$$
X=\{(f,f_t),\,f\in C(0,T;H^2),\,f_t\in L^2(0,T;L^2)\},
$$
with norm
$$
\|(f,f_t)\|_X=\max_{0\leq s\leq t}\|f(s)\|_{H^2}+\left(\int_0^t\|f_t(s)\|_{L^2}^2ds\right)^{0.5}.
$$

We define the operator $S[\bar{h},\bar{h}_t]$ by
$$
S[\bar{h},\bar{h}_t] =\left(h(t),w_2(\cdot,0,t)\right)=\left(h(x_1,0)
+\int_0^t w(x_1,0,s)\cdot e_2ds,w(x_1,0,t)\cdot e_2\right).
$$

As $\bar{h}^{\kappa\kappa}$ is $C^\infty$, the same is true for $\bar{\Psi},\bar{A},\bar{J}$. The usual elliptic estimates for \eqref{eqQapprox} provide the regularity 
$$
\nabla Q\in L^\infty(0,T;C^\infty(\bbR^2_\pm)).
$$ 
Using (\ref{HS_semiALE_reg_fix}a), we have that
$$
w \in L^\infty(0,T;C^\infty(\bbR^2_\pm)).
$$
Consequently, the operator $S$ verifies
$$
S:X\rightarrow X.
$$

For two pairs ($\bar{h}_1,\bar{h}_{1t}$) and ($\bar{h}_2,\bar{h}_{2t}$), we estimate the
he Lipschitz norm:
\begin{eqnarray*}
\|S[\bar{h}_1,\bar{h}_{1t}]-S[\bar{h}_2,\bar{h}_{2t}]\|_X&\leq& T\max_{0\leq t\leq T}|(w_{1}(t)-w_{2}(t))\cdot e_2|_2 \\
&&+\left(\int_0^T |(w_1(s)-w_2(s))\cdot e_2|_0^2 ds\right)^{0.5}\\
&\leq& \sqrt{T}C_\kappa\|\left(\bar{h}_1,\bar{h}_{1t}\right)-\left(\bar{h}_2,\bar{h}_{2t}\right)\|_X
\end{eqnarray*}
Now, if $T=T_\kappa$ is chosen small enough, then the mapping $S$ is a contraction and then there exists a unique fixed-point, which is a local solution of our approximate $\kappa $-problem.

\section{Proof of Theorem \ref{localsmall}: Local well-posedness for the infinitely-deep case}\label{sec2}
\subsection{$\kappa$-independent estimates}\label{localsmall3} 
In this section, we prove that there is a time of existence $T^*$, independent of $\kappa$, and a priori estimates  on on $[0,T^*]$ also independent of 
$ \kappa $;  we  will thus be able to pass to the limit as $ \kappa \to 0$ and conclude the existence of a limiting function $h$.

To do so, we define the higher-order energy function (or norm) that will be shown to be bounded independent of $\kappa $:
\begin{equation}\label{energy}
E(t)=\max_{0\leq s\leq t}\{|h^{\kappa}(s)|_2^2\}+\int_0^t\|w(s)\|_{2,\pm}^2ds.
\end{equation}

Then, using Proposition \ref{approxsol}, there exists an approximate solution up to time $T_\kappa$ for every $\kappa>0$. We can take $T_\kappa$ as small as needed to ensure that
$$
\sup_{0\leq t\leq T_\kappa}E(t)\leq z^*,\;\forall \kappa>0,
$$
for a constant $z^*$ that will specified below.  A priori, these times $T_\kappa$ may tend to zero as $\kappa \to 0$. In the following sections, we are going to obtain uniform bounds for $E(t)$ up to a uniform time $T^*$, preventing the shrinking of the lifespan of the solution as $ \kappa \to 0$.
 
For the sake of clarity, we take $s=0.25$ in the statement of Theorem \ref{localsmall} (the proof for general $s$ is  analogous) and consider $\sigma\ll1$ a universal constant (that will be specified below). Furthermore, we take $T_\kappa$ small enough so we can ensure that
\begin{equation}\label{smallnessbootstrap}
\sup_{0\leq t\leq T_\kappa}|h^\kappa(t)|_{1.75}< \sigma.
\end{equation} 

\subsubsection{The estimates of $\delta\psi,J$} 
Using classical elliptic theory for the equations (\ref{HS_semiALE_reg}d-f), we get
\begin{equation}\label{ellipticdeltapsi}
\|\nabla \delta\psi\|_{s,\pm}\leq C|h^{\kappa\kappa}|_{s+0.5}, 
\end{equation}
thus,
$$
\|\nabla\psi- \operatorname{Id} \|_{s,\pm}\leq C|h^{\kappa\kappa}|_{s+0.5}, \text{ and }\|J- 1\|_{s,\pm}=\|\delta\psi_{,2}\|_{s,\pm}\leq C|h^{\kappa\kappa}|_{s+0.5}.
$$

\subsubsection{Estimates for $h\in L^\infty(0,T_ \kappa ;L^2(\bbR))$, $v^\pm\in L^2(0,T_ \kappa ;L^2(\bbR^2_\pm))$}\label{secmaxprinL2}
We let $a = J\, A$ denote the cofactor matrix of $ \nabla \psi$; using the fact that
$ \nabla \psi^2 = e_2$, we  write (\ref{HS_ALE}a) as
\begin{equation}\label{cgs1000}
\mu J v^i + a^k_i \left(q + \rho  \psi^2\right)_{,k} = 0 \text{ in } \mathbb{R}^2_\pm \times (0,T_ \kappa ] \,.
\end{equation} 
Since on $\Gamma:= \mathbb{R}  $,  $\psi^2 = h$ and $v \cdot \tilde n = h_t$, taking the $L^2( \mathbb{R}^2_\pm)$ inner-product of (\ref{cgs1000}) with 
$v^i$, the fact that $a^k_i,_k =0$  by the Piola identity and that $ a^k_i N_k = \tilde n$ to obtain the basic $L^2$ energy law:
$$
\frac{1}{2}\frac{d}{dt}\|\jeps h(t)\|^2_{L^2(\bbR)}+\frac{1}{-\jump{\rho}}\|\sqrt{\mu^\pm J} v^\pm\|_{L^2(\bbR^2_\pm)}^2=0.
$$
Integrating in time, we find that
$$
|h^{\kappa\kappa}(t)|_0\leq |\jeps h(t)|_{0}\leq |\jeps\jeps h_0|_{0}\leq |h_0|_{0},
$$
and
$$
\mu^\pm\int_0^t\|\sqrt{J} v^\pm\|_{L^2(\bbR^2_\pm)}^2 ds\leq -\jump{\rho}|h_0|_{0}.
$$
From (\ref{cgs1001}) and the smallness bound \eqref{smallnessbootstrap}, we see that
$$
\|\nabla Q\|_{0,\pm}\leq C \sigma ,\,\text{and}\,
\|w\|_{0,\pm}\leq C \sigma \,.
$$

\subsubsection{Verifying the smallness condition for $|h^{\kappa}|_{1.75},\|w\|_{1.5,\pm}$}\label{sec4.3}
Using (\ref{HS_semiALE_reg_fix}g) together with the Cauchy-Schwarz and trace inequalities, we have that
\begin{equation}\label{H1.5}
|h(t)-\jeps h_0|_{1.5}=\left|\int_0^t w\cdot e_2 ds\right|_{1.5} \leq \sqrt{t}\sqrt{C\int_0^t\|w(s)\|_{2,\pm}^2 ds} \leq \sqrt{tCE(t)} \,,
\end{equation}
and that
$$
|h^{\kappa }(t)-\jeps^2 h_0|_{1.5}=\left|\jeps\int_0^t w\cdot e_2 ds\right|_{1.5} \leq \sqrt{tCE(t)}.
$$
We can ensure that $|h^{\kappa}(t)|_2\leq \sqrt{z^*}$, so that
$$
|h^{\kappa}(t)-\jeps^2 h_0|^2_{1.75}\leq C|h^{\kappa}(t)-\jeps^2 h_0|_{1.5}|h^{\kappa}(t)-\jeps^2 h_0|_2\leq C\sqrt{t}z^*,
$$ 
and, by choosing
\begin{equation}\label{T1}
T_\kappa\leq T^*_1=\left(\frac{(\sigma-|h_0|_{1.75})^2}{4Cz^*}\right)^2,
\end{equation}
we have that
\begin{equation}\label{H1.75}
|h^{\kappa}(t)|_{1.75}\leq |h_0|_{1.75}+\sqrt{C\sqrt{t}z^*}<\sigma,\,\,\forall\,0\leq t\leq T^\kappa.
\end{equation}

Using (\ref{HS_semiALE_reg}a,c) and the fact that $\delta \psi$ is the harmonic extension of  $h$,  it follows that $Q$ satisfies
\begin{alignat*}{2}
\text{div}\left[(J/\mu) A A^T\nabla Q\right]&=\frac{\rho}{\mu}\text{div}\left[\left(\text{Id}-J A A^T\right)\nabla \delta\psi\right]\qquad&&\text{in}\quad \bbR^2_\pm\,,
\end{alignat*}
with jump conditions given by (\ref{HS_semiALE_reg}b) and  \eqref{jump1}.  It follows that we have the following elliptic equation for $Q$:
\begin{subequations}\label{Qe_eq}
\begin{alignat*}{2}
\mu^{-1}\Delta Q &= \div \big[ \mu^{-1}(\id - J A A^T) \nabla (Q + \rho \delta\psi) \big] &&\text{in}\quad \bbR^2_\pm, \\
\jump{Q} &= 0 &&\text{on}\quad\{x_2 = 0\},\\
\bigjump{\mu^{-1}\smallexp{$\displaystyle{} \frac{\p Q}{\p \rN}$}} &= \jump{\mu^{-1}(\id - J A A^T) (\nabla Q) e_2} - \jump{\mu^{-1}\rho} J A^2_i A^j_i \delta \psi_{,j} \quad\ &&\text{on}\quad\{x_2 = 0\}.
\end{alignat*}
\end{subequations}

From standard elliptic estimates,
\begin{eqnarray*}
\|\nabla Q\|_{1.25,\pm} &\leq& C \Big[\big\|(\id - J A A^T) \nabla (Q + \rho \delta\psi)\big\|_{1.25,\pm} \\
&& + \big|\jump{\mu^{-1}(\id - J A A^T) \nabla Q e_2}\big|_{0.75} + \big|\jump{\mu^{-1}\rho A^2_i A^j_i \delta \psi_{,j}}\big|_{0.75}\Big] \\
&\leq& C \Big[\big\|\id - J A A^T\big\|_{L^\infty(\bbR^2)} \big(\|\nabla Q\|_{1.25,\pm} + \rho \|\nabla\delta\psi\big\|_{1.25,\pm}\big) \\
&&+ \big\|\id - J A A^T\big\|_{1.25,\pm} \big(\|\nabla Q\|_{L^\infty(\bbR^2)} + \rho \|\nabla\delta\psi\|_{L^{\infty}(\bbR^2)}\big) \\
&&+ \big|\jump{\mu^{-1}\rho JA^2_i A^j_i \delta \psi_{,j}}\big|_{0.75}\Big],
\end{eqnarray*}
where the constant $C$ depends on $\mu^\pm$. Using the smallness condition \eqref{H1.75}, we have
$$
\big|J A^2_i A^j_i \delta \psi_{,j}\big|_{0.75} \leq C \left(|h^{\kappa\kappa}|_{1.75},\,
\big\|\id - J A A^T\big\|_{1.25,\pm}+|h^{\kappa\kappa}|_{1.75}\right)\leq C |h^{\kappa\kappa}|_{1.75}.
$$
As a consequence, we have
\begin{equation}\label{Q1.25}
\|\nabla Q\|_{1.25,\pm} \leq C|h^{\kappa\kappa}|_{1.75}.
\end{equation}

For the higher norm, we have
\begin{align*}
\|\nabla Q\|_{1.5,\pm} &\le C \Big[\big\|(\id - J A A^T) \nabla (Q + \rho \delta\psi)\big\|_{1.5,\pm} \\
&\qquad + \big|\jump{\mu^{-1}(\id - J A A^T) \nabla Q e_2}\big|_1 + \big|\jump{\mu^{-1}\rho A^2_i A^j_i \delta \psi_{,j}}\big|_1\Big] \\
&\le C \Big[\big\|\id - J A A^T\big\|_{L^\infty(\bbR^2)} \big(\|\nabla Q\|_{1.5,\pm} + \rho \|\nabla\delta\psi\big\|_{1.5,\pm}\big) \\
&\qquad + \big\|\id - J A A^T\big\|_{1.5,\pm} \big(\|\nabla Q\|_{L^\infty(\bbR^2)} + \rho \|\nabla\delta\psi\|_{L^{\infty}(\bbR^2)}\big) \\
&\qquad + \big|\jump{\mu^{-1}\rho J A^2_i A^j_i \delta \psi_{,j}}\big|_1\Big]\\
&\le C \Big[\big\|\id - J A A^T\big\|_{L^\infty(\bbR^2)} \big(\|\nabla Q\|_{1.5,\pm} + \rho \|\nabla\delta\psi\big\|_{1.5,\pm}\big) \\
&\qquad + \big\|\id - J A A^T\big\|_{1.5,\pm} \big(\|\nabla Q\|_{1.25,\pm} + |h^{\kappa\kappa}|_{1.75} \big) \\
&\qquad + \big|\jump{\mu^{-1}\rho J A^2_i A^j_i \delta \psi_{,j}}\big|_1\Big]\,.
\end{align*}

Using \eqref{JAtN_id} and \eqref{jump1},
$$
\big|J A^2_i A^j_i \delta \psi_{,j}\big|_1 \leq C |h^{\kappa\kappa}|_{1.75}(1+|h^{\kappa\kappa}|_{1.75}) |h^{\kappa\kappa}|_2,
$$
and
$$
\big\|\id - J A A^T\big\|_{1.5,\pm}\leq C |h^{\kappa\kappa}|_{2}(1+|h^{\kappa\kappa}|_{1.75})\leq C |h^{\kappa\kappa}|_{2}.
$$
Using $1+\sigma<2$,
$$
\|\nabla Q\|_{1.5,\pm} \leq (C|h^{\kappa\kappa}|_{1.75} \big(1 + |h^{\kappa\kappa}|_{1.75}\big))\|\nabla Q\|_{1.5,\pm}
+C|h^{\kappa\kappa}|_{2}|h^{\kappa\kappa}|_{1.75}+|h^{\kappa\kappa}|_{2}|h^{\kappa\kappa}|_{1.75},
$$
and, using the smallness condition \eqref{H1.75},
\begin{equation}\label{Q1.5}
\|\nabla Q\|_{1.5,\pm} \leq C|h^{\kappa\kappa}|_{1.75}|h^{\kappa\kappa}|_{2}.
\end{equation}

Using (\ref{HS_semiALE1}a) and \eqref{eq1}, we obtain
\begin{align} 
\|w\|_{1.5,\pm} & \leq C|h^{\kappa\kappa}|_{1.75}|h^{\kappa\kappa}|_{2} \,,\\
\|w\|_{1.25,\pm} & \leq C|h^{\kappa\kappa}|_{1.75} \,.
\label{w1.25}
\end{align}

\subsubsection{The Rayleigh-Taylor stability condition revisited}\label{sec2.5.4} Once we have the smallness condition
$$
\sup_{0\leq t\leq T^\kappa}\|w(t)\|_{1.25,\pm} \leq C|h^{\kappa\kappa}(t)|_{1.75}\leq \sigma,
$$
we find that 
\begin{equation}\label{smallv}
\sup_{0\leq t\leq T^\kappa}\|v(t)\|_{1.25,\pm}\leq C\sup_{0\leq t\leq T^\kappa}\left\|\nabla \psi\cdot w\right\|_{1.25,\pm}\leq C\sigma.
\end{equation}
The Rayleigh-Taylor stability condition is controlled as follows:
$$
RT(t)>-\frac{\jump{\rho}}{2} -|\jump{\mu}|2\|v\|_{1.25,\pm}\geq -\frac{\jump{\rho}}{2} -|\jump{\mu}|C\sigma.
$$
Consequently, if we impose
$$
\sigma\leq \frac{-\jump{\rho}}{4C|\jump{\mu}|},
$$
the Rayleigh-Taylor stability condition is satisfied for every time $0\leq t\leq T^\kappa$. Furthermore, we have
\begin{equation}\label{RTalltime}
-\jump{\rho}-\jump{\mu}v\cdot n\sqrt{1+ h^{\kappa\kappa\prime 2}}\geq-\jump{\rho}-|\jump{\mu}|2\|v\|_{1.25,\pm}\geq \frac{-\jump{\rho}}{2},\,\,\forall\,0\leq t\leq T^\kappa.
\end{equation}

\subsubsection{Estimates for $h\in L^2(0,T;H^{2.5}(\bbR))$}\label{sectionH2.5}
Taking the inner-product of the equation (\ref{HS_ALE}a) with the tangent vector
$\psi'$, we find that
$$
\mu^\pm v^\pm \cdot \psi' + Q^\pm_{,1} + \rho^\pm \delta \psi,_1 = 0 \qquad\text{on}\quad\{x_2=0\}\,.
$$
Taking the difference of the equations above, by (\ref{HS_semiALE_reg}e), we obtain that
$$
\jump{\mu v \cdot \psi'} + \jump{\rho} h^{\kappa\kappa}_{,1} = 0\,.
$$
Then the equation above implies that
$$
\Bigjump{\mu v \cdot \frac{(1,h^{\kappa\kappa\prime})}{\sqrt{1+ h^{\kappa\kappa\prime 2}}}} + \jump{\rho} \frac{h^{\kappa\kappa\prime}}{\sqrt{1+ h^{\kappa\kappa\prime 2}}} = 0\,.
$$
Differentiating the equation above with respect to $x_1$ and using that the normal velocity is continuous, we conclude that $h^{\kappa \kappa}$ satisfies that
\begin{equation}\label{h_reg_eq}
-\jump{\rho}h^{\kappa\kappa\prime\prime} = \jump{\mu}(v\cdot n)\sqrt{1+h^{\kappa\kappa\prime 2}}h^{\kappa\kappa\prime\prime}+(1+h^{\kappa\kappa\prime 2}) \jump{\mu v' \cdot (1,h^{\kappa\kappa\prime})}\,.
\end{equation}

By Proposition \ref{H0.5_fg} and the trace theorem, the inequality above further implies that
\begin{eqnarray}\label{h2.50}
|h^{\kappa\kappa\prime\prime}|_{0.5} &\leq& C (1+\left|h^{\kappa\kappa}\right|^3_{1.75}) \left(\left|v^+\right|_{1.5}+\left|v^-\right|_{1.5}\right)+C |h^{\kappa\kappa\prime\prime}|_{0.5} |v\cdot (-h^{\kappa\kappa\prime},1)|_{0.75} \,.
\end{eqnarray}
Since $ v = \frac{\nabla \psi w}{J}$,
\begin{align}\label{h2.51}
|v^\pm|_{1.5} &\leq C\left\|\frac{(\text{Id}+\nabla(\delta\psi^\pm e_2)) w^\pm}{J}\right\|_{2,\pm}\nonumber\\
&\leq C\left\|\frac{w^\pm}{J}\right\|_{2,\pm}+C\left\|\frac{\nabla (\delta\psi^\pm e_2) w^\pm}{J}\right\|_{2,\pm}
\end{align}
with
\begin{eqnarray}\label{h2.52}
\left\|\frac{w^\pm}{J}\right\|_{2,\pm}&\leq& C\left\|w^\pm\right\|_{2,\pm}+C\|w^\pm\|_{L^\infty(\bbR^2)}\|\nabla\delta\psi\|_{2,\pm}+C\|\nabla w\|_{L^4}\|\nabla \delta\psi_{,2}\|_{L^4}\nonumber\\
&\leq& C\left\|w^\pm\right\|_{2,\pm}+C\|w^\pm\|_{1.25,\pm}\|\nabla\delta\psi\|_{2,\pm}+C\|w\|_{1.5,\pm}\|\nabla\delta\psi\|_{1.5,\pm},
\end{eqnarray}
\begin{eqnarray}\label{h2.53}
\left\|\frac{\nabla (\delta\psi^\pm e_2) w^\pm}{J}\right\|_{2,\pm}&\leq&  C\left\|w^\pm\right\|_{L^\infty(\bbR^2)}\left[\left\|\nabla\delta\psi\right\|_{2,\pm}\left(1+\left\|\nabla\delta \psi\right\|_{L^\infty(\bbR^2)}\right)\right.\nonumber\\
&&\left.+\left\|D^2\delta\psi^\pm\right\|_{L^4}^2\right]
+C\left\|\nabla \delta\psi J^{-1}\right\|_{L^\infty(\bbR^2)} \left\|w\right\|_{2,\pm}\nonumber\\
&\leq& C\left\|w\right\|_{1.25,\pm}\left[\left\|\nabla\delta\psi\right\|_{2,\pm}\left(1+\left\|\nabla\delta \psi\right\|_{1.25,\pm}\right)\right.\nonumber\\
&&\left.+\left\|\nabla\delta\psi\right\|_{1.5,\pm}^2\right]+C\left\|\nabla \delta\psi\right\|_{1.25,\pm} \left\|w\right\|_{2,\pm}.
\end{eqnarray}
Collecting the estimates \eqref{h2.50}-\eqref{h2.53}, we get
\begin{align*}
|h^{\kappa\kappa}|_{2.5} & \leq C\|w\|_{2,\pm}+C (1+\left|h^{\kappa\kappa}\right|_{1.75})^4\left(\left|h^{\kappa \kappa}\right|_{2.5} \left\|w\right\|_{1.25,\pm}\right.\\
& \qquad  +\left|h^{\kappa\kappa}\right|_{1.75} \left\|w\right\|_{2,\pm}
\left.+\left|h^{\kappa \kappa}\right|_{2}^2+\left|h^{\kappa \kappa}\right|_{2}\|w\|_{1.5,\pm}\right)+C\|v\|_{1.25,\pm}|h^{\kappa\kappa}|_{2.5}.
\end{align*}
Using \eqref{w1.25}, \eqref{smallv} and the smallness condition \eqref{H1.75}, we have that
\begin{eqnarray*}
|h^{\kappa\kappa}|_{2.5} &\leq& C\|w\|_{2,\pm}+ C \left(\left|h^{\kappa\kappa}\right|_{1.75} \left\|w\right\|_{2,\pm}+\left|h^{\kappa \kappa}\right|_{2}^2+\left|h^{\kappa \kappa}\right|_{2}\|w\|_{1.5,\pm}\right).
\end{eqnarray*}
Consequently,
\begin{eqnarray}\label{h2.5}
\int_0^t|h^{\kappa\kappa}|^2_{2.5} &\leq& CE(t)+t(E(t))^2+ C\sigma^2 E(t).
\end{eqnarray}

\subsubsection{The energy estimates}\label{sec4.5}
Writing (\ref{HS_semiALE_reg}a) as
$$
\mu w + \nabla (Q + \rho \delta\psi) = \left(\text{Id}-\frac{(\nabla\psi)^T\nabla\psi }{J}\right)\mu w,
$$
differentiating with respect to $x_1$ twice, testing the resulting equation against $w''$, using integration-by-parts on the gradient term, and  using (\ref{HS_semiALE1}b), we find that
$$
\|\sqrt{\mu}w''\|_{0,\pm}^2 -\jump{\rho}\int_\bbR \jump{(Q + \rho \delta\psi)''w''\cdot N}dx_1 =  \int_{\bbR^2}\left(\left(\text{Id}-\frac{(\nabla\psi)^T\nabla\psi}{J}\right)\mu w\right)''w''dx.
$$
Using (\ref{HS_semiALE_reg}h), we see that 
$$
-\jump{\rho}\int_\bbR \jump{(Q + \rho \delta\psi)''w''\cdot N}dx_1 = \frac{\jump{\rho}}{2}\frac{d}{dt}|h^{\kappa\prime\prime}|_0^2 \,,
$$
and defining
$$RHS=\int_{\bbR^2}\left(\left(\text{Id}-\frac{(\nabla\psi)^T\nabla\psi }{J}\right)\mu w\right)''w''dx\,,$$
we have that
$$
\|\sqrt{\mu}w''\|_{0,\pm}^2 + \frac{\jump{\rho}}{2}\frac{d}{dt}|h^{\kappa\prime\prime}|_0^2 = RHS \,,
$$
and we proceed to estimate $RHS$.

Using the H\"{o}lder inequality together with \eqref{eq1}, we get that 
\begin{align*}
RHS & =\int_{\bbR^2}\left(\left(\text{Id}-\frac{(\nabla\psi)^T\nabla\psi }{J}\right)\mu w\right)''w''dx\\
& \leq C\|w''\|_{0,\pm}^2\left(\|\nabla\delta\psi\|_{L^\infty(\bbR^2)}^2+\|\nabla\delta\psi\|_{L^\infty(\bbR^2)}\right)\\
& \qquad +C\|w''\|_{0,\pm}\left[\|\nabla\delta\psi''\|_{0,\pm}\|w\|_{L^\infty(\bbR^2)}\right.\\
& \qquad \left.+\|w'\|_{L^4(\bbR^2)}\|\nabla\delta\psi'\|_{L^4(\bbR^2)}\right]\left(\|\nabla\delta\psi\|_{L^\infty(\bbR^2)}+1\right).
\end{align*}
Using the Sobolev inequality, we obtain that
\begin{align*} 
RHS& \leq C\|w''\|_{0,\pm}^2\left(\|\nabla\delta\psi\|_{1.25,\pm}^2+\|\nabla\delta\psi\|_{1.25,\pm}\right) \\
& \qquad 
+C\|w''\|_{0,\pm}\left[\|\nabla\delta\psi\|_{2,\pm}\|w\|_{L^\infty(\bbR^2)}+\|w\|_{1.5,\pm}^2
+\|\nabla\delta\psi\|_{1.5,\pm}^2\right]
\left(\|\nabla\delta\psi\|_{1.25,\pm}+1\right).
\end{align*}
Using the elliptic estimate \eqref{ellipticdeltapsi}, we find that
\begin{align*}
RHS & \leq C\|w\|_{2,\pm}^2\left(|h^{\kappa\kappa}|_{1.75}^2+|h^{\kappa\kappa}|_{1.75}\right)\\
& \qquad +C\|w\|_{2,\pm}\left[|h^{\kappa\kappa}|_{2.5}\|w\|_{1.25,\pm}+\|w\|_{1,\pm}\|w\|_{2,\pm}\right.\\
& \qquad \left.+|h^{\kappa\kappa}|_{1.5}|h^{\kappa\kappa}|_{2.5}\right]\left(|h^{\kappa\kappa}|_{1.75}+1\right).
\end{align*}
Recalling \eqref{w1.25}, we get that
$$
RHS\leq C\|w\|_{2,\pm}^2|h^{\kappa\kappa}|_{1.75}
+C\|w\|_{2,\pm}\left[|h^{\kappa\kappa}|_{1.75}\left(|h^{\kappa\kappa}|_{2.5}+\|w\|_{2,\pm}\right)+|h^{\kappa\kappa}|_{1.5}|h^{\kappa\kappa}|_{2.5}\right]
$$
Integrating in time and using \eqref{energy}, \eqref{H1.5}, \eqref{H1.75} and \eqref{h2.5}, we obtain that
$$
\int_0^tRHS\leq C\sigma\left(E(t)+t(E(t))^2\right)
$$
thus, we conclude that
\begin{equation}\label{w''}
\frac{-\jump{\rho}}{2}|h^{\kappa}(t)|_2^2+\min\{\mu^+,\mu^-\}\int_0^t\|w''\|_{0,\pm}^2\leq \frac{-\jump{\rho}}{2}|h_0|_2^2+C\sigma E(t)+tC(E(t))^2.
\end{equation}

\subsubsection{The Hodge decomposition elliptic estimates}\label{sec4.6}
Using Proposition \ref{normaltrace}, we have that
$$
|w^2|_{1.5}\leq |w''\cdot e_2|_{-0.5}\leq C\left(\|w''\|_{0,\pm}+\|\div w''\|_{0,\pm}\right)\leq C\|w''\|_{0,\pm}.
$$
Consequently, we can bound $\int_0^t |w^2|^2_{1.5}ds$ using \eqref{w''}. Using that $u$ is irrotational in each phase, we obtain $u^2_{,1}-u^1_{,2}=A^j_1v^2_{,j}-A^j_2v^1_{,j}=0$. Recalling 
$$
v=J^{-1}\nabla\psi\cdot w,\text {i.e. } v^j=J^{-1}\psi^{j}_{,i}w^i,
$$
and we get
\begin{eqnarray*}
w^2_{,1}-w^1_{,2}&=&w^2_{,1}-w^1_{,2}-A^j_1(J^{-1}\psi^{2}_{,i}w^i)_{,j}+A^j_2(J^{-1}\psi^{1}_{,i}w^i)_{,j}\\
&=& w^2_{,1}(1-A^1_1J^{-1}\psi^{2}_{,2})+w^1_{,2}(1-A^2_2J^{-1}\psi^{1}_{,1})\\
&&+\sum_{(i,j)\neq (1,2)}A^j_2J^{-1}\psi^{1}_{,i}w^i_{,j}-\sum_{(i,j)\neq (2,1)}A^j_1J^{-1}\psi^{2}_{,i}w^i_{,j}\\
&&-A^j_1(J^{-1}\psi^{2}_{,i})_{,j}w^i+A^j_2(J^{-1}\psi^{1}_{,i})_{,j}w^i.
\end{eqnarray*}
Using $1-A^1_1J^{-1}\psi^{2}_{,2}=0,$ $1-A^2_2J^{-1}\psi^{1}_{,1}=\delta\psi_{,2}(2+\delta\psi_{,2})/(1+\delta\psi_{,2})^2,$ $A^1_2=-\psi^{1}_{,2}=0$ we further simplify
\begin{eqnarray*}
w^2_{,1}-w^1_{,2}&=&w^1_{,2}\frac{\delta\psi_{,2}(2+\delta\psi_{,2})}{(1+\delta\psi_{,2})^2}-\frac{\delta\psi_{,1}}{1+\delta\psi_{,2}}w^1_{,1}+\frac{\delta\psi_{,1}}{1+\delta\psi_{,2}}w^2_{,2}-\left(\frac{\delta\psi_{,1}}{1+\delta\psi_{,2}}\right)^2w^1_{,2}\\
&&-A^j_1J^{-1}_{,j}\psi^{2}_{,i}w^i-A^j_1J^{-1}\delta\psi_{,ij}w^i-\frac{\delta\psi_{,22}}{(1+\delta\psi_{,2})^3}w^1\\
&=&w^1_{,2}\frac{\delta\psi_{,2}(2+\delta\psi_{,2})}{(1+\delta\psi_{,2})^2}-\frac{2\delta\psi_{,1}}{1+\delta\psi_{,2}}w^1_{,1}-\left(\frac{\delta\psi_{,1}}{1+\delta\psi_{,2}}\right)^2w^1_{,2}\\
&&+2\frac{\delta\psi_{,1}\delta\psi_{,12}w^1}{(1+\delta\psi_{,2})^2}-\frac{\delta\psi_{,11}w^1}{1+\delta\psi_{,2}}-\frac{\delta\psi_{,22}(1+(\delta\psi_{,1})^2)}{(1+\delta\psi_{,2})^3}w^1.
\end{eqnarray*}
Due to Proposition \ref{H0.5_fg}, we find that
\begin{eqnarray*}
\|J^3 \curl w\|_{1,\pm}&\leq& C\|w\|_{2,\pm}\|\nabla\delta\psi\|_{1.25,\pm}(1+\|\nabla\delta\psi\|_{1.25,\pm})^2\\
&&+C\|w\|_{2,\pm}\|\nabla\delta\psi\|_{1.25,\pm}^2(1+\|\nabla\delta\psi\|_{1.25,\pm})\\
&&+C\|w\|_{1.25,\pm}\|\nabla\delta\psi\|_{2,\pm}\|\nabla\delta\psi\|_{1.25,\pm}(1+\|\nabla\delta\psi\|_{1.25,\pm})\\
&&+C\|w\|_{1.25,\pm}\|\nabla\delta\psi\|_{2,\pm}(1+\|\nabla\delta\psi\|_{1.25,\pm})^2\\
&\leq&C\|w\|_{2,\pm}\|\nabla\delta\psi\|_{1.25,\pm}(1+\|\nabla\delta\psi\|_{1.25,\pm})^2\\
&&+C\|w\|_{1.25,\pm}\|\nabla\delta\psi\|_{2,\pm}(1+\|\nabla\delta\psi\|_{1.25,\pm})^2.
\end{eqnarray*}
From the smallness condition \eqref{H1.75}, we have that
$$
{\frac{1}{2}}\|\curl w\|_{0,\pm}\leq \|J^3\curl w\|_{0,\pm},$$
$$
{\frac{1}{2}}\|\nabla \curl w\|_{0,\pm}\leq \|J^3 \nabla \curl w\|_{0,\pm}\leq\|\nabla(J^3\curl w)\|_{0,\pm}+\|\nabla J^3 \curl w\|_{0,\pm}.\,,
$$
and from the  Sobolev embedding theorem,
\begin{eqnarray*}
\|\nabla J^3 \curl w\|_{0,\pm}^2&\leq& C(1+\|\nabla\delta\psi\|_{1.25,\pm})^4\int_\bbR (\nabla\delta\psi_{,2})^2(\curl w)^2dx\\
&\leq& C(1+\|\nabla\delta\psi\|_{1.25,\pm})^4\|\nabla\nabla\delta\psi\|_{L^4}^2\|\curl w\|_{L^4}^2\\
&\leq& C(1+\|\nabla\delta\psi\|_{1.25,\pm})^4\|\nabla\delta\psi\|_{1.5,\pm}^2\|w\|_{1.5,\pm}^2.
\end{eqnarray*}
We conclude that
\begin{eqnarray*}
\|\curl w\|_{1,\pm}&\leq& C\|w\|_{2,\pm}\|\nabla\delta\psi\|_{1.25,\pm}(1+\|\nabla\delta\psi\|_{1.25,\pm})^2\\
&&+C\|w\|_{1.25,\pm}\|\nabla\delta\psi\|_{2,\pm}(1+\|\nabla\delta\psi\|_{1.25,\pm})^2\\
&&+C(1+\|\nabla\delta\psi\|_{1.25,\pm})^4\|\nabla\delta\psi\|_{1.5,\pm}^2\|w\|_{1.5,\pm}^2\\
&\leq& C\|w\|_{2,\pm}|h^{\kappa\kappa}|_{1.75}+C|h^{\kappa\kappa}|_{2.5}|h^{\kappa\kappa}|_{1.75}+C|h^{\kappa\kappa}|_{2}\|w\|_{1.5,\pm}^2.
\end{eqnarray*}

Using Proposition \ref{Hodge}, we get
$$
\|w\|_{2,\pm} \leq C \Big[\|w\|_{0,\pm} + \|\curl w\|_{1,\pm} + \|\div w\|_{1,\pm} + |w\cdot e_2|_{1.5}\Big],
$$
and, using \eqref{w''}, we get
\begin{eqnarray}\label{winL2H2}
\int_0^t\|w\|_{2,\pm}^2&\leq & C\left( \frac{-\jump{\rho}}{2}|h_0|_2^2+C\sigma E(t)+tC(1+E(t)+(E(t))^2)E(t)\right).
\end{eqnarray}
\subsubsection{A polynomial-type inequality for the energy function $E(t)$} Notice that
$$
|h^{\kappa\kappa}(t)|_{2}\leq |h^{\kappa}(t)|_{2}.
$$
Furthermore, as $h^{\kappa\kappa}\in L^2(0,T_\kappa;H^{2.5}(\bbR))$ and 
$$
|h^{\kappa\kappa}_t|_{1.5}\leq |h_{\kappa t}|_{1.5}= |w^2|_{1.5}\leq C\|w\|_{2,\pm},
$$ 
we have $h^{\kappa\kappa}_t\in L^2(0,T_\kappa;H^{1.5}(\bbR))$. Consequently $h^{\kappa\kappa}\in C(0,T_\kappa;H^2(\bbR))$ and $E(t)$ is a continuous function. Collecting the previous estimates \eqref{w''} and \eqref{winL2H2} yields
\begin{eqnarray}\label{polinomine}
E(t)&\leq & \mathcal{C}\left( \frac{-\jump{\rho}}{2}|h_0|_2^2+\sigma E(t)+t(1+E(t)+(E(t))^2)E(t)\right).
\end{eqnarray}

\subsubsection{The uniform-in-$\kappa$ time} Recall that we assume that $T_\kappa$ is small enough to guarantee that $E(t)\leq z^*$ for $z^*>0$ a constant (depending on the size of the initial data) that will be chosen below. We set
$$
\sigma = \frac{1}{2\mathcal{C}},
$$
where $\mathcal{C}$ is the constant appearing in \eqref{polinomine}. We note that $\mathcal{C}$ is a constant depending only on the constants from the
Sobolev embedding theorem  and the elliptic estimate (\ref{Hodge}).  We can simplify \eqref{polinomine} to find that
$$
E(t)\leq  2C|h_0|_2^2+t\mathcal{P}(E(t)).
$$
This inequality implies that there exists a uniform-in-$\kappa$ time, $T^*_2(z^*,|h_0|_2)$, such that 
$$
E(t)\leq z^*\;\forall t\leq \bar{T}_\kappa=\min\{T^*_1(|h_0|_1.75),T^*_2(z^*,|h_0|_2),T_\kappa\} \,;
$$
see Section 9 of \cite{CoSh2006} for a proof.
We set $z^*=4C|h_0|_2^2$, and recalling \eqref{T1}, we  define 
$$
T^*=\min\{T_1^*,T_2^*\},\,\tilde{T}_\kappa=\min\{T^*(|h_0|_2,|h_0|_{1.75}),T_\kappa\}.
$$
As a consequence, we have the bounds
$$
E(t)\leq 4C|h_0|_2^2,\,|h^{\kappa\kappa}(t)|_{1.75}< \sigma,\,\forall t\leq \tilde{T}_\kappa.
$$
Our goal now is to show that we can reach $t=T^*$. To do so, we argue by contradiction. 
First, we assume that $\tilde{T}_\kappa=T^*$. Then we have a uniform-in-$\kappa$ lifespan, and a bound for every approximate solution. 
As a consequence, we can pass to the limit in $\kappa$. On the other hand, if $\tilde{T}_\kappa=T_\kappa$, we can extend the solution up to $\tilde{T}_\kappa+\delta$, for a small enough $\delta=\delta(z^*)$. Moreover, this extended solution verifies 
$$
E(t)\leq 4C|h_0|_2^2,\,|h^{\kappa\kappa}(t)|_{1.75}< \sigma,\,\forall 0\leq t\leq T_\kappa+\delta,\;\;\forall\kappa.
$$
By induction, we can reach $T^*$. This concludes the existence portion of Theorem \ref{localsmall}.

\subsection{Passing to the limit as $ \kappa \to 0$}
Once we have the uniform bound
$$
\max_{0\leq s\leq t}\{|h^{\kappa}(s)|_2^2\}+\int_0^t|h^{\kappa\kappa}(s)|_{2.5}^2+\|w(s)\|_{2,\pm}^2ds\leq C,
$$
we obtain the existence of weak limits
$$
h\in L^\infty(0,T^*;H^2(\bbR))\cap L^2(0,T^*;H^{2.5}(\bbR)),
$$
$$
h_t\in L^\infty(0,T^*;H^1(\bbR))\cap L^2(0,T^*;H^{1.5}(\bbR)),
$$
$$
w\in L^\infty(0,T^*;H^{1.5}(\bbR^2_\pm))\cap L^2(0,T^*;H^{2}(\bbR^2_\pm)),
$$
$$
\nabla Q\in L^\infty(0,T^*;H^{1.5}(\bbR^2_\pm)).
$$
Using the Rellich-Kondrachov  compactness theorem, we can prove that $(h,w,Q)$ is a distributional solution to \eqref{HS_semiALE1}.

\subsection{The uniqueness of the solution}
To prove  uniqueness of  solutions,  we use the energy method.  We assume that there exists two solutions, $h_1$ and $h_2$, corresponding to the 
same initial data $h_0$. 
Furthermore, we have that the corresponding higher-order energy functions $E_1(t)$ and $E_2(t)$,  defined in  \eqref{energy},  are uniformly bounded:
$$
E_1(t)+E_2(t)\leq 2z^*,\;\;\forall 0\leq t\leq T^*.
$$
We consider the new higher-order energy function
$$
\overline{E}(t)=\max_{0\leq s\leq t}\{|\overline{h}(s)|_2^2\}+\int_0^t\|\overline{w}(s)\|_{2,\pm}^2ds,
$$
where we denote the difference of both solutions using a bar:
$$
\overline{h}=h_1-h_2,\,\overline{\delta\psi}=\delta\psi_1-\delta\psi_2  \text{ and }\overline{w}=w_1-w_2.
$$
We have that
$$
\overline{E}(t)\leq E_1(t)+E_2(t)\leq 2z^*,\;\;\forall 0\leq t\leq T^*.
$$
The difference verifies the following system
\begin{subequations}\label{HS_semiALE1diff}
\begin{alignat}{2}
\mu\bar{w} + \nabla (\bar{Q} + \rho \bar{\delta\psi}) &= \left(\text{Id}-\frac{(\nabla\psi_1)^T\nabla\psi_1 }{J_1}\right)\mu w_1\nonumber\\
&\quad-\left(\text{Id}-\frac{(\nabla\psi_2)^T\nabla\psi_2 }{J_2}\right)\mu w_2  \qquad&&\text{in}\quad \{x_2\neq 0\}\,,\\
\div \bar{w} &= 0 &&\text{in}\quad \{x_2\neq 0\}\,,\\
\jump{\bar{w}^2} = \jump{\bar{Q}} &= 0 &&\text{on}\quad \{x_2 = 0\}\,,\\
\Delta \bar{\delta \psi}^\pm &= 0 &&\text{in}\quad \bbR^2_\pm\,,\\
\bar{\delta\psi}^\pm &= \bar{h} \qquad&&\text{on}\quad \{x_2 = 0\}\,,\\
\bar{h}_t &= \bar{w} \cdot e_2 \qquad &&\text{on}\quad \{x_2 = 0\}\,,\\
\bar{h} &= 0 &&\text{on}\quad \bbR \times \{t=0\}\,.
\end{alignat}
\end{subequations}
Recalling the equation for the evolution of the interface, we have that
\begin{equation}\label{energyuniqueness6}
|\bar{h}(t)|_{1.5}\leq \sqrt{t}C\sqrt{\bar{E}(t)}\leq \sqrt{t}C\sqrt{2z^*},\;|\bar{h}(t)|_{1.75}\leq C\sqrt[4]{t}\sqrt{\bar{E}(t)}\leq C\sqrt[4]{t}\sqrt{2z^*},
\end{equation}
$$
\mu^{-1}\Delta \bar{Q} = \mu^{-1}\div \big[ (\id - J_1 A_1 A^\rT_1) \nabla (Q_1 + \rho \delta\psi_1) -(\id - J_2 A_2 A^\rT_2) \nabla (Q_2 + \rho \delta\psi_2) \big] 
$$
with jump conditions $\jump{\bar{Q}} = 0$ and
\begin{eqnarray*}
\bigjump{\mu^{-1}\smallexp{$\displaystyle{} \frac{\p \bar{Q}}{\p \rN}$}} &=& \jump{\mu^{-1}(\id - J_1 A_1 A^\rT_1) (\nabla Q_1) e_2} - \jump{\mu^{-1}\rho J_1 (A_1)^2_i (A_1)^j_i \delta \psi_{1,_j}}\\
&&-\jump{\mu^{-1}(\id - J_2 A_2 A^\rT_2) (\nabla Q_2) e_2} + \jump{\mu^{-1}\rho J_2 (A_2)^2_i (A_2)^j_i \delta \psi_{2,_j}}. 
\end{eqnarray*}
Using that
$$
\text{Id}-JAA^T = \left[\begin{array}{cc}
\delta\psi_{,2} & - \delta\psi_{,1} \\
- \delta\psi_{,1} &  \frac{\delta\psi_{,1}^2}{1+\delta\psi_{,2}}+\delta\psi_{,2}
\end{array}
\right]\,,
$$
elliptic estimates show that
\begin{eqnarray*}
\|\nabla \bar{Q}\|_{1.25,\pm}&\leq& C\left[\right.\|\text{Id}-J_1A_1A_1^T\|_{1.25,\pm}\left(\|\nabla \bar{Q}\|_{1.25,\pm}+\|\nabla \bar{\delta\psi}\|_{1.25,\pm}\right)\\
&&+\|\nabla(Q_2+\rho\delta\psi_2)\|_{1.25,\pm}\|\nabla \bar{\delta\psi}\|_{1.25,\pm}\\
&&+\|\bar{JA}\|_{1.25,\pm}\|\nabla \delta\psi_1\|_{1.25,\pm}(\|A_1-\text{Id}\|_{1.25,\pm}+1)\\
&&+(\|J_2A_2-\text{Id}\|_{1.25,\pm}+1)\|\nabla\bar{\delta\psi}\|_{1.25,\pm}(\|A_1-\text{Id}\|_{1.25,\pm}+1)\\
&&+(\|J_2A_2-\text{Id}\|_{1.25,\pm}+1)\|\nabla \delta\psi_2\|_{1.25,\pm}\|\bar{A}\|_{1.25,\pm}\left.\right]\\
&\leq&C|h_1|_{1.75}\left(\|\nabla \bar{Q}\|_{1.25,\pm}+|\bar{h}|_{1.75}\right)+C|h_2|_{1.75}|\bar{h}|_{1.75}\\
&&+C|\bar{h}|_{1.75}|h_1|_{1.75}(|h_1|_{1.75}+1)+C(|h_2|_{1.75}+1)|\bar{h}\|_{1.75}(|h_1|_{1.75}+1)\\
&&+C(|h_2|_{1.75}+1)|h_2|_{1.75}|\bar{h}|_{1.75},
\end{eqnarray*}
and, using the smallness condition \eqref{H1.75},
$$
\|\nabla \bar{Q}\|_{1.25,\pm}\leq C|\bar{h}|_{1.75}.
$$
Similarly, we find that
\begin{eqnarray*}
\|\nabla \bar{Q}\|_{1.5,\pm}&\leq& C\left[\|\text{Id}-J_1A_1A_1^T\|_{1.5,\pm}\left(\|\nabla \bar{Q}\|_{1.25,\pm}+\|\nabla \bar{\delta\psi}\|_{1.25,\pm}\right)\right.\\
&&+\|\nabla(Q_2+\rho\delta\psi_2)\|_{1.25,\pm}\|\nabla \bar{\delta\psi}\|_{1.5,\pm}\\
&&+\|\nabla(Q_2+\rho\delta\psi_2)\|_{1.5,\pm}\|\nabla \bar{\delta\psi}\|_{1.25,\pm}\\
&&+\|\bar{JA}\|_{1.5,\pm}\|\nabla \delta\psi_1\|_{1.25,\pm}(\|A_1-\text{Id}\|_{1.25,\pm}+1)\\
&&+(\|J_2A_2-\text{Id}\|_{1.25,\pm}+1)\|\nabla\bar{\delta\psi}\|_{1.5,\pm}(\|A_1-\text{Id}\|_{1.25,\pm}+1)\\
&&+(\|J_2A_2-\text{Id}\|_{1.25,\pm}+1)\|\nabla \delta\psi_2\|_{1.25,\pm}\|\bar{A}\|_{1.5,\pm}\left.\right]\\
&\leq&C(|h_1|_{2}+|h_2|_{2})|\bar{h}|_{1.75}+C(|h_1|_{1.75}+|h_2|_{1.75}+1)|\bar{h}|_{2}\,,
\end{eqnarray*}
so we conclude that
$$
\|\nabla \bar{Q}\|_{1.25,\pm}+\|\bar{w}\|_{1.25,\pm}\leq C\sqrt[4]{t}.
$$
\begin{equation}\label{energyuniqueness5}
\|\nabla \bar{Q}\|_{1.5,\pm}+\|\bar{w}\|_{1.5,\pm}\leq C\sqrt[4]{t}+C|\bar{h}|_2.
\end{equation}

Next, as we have that
$$
\bar{v}=\frac{\nabla\bar{\delta\psi}w_1}{J_1}+\frac{\nabla\psi_2\bar{w}}{J_2}+\nabla\psi_2w_1\frac{-\bar{J}}{J_1J_2},
$$
and
$$
|\bar{v}|_{1.5}\leq c\sigma\|\nabla\bar{\delta\psi}\|_{2,\pm}+C\|\bar{w}\|_{2,\pm}.\,,
$$
using \eqref{h_reg_eq}, we compute that
\begin{equation}\label{energyuniqueness4}
\int_0^t|\bar{h}(s)|^2_{2.5}ds\leq \mathcal{P}(\bar{E}(t)).
\end{equation}
Recalling \eqref{eq1} and for $i=1$ or $2$,  denoting the matrix $B_i$ by
$$
B_i=\left(\begin{array}{cc}\delta\psi_{i,2}-\delta\psi_{i,1}^2 & -\delta\psi_{i,1}(1+\delta\psi_{i,2})\\
-\delta\psi_{i,1}(1+\delta\psi_{i,2}) & \delta\psi_{i,2}(1+\delta\psi_{i,2})\end{array}\right),
$$ 
we write the right-hand side in (\ref{HS_semiALE1diff}a) as
$$
RHS=B_1\frac{\mu\bar{w}}{J_1}+B_1\frac{-\mu w_2\bar{J}}{J_2J_1}+(B_1-B_2)\frac{\mu w_2}{J_2}.
$$
Testing against $\bar{w}$ and integrating-by-parts in (\ref{HS_semiALE1diff}a), we get that
\begin{multline}\label{energyuniqueness3}
|\bar{h}(t)|_0^2+\min\{\mu^+,\mu^-\}\int_0^t\|\bar{w}(s)\|^2_{0,\pm}ds\\
\leq Cz^*\left[\int_0^t\|\bar{w}(s)\|_{0,\pm}^2ds+\int_0^t\|\bar{w}(s)\|_{0,\pm}\|\nabla\bar{\delta\psi}(s)\|_{0,\pm}ds\right]\leq Cz^*\mathcal{P}(\bar{E}(t))t.
\end{multline}
The energy estimates show that
\begin{equation}\label{energyuniqueness}
\min\{\mu^+,\mu^-\}\int_0^t\|\bar{w}''(s)\|_{0,\pm}^2-\frac{\jump{\rho}}{2}\frac{d}{dt}|\bar{h}(s)|_{2}^2 ds=\int_0^t\int_{\bbR^2}RHS''\bar{w}''dxds\leq C(\sqrt{t}+\sqrt[4]{t})\mathcal{P}(\bar{E}(t))\,,
\end{equation}
and once again using the Hodge decomposition, we find that
\begin{equation}\label{energyuniqueness2}
\int_0^t\|\bar{w}(s)\|_{2,\pm}^2ds\leq C(t+\sqrt{t}+\sqrt[4]{t})\mathcal{P}(\bar{E}(t)).
\end{equation}
Collecting the previous estimates \eqref{energyuniqueness6}-\eqref{energyuniqueness2} and using the smallness of $\sigma$, we get the following polynomial inequality
$$
\bar{E}(t)\leq (t+\sqrt{t}+\sqrt[4]{t})\mathcal{P}(\bar{E}(t)),
$$
which implies the uniqueness. This concludes the proof of Theorem \ref{localsmall} for a infinitely-deep domain.

\section{Proof of Theorem \ref{localsmall}: Local well-posedness for the confined case}\label{sec3}
We define our reference domains 
$$
\Omega^+=\{(x_1,x_2),\, x_1\in\bbR\, (\text{or }x_1\in\bbT),0<x_2<c_t\},
$$
$$
\Omega^-=\{(x_1,x_2),\, x_1\in\bbR\, (\text{or }x_1\in\bbT),c_b<x_2<0\},
$$
and 
the reference interface
$$
\Gamma=\{(x_1,x_2),\, x_1\in\bbR\, (\text{or }x_1\in\bbT),x_2=0\}.
$$
We denote by
$$
\Gamma_{bot}=\{(x_1,x_2),\,x_2=c_b\},\,
\Gamma_{top}=\{(x_1,x_2),\,x_2=c_t\},
$$
 the fixed {\it bottom} and {\it top} boundaries.

We consider $\delta\psi^\pm$ as the solution of
$$
\Delta\delta\psi^+=0,\, \delta\psi^+=h\text{ if } x_2\in\Gamma,\, \delta\psi^+=\tilde{t}(x)\text{ if } x_2\in \Gamma_{top}.
$$
and
$$
\Delta\delta\psi^-=0,\, \delta\psi^-=h\text{ if } x_2\in\Gamma,\, \delta\psi^-=\tilde{b}(x)\text{ if } x_2\in \Gamma_{bot}.
$$
We define the mapping $\psi^\pm=e+(0,\delta\psi^\pm)$. In particular, 
$$
\psi^\pm(\Gamma,t)=(x_1,h(x_1,t)),\,\psi^+(\Gamma_{top},t)=(x_1,t(x_1)),\,\psi^-(\Gamma_{bot},t)=(x_1,b(x_1)),
$$
so
$$
\psi:\Omega^\pm\mapsto\Omega^\pm(t).
$$
Using estimates similar to those in \eqref{diffeo}, $\psi$ is a diffeomorphism if $h,\tilde{t},\tilde{b}$ are small in the $H^{1.75}$ norm. We define $v=u\circ \psi$, $q=p\circ \psi$, $A=(\nabla\psi)^{-1}$, $J= \text{det}(\nabla\psi)$, $w^k=JA^k_iv^i$ and $Q=q+\rho x_2$.   We write
$$
n_b=\frac{(\tilde{b}'(x_1),-1)}{\sqrt{1+(\tilde{b}'(x))^2}}\text{ and }n_t=\frac{(\tilde{t}'(x_1),1)}{\sqrt{1+(\tilde{t}'(x))^2}}
$$ 
for the normal vectors at $t(x)$ and $b(x)$, respectively. Then, we have the boundary conditions
$$
u\cdot n_b=0\text{ at }(x_1,x_2)\in \{(x_1,b(x_1))\},u\cdot n_t=0\text{ at }(x_1,x_2)\in \{(x_1,t(x_1))\},
$$
which translate to
$$
v\cdot n_b=0\text{ at }\Gamma_{bot},v\cdot n_t=0\text{ at }\Gamma_{top}.
$$
Since $JA^Te_2=(-\psi^2_{,1},\psi^1_1)$, then
$$
JA^Te_2=(-\tilde{b}'(x_1),1)\text{ at }\Gamma_{bot},JA^Te_2=(-\tilde{t}'(x_1),1)\text{ at }\Gamma_{top}.
$$
Using this, we can write the following boundary conditions for the semi-ALE velocity
$$
v\cdot (-n_b)=v\cdot (JA^Te_2)=(JAv)\cdot e_2=w\cdot e_2=0\text{ at }\Gamma_{bot},
$$
$$
v\cdot n_t=v\cdot (JA^Te_2)=(JAv)\cdot e_2=w\cdot e_2=0\text{ at }\Gamma_{top}.
$$
As in Section \ref{sec2.2}, we obtain
\begin{subequations}\label{HS_semiALE1general}
\begin{alignat}{2}
\mu w + \nabla (Q + \rho \delta\psi) &= \left(\text{Id}-\frac{(\nabla\psi)^T\nabla\psi }{J}\right)\mu w  \qquad&&\text{in}\quad \{x_2\neq 0\}\,,\\
\div w &= 0 &&\text{in}\quad \{x_2\neq 0\}\,,\\
\jump{w^2} = \jump{Q} &= 0 &&\text{on}\quad \{x_2 = 0\}\,,\\
w^2 &= 0 &&\text{on}\quad \{x_2 = c_b,c_t\}\,,\\
\Delta \delta \psi^\pm &= 0 &&\text{in}\quad \Omega^{\pm}_\pm\,,\\
\delta\psi^\pm &= h \qquad&&\text{on}\quad \{x_2 = 0\}\,,\\
\delta\psi^+ &= \tilde{t} \qquad&&\text{on}\quad \{x_2 = c_t\}\,,\\
\delta\psi^- &= \tilde{b} \qquad&&\text{on}\quad \{x_2 = c_b\}\,,\\
h_t &= w \cdot e_2 \qquad &&\text{on}\quad \{x_2 = 0\}\,,\\
h &= h_0 &&\text{on}\quad \bbR \times \{t=0\}\,.
\end{alignat}
\end{subequations}
Multiplying (\ref{HS_semiALE1general}a) with $e_2$ and evaluating at $\Gamma_{top}$, we obtain that
$$
Q^2_{,2}=-\rho^+\delta\psi_{,2}-\mu^+ w_1\tilde{t}'\,,
$$
and similarly at $\Gamma_{bot}$, 
$$
Q^2_{,2}=-\rho^-\delta\psi_{,2}-\mu^- w_1\tilde{b}'.
$$

Given $|h_0|_{1.75}< \sigma\ll1$ and $|\tilde{t}|_{2},|\tilde{b}|_{2}\leq \tilde{\sigma}\ll1$, we can regularize the problem as in (\ref{HS_semiALE_reg_fix}a-h) and we get an approximate solution $(w_\kappa,Q_\kappa,h_\kappa)$ that exists up to time $T_\kappa>0$. This solution has a finite energy, $E(t)$, as defined in \eqref{energy}. We take $T_\kappa$ small enough so $E(t)\leq z^*$ (for a constant that will be chosen later). 

With the boundary conditions for $Q$, we can form the associated elliptic problem as in Section \ref{sec4.3} and we get the following
 bounds (analogous to \eqref{Q1.25}, \eqref{Q1.5}):
$$
\|\nabla Q\|_{1.25,\pm}\leq c|h|_{1.75}+c\|w\|_{1.25,\pm}(|\tilde{t}|_{1.75}+|\tilde{b}|_{1.75}).
$$
and
$$
\|\nabla Q\|_{1.5,\pm}\leq c|h|_{1.75}|h|_2+c\|w\|_{1.5,\pm}(|\tilde{t}|_2+|\tilde{b}|_2)+c|h|_2,
$$
In particular, using the elliptic estimates for the pressure $Q$, we have that
$$
\|w\|_{1.5,\pm}\leq C|h|_2,\,
\|w\|_{1.25,\pm}\leq C|h|_{1.75},
$$
where we have used the smallness of $\tilde{\sigma}$ to obtain the desired polynomial bounds. The bound $h\in L^2 ( 0,T_ \kappa ; H^{2.5} (\Gamma))$ is obtained in the same way as in the proof of Theorem \ref{localsmall}. Using the boundary condition $w_2^\pm=0$ and $x_2=c_t,c_b$, the new terms coming from the boundaries in the energy estimates vanishes and we obtain the inequality
$$
E(t)\leq C|h_0|_2^2+C\sigma_1 E(t) +t\mathcal{P}(E(t)),
$$
which, since $\sigma_1\ll1$, implies the existence of a uniform-in-$\kappa$ $T^*$ such that
$$
E(t)\leq z^*, |h(t)|_{1.75}<\sigma_1\,\forall\,0\leq t\leq \min\{T_\kappa,T^*\}.
$$
We reach $T^*$ by induction. The uniqueness is obtained in the same way. 
%
This proves Theorem \ref{localsmall}.

\section{Proof of Theorem \ref{globalsmall}: Global existence and decay to equilibrium}\label{sec4}
Recall that in this case we have $\Omega^+(t)\cup\Omega^-(t)=\bbT\times \bbR$.
\subsection{A linearization of (\ref{HS_semiALE})}\label{sec4.1.1}
We denote by $\hat{f}$ the Fourier series of $f$. We write $\Lambda$ for the square root of the Laplacian:
$$
\Lambda f=\sqrt{-\partial_x^2}f,\qquad \widehat{\Lambda f}(\xi)=|\xi|\hat{f}(\xi).
$$
It is well-known that the previous operator has a kernel representation
\begin{equation*}
\Lambda f(x_1) = \frac{1}{2\pi}\, \text{p.v.} \int_{-\pi}^\pi \frac{f(x_1)-f(x_1-s)}{\sin^2\left(\frac{s}{2}\right)}\, d s,
\end{equation*}
From \eqref{deltapsi} and $\delta \psi^{-}(x_1,x_2)=\delta \psi^{+}(x_1,-x_2)$, we have that
\begin{equation*}
\delta \psi^{\pm},_2 = \mp\Lambda h \qquad \text{on}\quad \{x_2 = 0\}\,,
\end{equation*}
so that the Dirichlet-to-Neumann map is the Zygmund operator. 

We define the Neumann-to-Dirichlet map $\Lambda^{-1}$ by
$$
\widehat{\Lambda^{-1} f}(\xi)=|\xi|^{-1}\hat{f}(\xi).
$$
Notice that if $f$ has zero mean,  the previous operator is well-defined.

Equation (\ref{HS_semiALE1}a) may be written as
$$
\mu w + \nabla (Q + \rho \delta\psi) = F
$$
with 
$$
F=(F_1,F_2)=\left(\text{Id}-\frac{(\nabla\psi)^T\nabla\psi }{J}\right)\mu w.
$$
By taking the inner product of this equation with $e_2$, and then evaluating on $\{x_2 = 0\}$, we find that
\begin{equation}\label{eqlinear}
\mu^\pm h_t + Q^{\pm}_{,2}+\rho^\pm\delta\psi^{\pm}_{,2}=F_2^\pm,
\end{equation}
where
$$
F_2^\pm=-(\mu^\pm w_1^\pm h'+\mu^\pm w_2(\mp\Lambda h))\,.
$$
Summing over the two phases, 
$$
F_2^++F_2^-=-(\mu^+ w_1^++\mu^-w^-_1) h'-\jump{\mu}\Lambda h w_2.
$$

On the other hand, taking the divergence of the equation (\ref{HS_semiALE1}a), we get
$$
\Delta Q=\text{div}F.
$$
The continuity of $q$ gives us the jump condition $\jump{Q}=0$. Using equation (\ref{HS_semiALE1}a) and \eqref{eq1}, 
$$
\jump{Q_{,2}}=\left(\rho^++\rho^-\right)\Lambda h-\jump{\mu}h_t+\jump{F_2},
\text{ with }
\jump{F_2}=-(\jump{\mu w_1} h'-(\mu^++\mu^-)w_2 \Lambda h).
$$
We define $\bar{Q}^\pm$ such that
\begin{eqnarray*}
\Delta \bar{Q}^\pm &=& 0 \text{ in }\quad \bbR^2_\pm \,,\\
\bar{Q}^\pm &=& -\frac{\rho^++\rho^-}{2}h+\frac{\jump{\mu}}{2}\Lambda^{-1}h_t \qquad\text{ on }\quad \{x_2 = 0\}\,.
\end{eqnarray*}
Then,
$$
\bar{Q}^\pm_{,2}=\pm\frac{\rho^++\rho^-}{2}\Lambda h\mp\frac{\jump{\mu}}{2}h_t,\text{ on } \{x_2=0\}.
$$

Consequently,  $\jump{\bar{Q}}=0$ and $\jump{\bar{Q}^\pm_{,2}}=(\rho^++\rho^-)\Lambda h-\jump{\mu}h_t$. Setting $\tilde{Q}=Q-\bar{Q}$,
then $\tilde{Q}$ is a solution of
\begin{equation}\label{Qtilda}
\Delta \tilde{Q}=\text{div}F,
\end{equation}
with the jump conditions  
\begin{equation}\label{Qtilda2}
\jump{\tilde{Q}}=0 \text{ and }\jump{\tilde{Q}_{,2}}=\jump{F_2}.
\end{equation}
As a consequence, equation \eqref{eqlinear} becomes
\begin{equation*}
\mu^\pm h_t \mp\frac{\jump{\mu}}{2}h_t \mp\rho^\pm\Lambda h\pm\frac{\rho^++\rho^-}{2}\Lambda h=F_2^\pm-\tilde{Q}^\pm_{,2},
\end{equation*}
Summing the equations for both phases, we obtain
\begin{equation}\label{eqlineal2b}
\frac{\mu^++\mu^-}{2}h_t =\frac{\jump{\rho}}{2}\Lambda h+\frac{F^+_2+F^-_2-\tilde{Q}_{,2}^+-\tilde{Q}_{,2}^-}{2}.
\end{equation}

\subsection{Energy estimates for the total norm}  For notational simplicity, we set $\jump{\rho}=-2$ and $\mu^++\mu^-=2$, but in what follows, any finite
values are permissible.   Using the Duhamel Principle on  \eqref{eqlineal2b}, we write the so-called mild solution as
\begin{equation}\label{mild}
h(t) = h_0e^{-\Lambda t}+\int_0^t\left(\frac{F^+_2(s)+F^-_2(s)-\tilde{Q}_{,2}^+(s)-\tilde{Q}_{,2}^-(s)}{2}\right)e^{-\Lambda (t-s)}ds\,.
\end{equation}
Note, that in this analysis, we are restricting our attention to zero mean, periodic functions.
As to the linear semi-group, it is well-known that 
\begin{equation}\label{linear-decay}
\|e^{-\Lambda t}\|_{L^2\mapsto L^2}\leq e^{-t}\,,
\end{equation} 
since the first eigenvalue of $\Lambda$ agrees with the first eigenvalue of $- \Delta $.

Let $\sigma_2$ denote a constant that will be fixed later.   We choose  $h_0\in H^2$  such that $|h_0|_2\leq \sigma_2\ll1$. 
Using Theorem \ref{localsmall}, there exists a local in time solution up to time $T=T(h_0)$. Moreover, this solution remains in the Rayleigh-Taylor stable regime and satisfies
$$
\max_{0\leq t\leq T}|h(t)|^2_{2}+\int_0^t |h(s)|_{2.5}^2ds\leq C_1|h_0|_2^2,
$$
and
\begin{equation}\label{smallglobal}
\max_{0\leq t\leq T}|h(t)|_{1.75}<\sigma_{0.25}\ll1,
\end{equation}
where $C_1$ and $\sigma_{0.25}$ are the constants appearing in Theorem \ref{localsmall}. We define the new {\it total norm} as
\begin{equation}\label{totalenergy}
\vertiii{(w,h)}^2_T=\max_{0\leq t\leq T}\left\{|h(t)|_2^2+e^{\alpha t}|h(t)|_0^2+\int_0^t\|w(s)\|_{2,\pm}^2ds\right\},
\end{equation}
for a given $0<\alpha<2$. Hence, a uniform bound for $\vertiii{(w,h)}_T$ for every $t>0$ implies the $e^{-\alpha t/2}$ decay-rate for $|h(t)|_0$. 

Just as we obtained the $H^{2.5}$ estimate for $h^{ \kappa \kappa }$ in (\ref{h2.50}), we have the following estimate:
\begin{eqnarray}\label{h6.2}
|h^{\prime\prime}|_{0.5} &\leq& C (1+\left|h\right|^3_{1.75}) \left(\left|v^+\right|_{1.5}+\left|v^-\right|_{1.5}\right)+C |h^{\prime\prime}|_{0.5} |v\cdot (-h^{\prime},1)|_{0.75} \,.
\end{eqnarray}
Using the estimates 
\eqref{w1.25}, \eqref{smallv}, \eqref{h2.51}-\eqref{h2.53}) together with \eqref{smallglobal}, we obtain that
\begin{align*}
\int_0^t|h(s)|_{2.5}^2ds & \leq C\left(\int_0^t\|w(s)\|^2_{2,\pm}ds+\int_0^t|h|_{2}^4ds+\int_0^t\left|h(s)\right|_{2}^2\|w(s)\|^2_{1.5,\pm}ds\right).
\end{align*}
Using the interpolation inequality
$|h|^2_2\leq C|h|_{1.5}|h|_{2.5}$, 
together with \eqref{smallglobal}, we find that
\begin{align*}
\int_0^t|h(s)|_{2.5}^2ds & \leq C\vertiii{(w,h)}^2\left(1+\vertiii{(w,h)}^2\right).
\end{align*}

Our goal is to show that $e^{\alpha t}|h(t)|_0^2$ remains small for all time.  To do so, we
take the $L^2(\Gamma)$-norm of equation \eqref{mild}, and find that
\begin{equation}\label{decay-inequality}
|h(t)|_0\leq e^{-t}|h_0|_0+\frac{1}{2}\int_0^t|F^+_2(s)+F^-_2(s)-\tilde{Q}_{,2}^+(s)-\tilde{Q}_{,2}^-(s)|_0e^{-(t-s)}ds.
\end{equation} 
We define
\begin{equation}\label{eqI1}
I_1=\frac{1}{2}\int_0^t|F^+_2(s)+F^-_2(s)|_0e^{-(t-s)}ds,
\end{equation}
\begin{equation}\label{eqI2}
I_2=\frac{1}{2}\int_0^t|\tilde{Q}_{,2}^+(s)+\tilde{Q}_{,2}^-(s)|_0e^{-(t-s)}ds,
\end{equation}
We are going to use the linear decay rate (\ref{linear-decay}) to establish the nonlinear decay rate for small solutions.   This will amount to
establishing certain integrability properties of the nonlinear term (\ref{decay-inequality}).

Notice now that, using \eqref{Qtilda} and \eqref{Qtilda2}, we have the bound
$$
\|\nabla\tilde{Q}\|_{0,\pm}\leq C\|F\|_{0,\pm}.
$$
Given $\phi\in H^1(\bbR^2)$, we compute
$$
\int_{\{x_2=0\}} \tilde{Q}_{,2}\phi dx_1=\int_{\Omega^\pm}\nabla \tilde{Q}\nabla \phi dx-\int_{\Omega^\pm} F\nabla\phi+\int_{\{x_2=0\}}F\cdot N\phi dx_1,
$$
so
$$
|\tilde{Q}_{,2}|_{-0.5}\leq C(\|F\|_{0,\pm}+|F_2|_{-0.5}).
$$
By elliptic estimates and the trace theorem,
$$
|\tilde{Q}_{,2}|_{0.5}\leq C(\|F\|_{1,\pm}+|F_2|_{0.5}).
$$
Thus, using interpolation,
$$
|\tilde{Q}_{,2}|_{0}\leq C(\|F\|_{0.5,\pm}+|F_2|_{0}).
$$

Using the H\"{o}lder inequality and the boundedness of the Hilbert transform in $L^p$ for $1<p<\infty$, we have that
$$
|F^+_2(s)+F^-_2(s)|_0\leq C|w|_{L^4}|h'|_{L^4}.
$$
Due to the Sobolev embedding theorem, the trace theorem and elliptic estimates, we have that
$$
|w|_{L^4}|h'|_{L^4}\leq C|w|_{0.25}|h|_{1.25}\leq C\|w\|_{0.75,\pm}|h|_{1.25}\leq C|h|^2_{1.25}.
$$
In particular, 
$$
|F^+_2(s)+F^-_2(s)|_0\leq C|h|_{0}|h|_{2.5}.
$$
Using \eqref{eqI1}, we find that
\begin{eqnarray}\label{decay1}
I_1&\leq& C\vertiii{(w,h)}_T^{0.5}\int_0^t(e^{\alpha s})^{-0.5}|h(s)|_{2.5}e^{-(t-s)}ds\nonumber\\
&\leq& C(1+\vertiii{(w,h)}_T^2)^{0.5}\vertiii{(w,h)}_T^{1.5}e^{-t}\left(\int_0^te^{(2-\alpha)s}ds\right)^{0.5}\nonumber\\
&\leq& \frac{C}{\sqrt{2-\alpha}}(1+\vertiii{(w,h)}_T^2)^{0.5}\vertiii{(w,h)}_T^{1.5}e^{-t}\left(e^{(2-\alpha)t}-1\right)^{0.5}.
\end{eqnarray}
The remaining terms \eqref{eqI2} are written as
$$
I_2\leq\frac{1}{2}\int_0^t(\|F\|_{0.5,\pm}+|F^+_2(s)|+|F^-_2(s)|_0)e^{-(t-s)}ds
$$
The terms with $|F^+_2(s)|+|F^-_2(s)|_0$ are similar to those with $|F^+_2(s)+F^-_2(s)|_0$. 
Using \eqref{eq1} and elliptic estimates, we have that
\begin{align}
\|F\|_{0,\pm} & \leq C\|w\|_{L^4}\|\nabla\delta\psi\|_{L^4}\nonumber\\
& \leq C\|w\|_{0.5,\pm}\|\nabla\delta\psi\|_{0.5,\pm}\nonumber\\
& \leq C|h|_{1}|h|_{1},\label{interF}
\end{align}
and, using \eqref{w1.25},
\begin{align}
\|\nabla F\|_{0,\pm} & \leq C\left(\|\nabla w\|_{L^4}\|\nabla\delta\psi\|_{L^4}+\|w\|_{L^4}\|\nabla^2\delta\psi\|_{L^4}\right)\nonumber\\
& \leq C\left(\|w\|_{1.5,\pm}\|\nabla \delta\psi\|_{0.5,\pm}+\|w\|_{0.5,\pm}\|\nabla\delta\psi\|_{1.5,\pm}\right)\nonumber\\
& \leq C|h|_1|h|_2.\label{interF2}
\end{align}
Due to linear interpolation between \eqref{interF} and \eqref{interF2}, we have
\begin{equation}
\|F\|_{0.5,\pm}  \leq C|h|_{1}|h|_{1.5} \leq C|h|_{0}|h|_{2.5}.\label{decay2}
\end{equation}
Collecting the estimates \eqref{decay1} and \eqref{decay2}, 
$$
\frac{1}{2}\int_0^t|F^+_2(s)+F^-_2(s)-\tilde{Q}_{,2}^+(s)-\tilde{Q}_{,2}^-(s)|_0e^{-(t-s)}ds
\leq (1+\vertiii{(w,h)}_T^2)^{0.5}\vertiii{(w,h)}_T^{1.5}e^{-\alpha t/2},
$$
and
\begin{align*}
e^{\alpha t}|h(t)|^2_0 & \leq 2\left(e^{(\alpha-2)t}|h_0|^2_0 +C\left(1+\vertiii{(w,h)}_T^2\right)\vertiii{(w,h)}_T^{3}\right)\\
& \leq 2|h_0|_0^2+\left(1+\vertiii{(w,h)}_T^2\right)\vertiii{(w,h)}_T^{3}.
\end{align*}
Now we have to estimate the terms 
$$
\max_{0\leq t\leq T} |h(t)|_2^2+\int_0^t\|w(s)\|^2_{2,\pm}ds.
$$
Using the same type of estimates as in Sections \ref{sec4.3},\ref{sec4.5} and \ref{sec4.6},
we get the inequality
$$
\vertiii{(w,h)}_T\leq C_2|h_0|_2+\mathcal{P}(\vertiii{(w,h)}_T),
$$
where the polynomial $\mathcal{P}$ has order $m$ with $m > 1$. Now, by choosing the initial data to be sufficiently small, we have a global bound 
$$
\vertiii{(w,h)}_T\leq 2C_2|h_0|_2\leq 2C_2\sigma_2.
$$
Furthermore, using interpolation between Sobolev spaces, we have
$$
\sup_{0\leq t\leq T}|h(t)|_{1.75}^2\leq 2C_2\sigma_2e^{-\frac{\alpha t}{8}}.
$$
We take $\sigma_2$ small enough so that
$$
2C_2\sigma_2<\sigma_{0.25},
$$
and we obtain that the smallness of $|h|_{1.75}$ propagates. 

Consequently, at time $t=T$, the solution remains in the stable regime (see Section \ref{sec2.5.4}), and the condition \eqref{smallglobal} is, in fact, improved.
 Due to this fact, we can apply Theorem \ref{localsmall} to continue the solution up to $t=2T$. As the same estimates hold in the time interval $nT\leq t\leq (n+1)T$ for $n\in\bbZ^+$, we conclude the proof of Theorem \ref{globalsmall} by means of a classical continuation argument.

\section{Proof of Theorem \ref{localonephase}: Local well-posedness for the one-phase problem}\label{sec5}
We now focus our attention  on the one-phase Muskat problem (\ref{HS_Eulerian_Onephase}a-e).

\subsection{Constructing the family of  diffeomorphisms $\psi( \cdot ,t)$}\label{subsection_psi}
We define our reference domain,  fixed {\it bottom} boundary, and reference interface, respectively, as follows:
\begin{equation}\label{ref_domain}
\Omega=\mathbb{T}\times[c_b,0]\,, \quad \Gamma_{bot}=\{(x_1,c_b),x_1\in\bbT\}\,, \text{ and }\Gamma=\{(x_1,0),x_1\in\bbT\}.
\end{equation} 
In particular, our reference domain is $C^\infty$. We let $N=e_2$ denote
the unit normal vector on $\Gamma$.  Given a function $h\in C(0,T;H^2)$ with initial data $h(0)=h_0$, we fix $0<\delta\ll1$ and define
\begin{equation}\label{Omegadelta}
\Omega^\delta(0)=\{(x_1,x_2),\, x_1\in\bbT,\,c_b<x_2<\jdel h_0(x_1)\},
\end{equation}
\begin{equation}\label{Gammadelta}
\Gamma^\delta(0)=\{(x_1,\jdel h_0(x_1)),\, x_1\in\bbT\},
\end{equation}
and
\begin{equation}\label{phi1}
\phi_1(x_1,x_2)=\left(x_1,x_2+\jdel h_0(x_1)\left(1-\frac{x_2}{c_b}\right)\right).
\end{equation}
This function $\phi_1:\Omega\rightarrow\Omega^\delta(0)$ is a $C^\infty$ diffeomorphism.

Next, we define the function $\phi_2:\Omega^{\delta}(0)\rightarrow\Omega(0)$ as the solution to the following elliptic problem:
\begin{subequations}\label{phi2}
\begin{alignat}{2}
\Delta \phi_2&= 0  \qquad&&\text{in}\quad \Omega^\delta(0)\times[0,T]\,,\\
\phi_2 &= e+[h_0(x_1)-\jdel h_0(x_1)]e_2 \qquad &&\text{on}\quad \Gamma^\delta(0)\times[0,T]\,,\\
\phi_2 &= e &&\text{on}\quad \Gamma_{bot}\times[0,T] \,.
\end{alignat}
\end{subequations}
Since $\Omega^\delta(0)$ is a $C^ \infty $ domain, standard
elliptic regularity theory shows that $\phi_2 \in H^{2.5}(\Omega^\delta(0))$, and since
for $ \delta >0$ taken sufficiently small, $|h_0-\jdel h_0|_{2}\ll 1$, $\|\nabla \phi_2 - \text{Id} \|_{C^0} \ll 1$;
hence, from the inverse function theorem,
$\phi_2:\Omega^{\delta}(0)\rightarrow\Omega(0)$ is an $H^{2.5}$-class diffeomorphism. 

We define
\begin{equation}\label{psi0a}
\psi(0)=\phi_2\circ\phi_1:\Omega\rightarrow \Omega(0).
\end{equation}
This mapping is also a diffeomorphism that maps
$$
\psi(0):\Gamma\rightarrow \Gamma(0)
$$ 

Furthermore, using the chain rule, we have that
$$
\|\psi(0)\|_{2,-}\leq  c(\delta)|h_0|_{1.5},\,\,
\|\psi(0)\|_{3,-}\leq  c(\delta)|h_0|_{2.5}.
$$
Using interpolation, we obtain
\begin{equation}\label{psi02.5}
\|\psi(0)\|_{2.5,-}\leq c(\delta)|h_0|_{2}.
\end{equation}
(We note that $\delta>0$ is fixed number, so the dependence of the constant in (\ref{psi02.5}) on $\delta$ is harmless.)
We have thus defined  our initial diffeomorphism $\psi(0)$; we next define our time-dependent family of diffeomorphisms $\psi(t) = \psi( \cdot, t)$ as 
follows:
\begin{subequations}\label{psita}
\begin{alignat}{2}
\Delta \psi(t)&= \Delta \psi(0)  \qquad&&\text{in}\quad \Omega\times[0,T]\,,\\
\psi(t)&= e+h(x_1,t)e_2 \qquad &&\text{on}\quad \Gamma\times[0,T]\,,\\
\psi(t)&= e \qquad &&\text{on}\quad \Gamma_{bot}\times[0,T]\,.
\end{alignat}
\end{subequations}
Writing $J(t)=\text{det}(\nabla\psi(t))$, we have the bounds
\begin{equation}\label{J1.25}
\|J(t)-J(0)\|_{1.25,-}\leq C\|\psi(t)-\psi(0)\|^2_{2.25,-}\leq C|h(t)- h_0|_{1.75}^2. 
\end{equation}
Consequently, using $h\in C(0,T;H^2)$, for sufficiently small time $t$, we have
$$
\min_{x\in\Omega^-}\frac{J(0)}{2}< J(t)< 2\max_{x\in\Omega^-} J(0),
$$
and we once again see that $\psi(t): \Omega \to \Omega(t)$ is a diffeomorphism. 
Furthermore, $\psi(t)$ is a $H^{2.5}$-class diffeomorphism thanks to the elliptic estimate
$$
\|\psi(t)\|_{2.5,-}\leq  c(|h(t)|_{2}+1).
$$

\subsection{The ALE formulation}
With $\psi(t)= \psi( \cdot ,t)$ defined in Section \ref{subsection_psi} (see \eqref{psi0a} and \eqref{psita}),  we set $A=(\nabla\psi)^{-1}$ and $J =\det \nabla \psi$.
As we noted above,  $\psi(t,\Gamma)=\Gamma(t)$.   We define our ALE variables: $v=u\circ\psi,q=p\circ\psi$. 

We let 
$$
\tilde{\tau}=(1,h'(x_1,t)),\,\tilde{n}=(-h'(x_1,t),1),
$$
denote the (non-unitary) tangent and normal vectors, respectively,  to $\Gamma(t)$. 
We let $g= |\psi'|^2$ denote the induced metric,  and define the unit tangent vector $\tau=\tilde{\tau}/\sqrt{g}$ and the unit normal vector $n=\tilde{n}/\sqrt{g}$. Since the interface $\Gamma(t)$  moves with the fluid, 
$$
v\cdot \tilde{n}=\psi_t\cdot \tilde{n}=h_tN\cdot \tilde{n}=h_t.
$$
Hence, the ALE representation of the one-phase Muskat problem is given as
\begin{subequations}\label{HS_ALE_Onephase}
\begin{alignat}{2}
v^i+ A^k_i(q+\psi^2)_{,k}&=0 \qquad&&\text{in}\quad\Omega\times[0,T]\,,\\
A^i_jv^j_{,i} &= 0 &&\text{in}\quad\Omega\times[0,T]\,,\\
h(t)&=h_0+\int_0^t v^i \tilde{n}_ids\qquad&&\text{on}\quad\Gamma\times[0,T]\,,\\
q&=0 &&\text{on}\quad\Gamma\times[0,T]\,,\\
v\cdot e_2&=0 &&\text{on}\quad\Gamma_{bot}\times[0,T]\,.
\end{alignat}
\end{subequations}

\subsubsection{The matrix $A$}
From the identity $A\nabla\psi=\text{Id}$, we see that
\begin{equation}\label{propA}
A_t=-A\nabla \psi_t A,\qquad A_{,k}=-A\nabla \psi_{,k} A,\qquad A''=-2 A'\nabla\psi' A-A\nabla\psi'' A.
\end{equation}
These identities will be often used.

\subsection{A smooth approximation of the ALE formulation}
Given an initial data $h_0\in H^2$ and two regularization parameters $\epsilon,\kappa>0$, we define a smooth approximation of the initial height function $\jkap h_0$. We write $h_{\epsilon,\kappa}(x_1,t)$ for the free boundary corresponding to the initial data $\jkap h_0$. 

We define 
$$
\Omega^{\delta,\epsilon}(0)=\{(x_1,x_2),\, x_1\in\bbT,\,c_b<x_2<\jdel\jkap h_0(x_1)\},
$$
$$
\Gamma^{\delta,\epsilon}(0)=\{(x_1,\jdel\jkap h_0(x_1)),\, x_1\in\bbT\},
$$
and 
\begin{equation}\label{phi1epsilonkappa}
\phi^{\epsilon,\kappa}_1(x_1,x_2)=\left(x_1,x_2+\jdel\jkap h_0(x_1)\left(1-\frac{x_2}{c_b}\right)\right).
\end{equation}
We construct $\phi^{\epsilon,\kappa}_2$ by solving
\begin{subequations}\label{phi2epsilonkappa}
\begin{alignat}{2}
\Delta \phi^{\epsilon,\kappa}_2&=0\qquad &&\text{on}\quad \Omega^{\delta,\epsilon}(0)\times[0,T_{\epsilon,\kappa}]\,,\\
\phi^{\epsilon,\kappa}_2(t)&= e+[\jeps\jeps\jkap h_0(x_1)-\jdel\jkap h_0(x_1)]e_2 \qquad &&\text{on}\quad \Gamma\times[0,T_{\epsilon,\kappa}],\,,\\
\phi^{\epsilon,\kappa}_2 &= e &&\text{on}\quad \Gamma_{bot}\times[0,T_{\epsilon,\kappa}] \,.
\end{alignat}
\end{subequations}

We can use Proposition \ref{approxsol} together with \eqref{phi1epsilonkappa} and \eqref{phi2epsilonkappa} to construct solutions to the approximate $ \epsilon \kappa$-problem on a time interval $[0,T_{\epsilon,\kappa}]$:
\begin{subequations}\label{HS_ALE_Onephase_reg}
\begin{alignat}{2}
v_{\epsilon,\kappa}^i+ (A_{\epsilon,\kappa})^k_i(q_{\epsilon,\kappa}+\psi_{\epsilon,\kappa}^2)_{,k}&= 0 \qquad&&\text{in}\quad\Omega\times[0,T_{\epsilon,\kappa}]\,,\\
(A_{\epsilon,\kappa})^i_j(v_{\epsilon,\kappa})^j_{,i} &= 0 &&\text{in}\quad\Omega\times[0,T_{\epsilon,\kappa}]\,,\\
h_{\epsilon,\kappa}(t)&=\jkap h_0+\int_0^t v_{\epsilon,\kappa}^i J_{\epsilon,\kappa}(A_{\epsilon,\kappa})^k_iN^k ds\qquad&&\text{on}\quad\Gamma\times[0,T_{\epsilon,\kappa}]\,,\\
q_{\epsilon,\kappa}&=0 &&\text{on}\quad\Gamma\times[0,T_{\epsilon,\kappa}]\,,\\
v_{\epsilon,\kappa}\cdot e_2&=0 &&\text{on}\quad\Gamma_{bot}\times[0,T_{\epsilon,\kappa}]\,,\\
\psi_{\epsilon,\kappa}&=\phi^{\epsilon,\kappa}_2\circ\phi^{\epsilon,\kappa}_1&&\text{in}\quad\Omega\times\{t=0\}\,,\\
\Delta \psi_{\epsilon,\kappa}(t)&=\Delta \psi_{\epsilon,\kappa}(0)&&\text{in}\quad\Omega\times[0,T_{\epsilon,\kappa}]\,,\\
\psi_{\epsilon,\kappa}(t)&=e+\jeps\jeps h_{\epsilon,\kappa}(t)N&&\text{on}\quad\Gamma\times[0,T_{\epsilon,\kappa}]\,,\\
\psi_{\epsilon,\kappa}(t)&=e&&\text{on}\quad\Gamma_{bot}\times[0,T_{\epsilon,\kappa}] \,,
\end{alignat}
\end{subequations}
where
$$
A_{\epsilon,\kappa} = [\nabla \psi_{\epsilon,\kappa}] ^{-1} \text{ and } J_{\epsilon,\kappa} = \det \nabla \psi_{\epsilon,\kappa} \,.
$$

Having solutions to (\ref{HS_ALE_Onephase_reg}), we focus on obtaining the uniform (in $\epsilon$ and $\kappa$) lifespan. We are going to perform the estimates in a two step procedure. First, we focus on $\kappa-$independent estimates (that may depend on $\epsilon$), and then we focus on $\epsilon-$independent estimates.

To simplify notation, we drop the $\epsilon$ and $\kappa$ notation except when it is computationally used,  but note that our dependent variables implicitly depend upon $\epsilon$ and $\kappa$.

\subsection{$\kappa$-independent estimates} Abusing notation, we redefine
$$
\tilde{\tau}=(1,\jeps\jeps h'(x_1,t)),\,\tilde{n}=(-\jeps\jeps h'(x_1,t),1).
$$

We define the higher-order energy function to be
$$
E(t)=\max_{0\leq s\leq t}|h^{\kappa}(s)|^2_{2}+\int_0^t\|v(s)\|_{2,-}^2ds.
$$
The solutions to (\ref{phi2epsilonkappa}) have sufficient regularity to ensure that our higher-order energy function  $E(t)$ is continuous.
We take $T_{\epsilon,\kappa}$ small enough to ensure that the following four conditions hold:
\begin{enumerate}
\item for a fixed constant $\delta_1>0$ that only depends on $h_0$,
 \begin{equation}\label{bootstrap1}
\|A(t)-A(0)\|_{L^\infty}\leq \delta_1\ll 1 \,;
\end{equation}

\item  $E(t)\leq z^*$ for a fixed constant $z^*$ (that will be chosen below)\,;
\item $\min_{0\leq t\leq T_{\kappa}}-q_{,2}(t)>-\frac{q_{,2}(0)}{2}$\,;
\item with $c_b$ given in (\ref{ref_domain}), 
 \begin{equation}\label{bootstrap4}
\min_{x_1} h(x_1,t)>c_b \,.
\end{equation}
\end{enumerate}
Again, we let $C$ denote a constant that may change from line to line. This constant may depend on $h_0$ and $\epsilon$, but not on $\kappa.$ 
We let $\mathcal{P}(x)$ denote a polynomial with coefficients that may depend on $h_0$ and $\epsilon$, but, again, they do not depend on $\kappa$. This polynomial may change from line to line. 

Our goal is to prove the following polynomial estimate for the energy:
$$
E(t)\leq \mathcal{M}_0+\sqrt[12]{t}\mathcal{Q}(E(t)),
$$
for a certain constant $\mathcal{M}_0$ and polynomial $\mathcal{Q}$. We choose $T_{\epsilon, \kappa}\leq \min\{1,T^*_1\}$ with $T^*_1$ such that
$$
\mathcal{Q}(z^*)\left(T^*_1\right)^{1/12}\leq \delta_2\ll 1,
$$
for $ \delta _2$ a fixed constant satisfying  $0<\delta_2<\delta_1\ll 1$.

\subsubsection{Estimates for some lower-order norms of $h^{\kappa}$} 
From (\ref{HS_ALE_Onephase_reg}c),
\begin{equation}\label{ht1.5}
\int_0^t |h_{t}|^2_{1.5}ds\leq  C \, E(t).
\end{equation}

Using \eqref{ht1.5} together with the fundamental theorem of calculus, we have that
\begin{equation}\label{lowf}
|h(t)-\jkap h_0|_{1.5}\leq \sqrt{t}\left(\int_0^t|h_{t}|_{1.5}^2ds\right)^{1/2}\leq C \sqrt{t}\sqrt{E(t)}
\end{equation}

Now,
\begin{equation}\label{lowf2}
|h^{\kappa}(t)-\jkap h^{\kappa}_0|_{1.75}\leq C|h(t)-\jkap h_0|_{1.5}^{1/2}|h^{\kappa}(t)-\jkap h^{\kappa}_0|_{2}^{1/2}\leq C \sqrt{E(t)}t^{1/4},
\end{equation}
and
$$
|h^{\kappa}(t)|_{1.75}\leq C|h_0|_{1.75}.
$$
Notice that, by taking a small enough time and using \eqref{lowf}, we recover our \emph{bootstrap} assumption \eqref{bootstrap4}.

\subsubsection{Some estimates for the mapping $\psi$} We consider here the regularity properties of the mapping $\psi$ given in (\ref{HS_ALE_Onephase_reg}e-h). We have the following estimates
$$
\|\psi(0)\|_{2,-}\leq  C(\delta)|h_0|_{1.5},\,\,
\|\psi(0)\|_{2.5,-}\leq C(\delta)|h_0|_{2},\,\,\|\psi(0)\|_{3,-}\leq  C(\delta)|h_0|_{2.5},
$$
and, using elliptic estimates, \eqref{J1.25}, and \eqref{lowf2},
\begin{equation}\label{psi2.25b}
\|\psi(t)-\psi(0)\|_{2.25,-}\leq C|\psi(t)-\psi(0)|_{1.75}\leq C|h(t)-h(0)|_{1.75}\leq \sqrt[4]{t}C\sqrt{E(t)},
\end{equation}
\begin{equation}\label{J1.25b}
\|J(t)-J(0)\|_{1.25,-}\leq \sqrt[4]{t}C\sqrt{E(t)}. 
\end{equation}
By taking a small enough time, we can obtain the uniform bounds
\begin{equation}\label{Jmin}
\max_{0\leq t\leq T_{\epsilon, \kappa}} \|J(t)\|_{1.25,-}+\|\psi(t)\|_{2.25,-}\leq C,\qquad \min_{0\leq t\leq T_{\epsilon, \kappa}}\min_{x\in\Omega} J(t)\geq C.
\end{equation}
Using elliptic estimates as in Section \ref{subsection_psi}, we have
\begin{equation}\label{boundpsi}
\|\psi(t)\|_{2.5,-}\leq  C(|h(t)|_{2}+1),\qquad \|\psi(t)\|_{3,-}\leq  C(|h(t)|_{2.5}+1)
\end{equation}
Furthermore,
\begin{equation}\label{boundA}
\|A(t)-A(0)\|^2_{1,-}\leq tE(t) \,,
\end{equation}
and using  interpolation once again, we have that
\begin{align}
\|A(t)-A(0)\|_{1.25,-}^2 & \leq C\|A(t)-A(0)\|_{1,-}\|A(t)-A(0)\|_{1.5,-}\leq\sqrt{t}C\sqrt{E(t)} \,, \label{boundA2}\\
\|A(t)-A(0)\|_{1.375,-}^2& \leq \sqrt[4]{t}C\sqrt{E(t)}\,. \label{boundA3}
\end{align}
In particular, by taking a small enough time, our previous \emph{bootstrap} assumption \eqref{bootstrap1} is strengthened. Furthermore, using \eqref{boundA3},
$$
\|A(t)\|_{1.375,-}\leq C.
$$
\subsubsection{Some estimates for lower-order norms of $v$}
Just as in Section \ref{secmaxprinL2}, we have the following $L^2$ energy law:
$$
|\jeps h (t)|^2_0+2\int_0^t\|v(s)\|_0^2ds=|\jeps \jkap h_0|_0^2\,,
$$ 
from which it follows that
\begin{equation}\label{lowv}
2\int_0^t\|v(s)\|_{0,-}^2ds\leq |h_0|_0^2.
\end{equation}

\subsubsection{The estimates for the pressure}
The elliptic problem for $q$  is
\begin{alignat*}{2}
-(A^i_jA^k_jq_{,k})_{,i} & =0  && \text{ in }\Omega \,,\\
q& =0 && \text{ on }\Gamma\,, \\
q,_k A^k_j A^i_j N_i &=\psi^2_{,2} && \text{ on } \Gamma_{bot} \,,
\end{alignat*} 
where we recall that on $\Gamma$, $N= e_2$ while on $\Gamma_{bot}, N=-e_2$.

We have that $A_0 A_0^T$ is symmetric and positive semi-definite: $[A_0 A_0^T]^i_j   \xi _i\xi _j \ge \mathcal{L} | \xi |^2$;  consequently, due to \eqref{boundA3},
$$
\|A_0A_0^T-A(t) A^T(t)\|_{L^\infty}\leq  C \sqrt{t}\sqrt{E(t)}  \,,
$$
and we see that for $t$ sufficiently small, 
$$
{\frac{\mathcal{L} }{2}} |\xi|^2\leq [A (\cdot, t ) A^T(\cdot , t)]^i_j\xi^i\xi^j\leq 2 \mathcal{L} |\xi|^2.
$$

We have that
$$
C\|\nabla q\|_{0,-}^2\leq \int_\Omega A^i_jA^k_jq_{,k}q_{,i} dx=\int_{\Gamma_{bot}}\psi^2_{,2}q ds.
$$

In particular, due to Poincar\'e inequality, there exists a universal constant such that
$$
\|q\|_{1,-}\leq C.
$$
Elliptic estimates (see Lemma \ref{lemaa6}) together with \eqref{boundA2} show that
$$
\|q\|_{2.25,-}\leq C\|\nabla q\|_{L^\infty(\Omega^-)}\leq C\|q\|_{2.125},
$$
and then, using interpolation and Young's inequality, we find the bound
\begin{equation}\label{Q1.25onephase}
\|q\|_{2.25,-}\leq C.
\end{equation}
Thus, once again,  elliptic estimates show that
\begin{equation}\label{Q1.5onephase}
\|q\|_{2.5,-}\leq C\left(1+\|A(t)\|_{1.5,-})\|\nabla q\|_{L^\infty(\Omega^-)}\right)\leq C(1+\|A(t)\|_{1.5,-})  \,,
\end{equation}
 and consequently, 
\begin{equation}\label{v1.5}
\sup_{0\leq t\leq T_{\epsilon,\kappa}} \|v\|_{1.5,-}\leq C(|h^{\kappa\kappa}|_2+1),
\quad
\sup_{0\leq t\leq T_{\epsilon,\kappa}} |h_{t}|_{1}\leq CE(t).
\end{equation}

\subsubsection{The Rayleigh-Taylor stability condition revisited}
By the assumption \eqref{RTonephase} in the Theorem \ref{localonephase}, for $0<\epsilon,\kappa\ll1$ taken sufficiently small,
$$
-\nabla p(0)\cdot \tilde{n}(0)>0\text{ at }\Gamma(t),
$$
so
$$
-A^2_i(0)q_{,2}(0)\tilde{n}^i(0)=-JA^2_i(0)A^2_i(0)q_{,2}(0)>0\text{ at }\Gamma.
$$
In particular,
\begin{equation}\label{lambdaRT}
\lambda=\min_{x_1}-q_{,2}(0)>0\text{ at }\Gamma.
\end{equation}
To simplify notation, we write
$$
B^{ik}(t)=A^i_j(t)A^k_j(t),
$$
and we study the elliptic problem for 
$$
\bar{q}=q(t)-q(0):
$$
\begin{align*}
-(B^{ik}(t)\bar{q}_{,k})_{,i}& =-([B^{ik}(0)-B^{ik}(t)]q_{,k}(0))_{,i}&& \qquad\text{in}\,\Omega\times[0,T_{\epsilon,\kappa}]\\
\bar{q}& =0 && \qquad\text{in}\,\Gamma\times[0,T_{\epsilon,\kappa}]\\
\bar{q}_{,k} B^{ik}(t) N_i& =[B^{ik}(0)-B^{ik}(t)]q_{,k}(0)N_i+ \psi^2_2(t)-\psi^2_2(0) && \qquad \text{in}\,\Gamma_{bot}\times[0,T_{\epsilon,\kappa}].
\end{align*}

Using elliptic estimates together with the estimates \eqref{psi2.25b}, \eqref{boundA2}, \eqref{boundA3} and the smallness condition on the time, we obtain
\begin{eqnarray*}
\|\bar{q}\|_{2,-}&\leq& C\left(\|([B^{ik}(0)-B^{ik}(t)]q_{,k}(0))_{,i}\|_{0,-}+|[B^{ik}(0)-B^{ik}(t)]q_{,k}(0)N_i+ \psi^2_{,2}(t)-\psi^2_{,2}(0)|_{0.5}\right)\\
&\leq& C\left(\|B(0)-B(t)\|_{1.25,-}\|q(0)\|_{2,-}+\|\nabla[B(0)-B(t)]\|_{0,-}\|q(0)\|_{2.25,-}\right.\\
&&\left.+|[B^{2k}(0)-B^{2k}(t)]|_{0.5}|q_{,k}(0)|_{0.75}+|\psi^2_{,2}(t)-\psi^2_{,2}(0)|_{0.5}\right)\\
&\leq& \sqrt{t}\mathcal{P}(E(t))\\
&\leq& \delta_2.
\end{eqnarray*}
We use the inequality
$$
\|fg\|_{r,-}\leq C\|f\|_{r,-}\|g\|_{s,-},\,\,0\leq r\leq s,\,s>1+r 
$$
to find that
$$
\|[B^{ik}(0)-B^{ik}(t)]q_{,ki}(0)\|_{0.25,-}\leq C\|q(0)\|_{2.25,-}\|[B^{ik}(0)-B^{ik}(t)]\|_{1.375,-}.
$$
We apply \eqref{boundA3} to find that
$$
\|[B^{ik}(0)-B^{ik}(t)]q_{,ki}(0)\|_{0.25,-}\leq \sqrt[8]{t}\mathcal{P}(E(t)).
$$
This is the only place where the bound \eqref{boundA3} plays an essential role. For any other smallness estimate concerning $A(t)-A(0)$ it is enough with \eqref{boundA2}.

We want a bound showing the smallness of $\bar{q}_{,2}$ pointwise on $\Gamma$. As a result, we need an estimate stronger than just $H^2$. We focus our attention then in $H^{2.25}$. Elliptic regularity then shows that 
\begin{align} 
\|\bar{q}\|_{2.25,-}&\leq C\left(\|([B^{ik}(0)-B^{ik}(t)]q_{,k}(0))_{,i}\|_{0.25,-}+|[B^{ik}(0)-B^{ik}(t)]q_{,k}(0)N_i+ \psi^2_{,2}(t)-\psi^2_{,2}(0)|_{0.75}\right.\nonumber\\
& \qquad \left.+(1+\|B(t)\|_{1.25,-})\|\nabla \bar{q}\|_{L^\infty(\Omega^-)}\right)\nonumber\\
&\leq \sqrt[8]{t}\mathcal{P}(E(t))\nonumber\\
&\leq \delta_2\label{qbar}.
\end{align} 
Consequently, on $\Gamma$, we have that
$$
-q_{,2}(x_1,t)=-q_{,2}(x_1,t)+q_{,2}(x_1,0)-q_{,2}(x_1,0)\geq -q_{,2}(x_1,0)-C\delta_2,
$$
and our \emph{bootstrap} assumption \eqref{bootstrap4} is satisfied:
$$
-\min_{x_1}q_{,2}(x_1,t)\geq -\min_{x_1}q_{,2}(x_1,0)-C\delta_2\geq -\frac{\min_{x_1}q_{\kappa,2}(x_1,0)}{2}.
$$

\subsubsection{The estimate for  $h\in L^2 ( 0,T_ \kappa ; H^{2.5} (\Gamma))$}
From equation (\ref{HS_ALE_Onephase_reg}a), we see that
$$
v\cdot \tau=-\tau\cdot e_2\; \text{ at }\Gamma. 
$$
It follows that
$$
-\frac{v'\cdot \tau}{\tilde{n}\cdot e_2+v\cdot \tilde{n}}=-\frac{v'\cdot \tau}{1+h_{ t}}=\frac{h^{\kappa\kappa\prime\prime}}{g^{3/2}}.
$$
Thus,
\begin{eqnarray*}
h^{\kappa\kappa\prime\prime}&=&-\frac{v_{1}'+h^{\kappa\kappa\prime}v_{2}'}{1+h_t}(1+(h^{\kappa\kappa\prime})^2)\\
&=&-(v_{ 1}'+h^{\kappa\kappa\prime}v_{ 2}')(1+(h^{\kappa\kappa\prime})^2)+\frac{(v_{ 1}'+h^{\kappa\kappa\prime}v_{ 2}')h_{ t}}{1+h_{ t}}(1+(h^{\kappa\kappa\prime})^2),
\end{eqnarray*}
and, using \eqref{v1.5},
$$
\int_0^t|h^{\kappa\kappa}|_{2.5}^2ds\leq CE(t).
$$

\subsubsection{The energy estimates}
We write (\ref{HS_ALE_Onephase}a) as
$$
v^i+ A^k_i(q_{,k}+\psi^2)_{,k}=0 \text{ in }\Omega.
$$
We take two horizontal derivatives of this expression, test against $v''$ and integrate by parts to find that
$$
\int_0^t\int_{\Omega^-}|v''|^2 dxdy+\mathfrak{I}_1+\mathfrak{I}_2+\mathfrak{I}_3=0.
$$
The higher-order terms are
\begin{align*} 
\mathfrak{I}_1& =\int_0^t\int_{\Omega^-}A^k_i (q+\psi\cdot e_2)''_{,k}(v^i)''dxdy, \\
\mathfrak{I}_2&=\int_0^t\int_{\Omega^-}(A^k_i)'' (q+\psi\cdot e_2)_{,k}(v^i)''dx dy,\end{align*} 
while 
$$
\mathfrak{I}_3=2\int_0^t\int_{\Omega^-}(A^k_i)' (q+\psi\cdot e_2)'_{,k}(v^i)''dx dy.
$$ 
is the lower-order term. Integrating by parts in the term $I_1$ and using $JA^k_iN^k=\sqrt{g}n_i$, we obtain
$$
\mathfrak{I}_1=\mathfrak{J}_1+\mathfrak{J}_2,
$$
with
\begin{align*} 
\mathfrak{J}_1&=-\int_0^t\int_{\Omega^-}(q+\psi\cdot e_2)''(A^k_i (v^i)'')_{,k}dx dy, \\
\mathfrak{J}_2&=\int_0^t\int_\Gamma \psi''\cdot e_2 J^{-1}(v''\cdot \tilde{n}) dsdy=\int_0^t\int_\Gamma J^{-1} \jeps\jeps h'' (v''\cdot \tilde{n}) dsdy.
\end{align*} 
Using the Piola identity  $(JA^k_i)_{,k} =0$ and the divergence-free condition $v^i,_k A^k_i=0$, we see that
$$
(A^k_i (v'')^i)_{,k}=(A^k_i)_{,k} (v'')^i+A^k_i (v'')^i_{,k}=-J_{,k} A^k_i J^{-1}(v'')^i-(A^k_i)'' v^i_{,k}-2(A^k_i)'(v^i)'_{,k},
$$
and $\mathfrak{J}_1=\mathfrak{K}_1+\mathfrak{K}_2+\mathfrak{K}_3$ where
\begin{align*} 
\mathfrak{K}_1&=\int_0^t\int_{\Omega^-}(q+\psi^2)''(A^k_i)'' v^i_{,k}dx dy, \\
\mathfrak{K}_2&=\int_0^t\int_{\Omega^-}(q+\psi^2)''2 (A^k_i)' (v^i)'_{,k}dx dy,\\
\mathfrak{K}_3&=\int_0^t\int_{\Omega^-}(q+\psi^2)''J_{,k}J^{-1}A^k_i(v^i)''dx dy
\end{align*} 
The term $\mathfrak{K}_2$ can be easily bounded using \eqref{propA}, \eqref{boundpsi} and \eqref{Q1.5onephase} together with the Sobolev embedding theorem:
$$
|\mathfrak{K}_2|\leq C\int_0^t\|v\|_{2,-}\|A'\|_{L^4}\left(\|q\|_{2.5,-}+\|\psi\cdot e_2\|_{2.5,-}\right)dy\leq \sqrt{t}\mathcal{P}(E(t)).
$$
To bound the term $\mathfrak{K}_3$, we use H\"{o}lder's inequality with an $L^2-L^4-L^4-L^\infty$ bound, we have that
$$
\mathfrak{K}_3\leq \sqrt{t}\mathcal{P}(E(t)).
$$

The term $\mathfrak{K}_1$ can be simplified using \eqref{propA}; we write
$\mathfrak{K}_1=\mathfrak{L}_1+\mathfrak{L}_2$,
with
\begin{align*} 
\mathfrak{L}_1&=-\int_0^t\int_{\Omega^-}(q+\psi^2)''(2A'\nabla\psi' A)^k_i v^i_{,k}dx dy, \\
\mathfrak{L}_2&=-\int_0^t\int_{\Omega^-}(q+\psi^2)''A^k_j\psi^j_{,11r} A^r_i v^i_{,k}dx dy,
\end{align*} 
where we recall that $\psi_{,11}=\psi''$. $\mathfrak{L}_1$ is estimated using H\"{o}lder's inequality and the Sobolev embedding theorem:
\begin{eqnarray*}
|\mathfrak{L}_1|&\leq& \int_0^t (\|q\|_{2.5,-}+\|\psi\|_{2.5,-})\|A\|_{L^\infty}\|v\|_{1.5,-}\|A\|_{1.5,-}\|\nabla \psi\|_{1.5,-}dy\\
&\leq & C\sqrt{t}\mathcal{P}(E(t)).
\end{eqnarray*}
Similarly, 
\begin{eqnarray*}
|\mathfrak{L}_2|&\leq& C(\| q\|_{2.5,-}+\|\psi^2\|_{2.5,-})\|A\|_{L^\infty}^2\sqrt{t}\left(\int_0^t\|\psi(y)\|_{3,-}^2dy\right)^{0.5}\|v\|_{1.5,-}\\
&\leq&\sqrt{t}\mathcal{P}(E(t)).
\end{eqnarray*}

Next, using \eqref{propA}, we write
$
\mathfrak{I}_2=\mathfrak{K}_4+\mathfrak{K}_5$, where
\begin{align*} 
\mathfrak{K}_4&=-\int_0^t\int_{\Omega^-}A^k_j\psi^j_{,11r} A^r_i (q+\psi^2)_{,k}(v^i)''dx dy,\\
\mathfrak{K}_5&=-\int_0^t\int_{\Omega^-}2(A')^k_j\psi^j_{,1r} A^r_i (q+\psi^2)_{,k}(v^i)''dx dy.
\end{align*}  
We have that
$$
|\mathfrak{K}_5|\leq \int_0^t C\|A\|_{1.5,-}\|\nabla\psi\|_{1.5,-}\|A\|_{L^\infty}\| \nabla (q+\psi^2)\|_{L^\infty}\|v\|_{2,-}dy\leq \sqrt{t}\mathcal{P}(E(t)).
$$
For $\mathfrak{K}_4$, we integrate-by-parts and write $\mathfrak{K}_4=\mathfrak{L}_3+\mathfrak{L}_4$, where
\begin{align*} 
\mathfrak{L}_3&=\int_0^t\int_{\Omega^-}\psi^j_{,11} (A^k_jA^r_i (q+\psi^2)_{,k}(v^i)'')_{,r}dx dy, \\
\mathfrak{L}_4&=-\int_0^t\int_{\Gamma}\psi^j_{,11} A^k_jA^r_i (q+\psi^2)_{,k}(v^i)''N^rds.
\end{align*} 
We further decompose $\mathfrak{L}_3$ as $\mathfrak{L}_3=\mathfrak{M}_1+\mathfrak{M}_2+\mathfrak{M}_3$, where
\begin{align*} 
\mathfrak{M}_1&=\int_0^t\int_{\Omega^-}\psi^j_{,11} A^k_{j,r}A^r_i (q+\psi^2)_{,k}(v^i)''dx dy, \\
\mathfrak{M}_2&=\int_0^t\int_{\Omega^-}\psi^j_{,11} A^k_{j}A^r_{i,r} (q+\psi^2)_{,k}(v^i)''dx dy, \\
\mathfrak{M}_3&=\int_0^t\int_{\Omega^-}\psi^j_{,11} A^k_{j}A^r_i (q+\psi^2)_{,rk}(v^i)''dx dy, \\
\mathfrak{M}_4&=\int_0^t\int_{\Omega^-}\psi^j_{,11} A^k_{j}A^r_i (q+\psi^2)_{,k}(v^i)''_{,r}dx dy.
\end{align*} 
For the first three terms,
\begin{eqnarray*}
|\mathfrak{M}_1|+|\mathfrak{M}_2|+|\mathfrak{M}_3|&\leq& \int_0^t\|\nabla\psi\|_{1.5,-}\|v\|_{2,-}\|A\|_{L^\infty}\left[\|A\|_{1.5,-}(\|\nabla q\|_{1.25,-}\right.\\
&&+\left.\|\nabla\psi\|_{1.25,-})+\|A\|_{L^\infty}(\|\nabla q\|_{1.5,-}+\|\nabla\psi\|_{1.5,-})\right]dy\\
&\leq& \sqrt{t}\mathcal{P}(E(t)).
\end{eqnarray*}
In the term $\mathfrak{M}_4$, we use $v^i,_k A^k_i =0$  and write
$\mathfrak{M}_4=\mathfrak{N}_1+\mathfrak{N}_2$, where
\begin{align*} 
\mathfrak{N}_1&=-\int_0^t\int_{\Omega^-}\psi^j_{,11} A^k_{j}(A^r_i)'' (q+\psi^2)_{,k} v^i_{,r}dx, \\
\mathfrak{N}_2&=-2\int_0^t\int_{\Omega^-}\psi^j_{,11} A^k_{j}(A^r_i)' (q+\psi^2)_{,k} v^i_{,1r}dx. 
\end{align*} 
These terms can be estimated in the same fashion as the term $K_1$ above.  Also,
\begin{eqnarray*}
|\mathfrak{N}_1|&\leq& \sqrt{t}C\left(\int_0^t\|\psi(y)\|_{3,-}^2dy\right)^{0.5}\|v\|_{1.5,-}\|\nabla\psi\|_{1.5,-}(\|\nabla q\|_{1.25,-}+\|\nabla \psi\|_{1.25,-})\\
&\leq& \sqrt{t}\mathcal{P}(E(t)),
\end{eqnarray*}
and
\begin{eqnarray*}
|\mathfrak{N}_2|&\leq& \sqrt{t}C\left(\int_0^t\|v(y)\|_{2,-}^2dy\right)^{0.5}\|\nabla \psi \|_{1.5,-}^2(\|\nabla q\|_{1.25,-}+\|\nabla \psi\|_{1.25,-})\\
&\leq& \sqrt{t}\mathcal{P}(E(t)).
\end{eqnarray*}
The term $\mathfrak{I}_3$ can be bounded using H\"{o}lder's inequality and the Sobolev embedding theorem:  
$$
|\mathfrak{I}_3|\leq \sqrt{t}\mathcal{P}(E(t)).
$$
We next analyze the boundary integrals. We have that
$$
BI=\mathfrak{J}_2+\mathfrak{L}_4=\int_0^t\int_{\Gamma}(\psi''\cdot (v+e_2))((v^i)'' \tilde{n}_iJ^{-1})ds.
$$
To estimate this terms we will extensively use the lower bound for $J$. We write $BI=\mathfrak{O}_1+\mathfrak{O}_2+\mathfrak{O}_3$, where
\begin{align*} 
\mathfrak{O}_1&=\int_0^t\int_{\Gamma}(\psi''\cdot (v+e_2))h_t''J^{-1}dsdy, \\
\mathfrak{O}_2&=-\int_0^t\int_{\Gamma}(\psi''\cdot (v+e_2))(v \cdot \tilde{n}'' J^{-1})dsdy \\
\mathfrak{O}_3&=-2\int_0^t\int_{\Gamma}(\psi''\cdot (v+e_2))(v' \cdot \tilde{n}'J^{-1})dsdy.
\end{align*} 
The inequality $|v|_{1}\leq C\|v\|_{1.5,-}$ together with the embedding $H^{0.25}(\Gamma)\subset L^4(\Gamma)$ shows that
$$
|\mathfrak{O}_3|\leq C(|v|_1^2+1)\int_0^t |h^{\kappa\kappa}|^2_{2.25}dy\leq \sqrt{t}\mathcal{P}(E(t)).
$$
The term $\mathfrak{O}_2$ reads
$$
\mathfrak{O}_2=\int_0^t\int_{\Gamma}h^{\kappa\kappa\prime\prime}(v_2+1)(v_1 h^{\kappa\kappa\prime\prime\prime})J^{-1}dsdy.
$$
By forming an exact derivative, integrating-by-parts and using \eqref{boundpsi}, we see that
$$
|\mathfrak{O}_2|\leq C\int_0^t |h^{\kappa\kappa\prime\prime}|_{L^3}^2|\nabla\psi'|_{L^3}dy\leq C\int_0^t |h^{\kappa\kappa}|_{2+1/6}^2\|\psi\|_{2+2/3,-}dy\leq C\int_0^t |h^{\kappa\kappa}|_{2+1/6}^3dy.
$$
Consequently, due to the interpolation inequality
$$
|h^{\kappa\kappa}|_{2+1/6}^3\leq C|h^{\kappa\kappa}|_{2}^2|h^{\kappa\kappa}|_{2.5},
$$
we find that
$$
|\mathfrak{O}_2|\leq \sqrt{t}\mathcal{P}(E(t)).
$$
Using $[(v+e_2)\cdot \tau]=0$ and $\sqrt{g}n_i=JA^k_iN^k$ , the term $\mathfrak{O}_1$ can be written as
\begin{align*}
\mathfrak{O}_1&=\int_0^t\int_{\Gamma}(\psi''\cdot [(v+e_2)\cdot n] n)h_t''J^{-1}dsdy\\
&=\int_0^t\int_{\Gamma}(\psi''\cdot [-A^2_iq_{,2} (\sqrt{g})^{-1}A^2_i] n)h_t''dsdy\\
&=\int_0^t\int_{\Gamma}\psi''\cdot [-q_{,2} ] \tilde{n}h_t'' J^{-2}dsdy\\
&=\int_0^t\int_{\Gamma}h^{\kappa\kappa\prime\prime}[-q_{,2}]h_t''J^{-2}dsdy\\
&=\int_0^t\int_{\Gamma}h^{\kappa\kappa\prime\prime}\left[\frac{-q_{,2}(t)}{J^{-2}(t)}+\frac{q_{,2}(0)}{{J^{-2}(0)}}-\frac{q_{,2}(0)}{J^{-2}(0)}\right]h_t''dsdy\\
&=\mathfrak{P}_1+\mathfrak{P}_2+\mathfrak{P}_3.
\end{align*}
Using \eqref{qbar}, 
\begin{align*}
|\mathfrak{P}_1| & =\left|\int_0^t\int_{\Gamma}h^{\kappa\kappa\prime\prime}J^{-2}(t)[q_{,2}(t)-q_{,2}(0)]h_t''dsdy\right|\\
& \leq C\int_0^t|h^{\kappa\kappa}|_{2.5}\|J^{-2}\|_{L^\infty}\|q_{,2}(t)-q_{,2}(0)\|_{1.25,-}|h_t|_{1.5}dy\\
& \leq \delta_2CE(t).
\end{align*}
The second error term can be bounded in the same way using \eqref{J1.25b}:
\begin{align*}
|\mathfrak{P}_2| & =\left|\int_0^t\int_{\Gamma}h^{\kappa\kappa\prime\prime}q_{,2}(0)[J^{-2}(t)-J^{-2}(0)]h_t''dsdy\right|\\
& \leq C\int_0^t|h^{\kappa\kappa}|_{2.5}\|J(t)-J(0)\|_{1.25,-}|h_t|_{1.5}dy\\
& \leq \delta_2CE(t).
\end{align*}
Finally, $\mathfrak{P}_3=\mathfrak{Q}_1+\mathfrak{Q}_2$ with
\begin{align} 
\mathfrak{Q}_1&=\int_0^t\int_{\Gamma}h^{\kappa\prime\prime}[\jeps(-q_{,2}(0)J^{-2}(0)h_t'')-[-q_{,2}(0)J^{-2}(0)]\jeps h_t'']dsdy,  \label{Qone}\\
\mathfrak{Q}_2&=\int_0^t\int_{\Gamma}h^{\kappa\prime\prime}[-q_{,2}(0)J^{-2}(0)] h_t^{\kappa\prime\prime}dsdy. \nonumber
\end{align} 
The term $\mathfrak{Q}_1$ can be bounded using Proposition \ref{commutator}:
$$
|\mathfrak{Q}_1|\leq \int_0^t|h^{\kappa\kappa}|_2|q_{,12}(0)J^{-2}(0)+q_{,2}(0)J^{-3}(0)J_{,1}(0)|_{L^{\infty}}|h^{\kappa}_t|_1
dy.
$$
The term $|q_{,12}(0)|_{L^{\infty}}$ can be bounded (using standard elliptic estimates) in terms of the initial data as long as the initial data verifies $|\jkap h_0|_{2.5+s}<\infty$, $s>0$. The same situation arises when dealing with $J_{,1}(0)$. Consequently, this term $\mathfrak{Q}_1$ requires $\epsilon>0$, and, in this latter case, we have
$$
|\mathfrak{Q}_1|\leq\sqrt{t}\mathcal{P}(E(t)).$$

Recalling \eqref{Jmin} and \eqref{lambdaRT}, the term $Q_2$ gives us an energy term
$$
C\frac{\lambda}{2}\left[|h^{\kappa\prime\prime}|_0^2-|h_0^{\kappa\epsilon\prime\prime}|_0^2\right]\leq \frac{1}{2}\int_{\Gamma}-q_{,2}(0)J^{-2}(0)[(h^{\kappa\prime\prime})^2-(h_0^{\kappa\epsilon\prime\prime})^2]ds \,;
$$
hence, 
\begin{equation}\label{almostdone}
\int_0^t\|v(y)\|_{0,-}^2+\|v''(y)\|_{0,-}^2dy+|h^{\kappa}(t)|_2^2\leq \mathcal{M}_0+\sqrt[12]{t}\mathcal{P}(E(t)),
\end{equation}
where $\mathcal{M}_0$ is a number depending only on the initial data, $h_0$, and the value of the regularizing parameter $\epsilon>0$.

\subsubsection{The Hodge decomposition elliptic estimates} Since in each phase,  $ \operatorname{curl} u=0$, it follows that
$ v^2,_j A^j_1 - v^1,_jA^j_2 =0$.   Therefore,
$$
(A^j_1(t)-A^j_1(0))v^2_{,j}-(A^j_2(t)-A^j_2(0))v^1_{,j}
=-A^j_1(0)v^2_{,j}+A^j_2(0)v^1_{,j} \,,
$$
so that
$$
\|A^j_1(0)v^2_{,j}-A^j_2(0)v^1_{,j}\|_{1,-}\leq C\|A(t)-A(0)\|_{L^\infty}\|v\|_{2,-}+\|A(t)-A(0)\|_{1.5,-}\|v\|_{1.5,-},
$$
and
$$
\int_0^t\|A^j_1(0)v^2_{,j}(y)-A^j_2(0)v^1_{,j}(y)\|_{1,-}^2dy\leq\sqrt{t}\mathcal{P}(E(t)).
$$
Similarly, since in each phase $v^j,_i A^i_j =0$, 
$$
[A^i_j(t)-A^i_j(0)]v^j_{,i}=-A^i_j(0)v^j_{,i},
$$
and
$$
\int_0^t\|A^i_j(0)v^j_{,i}(y)\|_{1,-}^2dy\leq\sqrt{t}\mathcal{P}(E(t)).
$$
Finally,
$$
|v_2|_{1.5}\leq |v''\cdot N|_{-0.5}\leq C\|v''\|_{0,-}\leq \mathcal{M}_0+\sqrt[12]{t}\mathcal{P}(E(t)).
$$
Applying Proposition \ref{Hodge2}, we obtain
\begin{equation}\label{vH2}
\int_0^t\|v(y)\|_{2,-}^2dy\leq \mathcal{M}_0+\sqrt[12]{t}\mathcal{P}(E(t)).
\end{equation}
\eqref{vH2} together with \eqref{almostdone} and the properties of the mollifiers gives us the bound
$$
E(t)\leq \mathcal{M}_0+\sqrt[12]{t}\mathcal{Q}(E(t)),
$$
with $E(t)$ being a continuous function. Thus, we infer the existence of $T^*_\epsilon$ such that
$$
E(t)\leq 2\mathcal{M}_0\,\,\forall 0\leq t\leq T^*_\epsilon.
$$
Notice that $T^*_\epsilon$ depends only on $\epsilon$ and $h_0$.
\subsubsection{Passing to the limit and uniqueness}
Once the uniform bounds are obtained, we can pass to the limit $\kappa\rightarrow0$ in the standard way using Rellich-Kondrachov theorem. 

\subsection{$\epsilon$-independent estimates} In the above analysis, only the integral $Q_1$ in (\ref{Qone}) depends on our smoothing parameter $\epsilon>0$; nevertheless, upon  passing to the limit $\kappa\rightarrow0$, the integral $Q_1$ no longer appears. The main point is that the regularizing effect due to $\epsilon>0$ was only necessary because of $\kappa>0$. As $\kappa=0$, we can now close the estimates and tend $\epsilon$ to zero.

After taking the limit in $\kappa$, we have a solution to the following system
\begin{align*}
v_{\epsilon}^i+ (A_{\epsilon})^k_i(q_{\epsilon}+\psi_{\epsilon}^2)_{,k}&= 0 \qquad&&\text{in}\quad\Omega\times[0,T_{\epsilon}]\,,\\
(A_{\epsilon})^i_j(v_{\epsilon})^j_{,i} &= 0 &&\text{in}\quad\Omega\times[0,T_{\epsilon}]\,,\\
h_{\epsilon}(t)&=\jkap h_0+\int_0^t v_{\epsilon}^i \tilde{n}_i ds\qquad&&\text{on}\quad\Gamma\times[0,T_{\epsilon}]\,,\\
q_{\epsilon}&=0 &&\text{on}\quad\Gamma\times[0,T_{\epsilon}]\,,\\
v_{\epsilon}\cdot e_2&=0 &&\text{on}\quad\Gamma_{bot}\times[0,T_{\epsilon}]\,,\\
\psi_{\epsilon}&=\phi^{\epsilon}_2\circ\phi^{\epsilon}_1&&\text{in}\quad\Omega\times\{t=0\}\,,\\
\Delta \psi_{\epsilon}(t)&=\Delta \psi_{\epsilon}(0)&&\text{in}\quad\Omega\times[0,T_{\epsilon}]\,,\\
\psi_{\epsilon}(t)&=e+h_\epsilon(t)N&&\text{on}\quad\Gamma\times[0,T_{\epsilon}]\,,\\
\psi_{\epsilon}(t)&=e&&\text{on}\quad\Gamma_{bot}\times[0,T_{\epsilon}] \,,
\end{align*}
and $\phi^{\epsilon}_2$ and $\phi^{\epsilon}_2$ are given by
\begin{equation*}
\phi^{\epsilon}_1(x_1,x_2)=\left(x_1,x_2+\jdel\jkap h_0(x_1)\left(1-\frac{x_2}{c_b}\right)\right),
\end{equation*}
and
\begin{align*}
\Delta \phi^{\epsilon}_2&=0\qquad &&\text{on}\quad \Omega^{\delta,\epsilon}(0)\times[0,T_{\epsilon}]\,,\\
\phi^{\epsilon}_2(t)&= e+[\jkap h_0(x_1)-\jdel\jkap h_0(x_1)]e_2 \qquad &&\text{on}\quad \Gamma\times[0,T_{\epsilon}],\,,\\
\phi^{\epsilon}_2 &= e &&\text{on}\quad \Gamma_{bot}\times[0,T_{\epsilon}] \,.
\end{align*}
Now we define the energy
$$
E(t)=\max_{0\leq s\leq t}|h(s)|_{2}+\int_0^t\|v(s)\|_{2,-}^2ds.
$$
We repeat the energy estimates. The only modification affects the term $\mathfrak{O}_1$, that now reads
\begin{align*}
\mathfrak{O}_1&=\int_0^t\int_{\Gamma}(\psi''\cdot [(v+e_2)\cdot n] n)h_t''J^{-1}dsdy\\
&=\int_0^t\int_{\Gamma}(\psi''\cdot [-A^2_iq_{,2} (\sqrt{g})^{-1}A^2_i] n)h_t''dsdy\\
&=\int_0^t\int_{\Gamma}\psi''\cdot [-q_{,2} ] \tilde{n}h_t'' J^{-2}dsdy\\
&=\int_0^t\int_{\Gamma}h''[-q_{,2}]h_t''J^{-2}dsdy\\
&=\int_0^t\int_{\Gamma}h''\left[\frac{-q_{,2}(t)}{J^{-2}(t)}+\frac{q_{,2}(0)}{{J^{-2}(0)}}-\frac{q_{,2}(0)}{J^{-2}(0)}\right]h_t''dsdy\\
&=\mathfrak{P}_1+\mathfrak{P}_2+\mathfrak{P}_3.
\end{align*}
These terms can be bounded in a straightforward way. We get the polynomial estimate
$$
E(t)\leq \mathcal{M}_0+\sqrt[12]{t}\mathcal{Q}(E(t)),
$$
and the existence of $T^*$ such that
$$
E(t)\leq 2\mathcal{M}_0\,\,\forall 0\leq t\leq T^*.
$$
This $T^*$ only depends on the initial data $h_0$. Now, we can pass to the limit $\epsilon\rightarrow0$ using Rellich-Kondrachov theorem. The uniqueness is obtained using the energy method as in Section \ref{sec4}. This concludes with the proof of Theorem \ref{localonephase}.

\section{Proof of Theorem \ref{Cinftyonephase}: Instantaneous parabolic smoothing}\label{sec6}
The proof of this result is a two-step procedure. First, we show that we always can gain an extra half derivative almost everywhere in time. The second step of the argument is a classical bootstrapping procedure.

\subsection{Two-phase Muskat problem}  We begin with the  two-phase case, and consider  initial data $h_{\delta0}\in H^3$ for the infinitely-deep Muskat problem (\ref{HS_Eulerian}a-e)) or the confined Muskat problem (\ref{HS_Eulerian}a-d,e',f) satisfying the smallness criterion \eqref{cgs1} in Theorem \ref{localsmall}. 

We define the higher-order energy function
\begin{equation}\label{energyH3}
E(t)=\max_{0\leq s\leq t}\{|h(s)|_3^2\}+\int_0^t\|w(s)\|_{3,\pm}^2ds.
\end{equation}
Repeating our energy estimates using three tangential derivatives rather than two, we obtain the polynomial inequality
$$
E(t)\leq C|h_{\delta0}|_3^2+\sqrt{t}\mathcal{P}(E(t)).
$$
As a consequence, there exists a time $T^*$ such that we have the bound
$$
\max_{0\leq s\leq T^*}\{|h(s)|_3^2\}+\int_0^t|h(s)|_{3.5}^2ds\leq C|h_{\delta0}|_3^2.
$$
Interpolating with the bound obtained in Theorem \ref{localsmall}, 
we have that
\begin{equation}\label{cinfty2}
\max_{0\leq s\leq T^*}\{|h(s)|_{2.5}^2\}+\int_0^t|h(s)|_{3}^2ds\leq C|h_{\delta0}|_{2.5}^2.
\end{equation}
Now, given $h_0\in H^2$ satisfying the smallness condition \eqref{cgs1}, due to Theorem \ref{localsmall}, we have a solution 
$h\in C([0,T^*],H^2)\cap L^2(0,T^*;H^{2.5}(\Gamma))$. In particular, we can choose $0<\delta\leq T^*$ arbitrarily small so that
$h(\delta)=h_{\delta0}\in H^{2.5}(\Gamma)$ and verifies the smallness criterion \eqref{cgs1}. 
We are going to use $h_{\delta0}$ as the new initial data for the problem. Applying \eqref{cinfty2}, we have thus that the initial data $h_{\delta0}$ provides us with a solution
$$
h_\delta\in C([\delta,T^*]H^{2.5}(\Gamma))\cap L^2(\delta ,T^*;H^{3}(\Gamma)).
$$ 
Due to the uniqueness of solution proved in Theorem \ref{localsmall}, we conclude that the original initial data $h_0$ gives us a solution 
$$
h\in C([0,T^*],H^2(\Gamma))\cap C([\delta,T^*],H^{2.5}(\Gamma))\cap L^2(\delta,T^*; H^{3}(\Gamma))
$$ 
for an arbitrarily small $\delta>0$. Now we proceed by bootstrapping. We can repeat the argument and show that for every positive time, we have that the unique solution in Theorem \ref{localsmall} is 
$$
h(\cdot,t)\in C^\infty(\Gamma)\,\,\text{ if }\delta\leq t\leq T^*,\,\,\forall\delta>0.
$$

\subsection{One-phase Muskat problem} For the one-phase Muskat problem (\ref{HS_Eulerian_Onephase}a-e), we consider $h_{\delta0}\in H^3$ satisfying the Rayleigh-Taylor stability condition \eqref{RTstable}. Once again redoing the energy estimates with three tangential derivatives, there exists a time 
$T^*$, and the bound
$$
\max_{0\leq s\leq T^*}\{|h(s)|_3^2\}+\int_0^t|h(s)|_{3.5}^2ds\leq C|h_{\delta0}|_3^2.
$$
Interpolating with the bound obtained in Theorem \ref{localonephase}, we obtain the bound \eqref{cinfty2}. Now, given $h_0\in H^2(\Gamma)$ satisfying the Rayleigh-Taylor stability condition \eqref{RTstable}, due to Theorem \ref{localonephase}, we have a solution $h\in C([0,T^*],H^2(\Gamma))\cap C([\delta,T^*],H^{2.5}(\Gamma))\cap L^2(\delta,T^*; H^{3}(\Gamma))$.
By bootstrapping,  we see that  $h(\cdot , t)\in C^\infty(\Gamma)$ if $t\geq\delta>0$.

\section*{Acknowledgements}
AC was supported by the Ministry of Science and Technology (Taiwan) under grant MOST-103-2115-M-008-010-MY2 and by the National Center 
of Theoretical Sciences.   RGB was supported by OxPDE via the EPSRC grant EP/I01893X/1.
 SS was supported by the National Science Foundation under grant and DMS-1301380, by OxPDE at the University of Oxford, 
and by the Royal Society Wolfson Merit Award.     Some of this work was completed 
during the program {\it Free Boundary Problems and Related Topics} at the Isaac Newton Institute for Mathematical Sciences at Cambridge, UK.   We are
grateful to the organizers,  Gui-Qiang Chen, Henrik Shahgholian and Juan Luis V\'{a}zquez,   for both the invitation to participate in the program and to contribute to  this special volume.

\appendix
\section{Auxiliary results}

\subsection{The $H^{d/2}$-norm of products}
We need the following
\begin{proposition}\label{H0.5_fg}
For all $\delta > 0$, there exists $C_\delta > 0$ such that
\begin{equation*}
|fg|_{0.5} \le C_\delta |f|_{0.5+\delta} |g|_{0.5}\,.
\end{equation*}
and, in two dimensions,
\begin{equation*}
\|fg\|_{1,\pm} \le C_\delta \|f\|_{1+\delta,\pm} \|g\|_{1,\pm}\,.
\end{equation*}
\end{proposition}
\begin{proof}
The $L^2$ part can be bounded as follows:
\begin{equation}\label{app1}
|fg|^2_0\leq \|f\|_{L^\infty(\bbR)}^2|g|_0^2\leq C_\delta|f|_{0.5+\delta}^2|g|_{0.5}^2,
\end{equation}
where we have used the Sobolev embedding
$$
H^{0.5+\delta}(\bbR)\hookrightarrow L^\infty(\bbR).
$$
The seminorm term can be bounded using Kato-Ponce inequality for $\Lambda=\sqrt{-\partial_x^2}$
$$
|\Lambda^{0.5}(fg)|_0\leq C_\delta\left(\|g\|_{L^{\frac{1}{\delta}}(\bbR)}\|\Lambda^{0.5} f\|_{L^{\frac{2}{1-2\delta}}(\bbR)}+\|f\|_{L^{\infty}(\bbR)}\|\Lambda^{0.5} g\|_{L^{2}(\bbR)}\right).
$$
The Sobolev embeddings
$$
H^{\delta}(\bbR)\hookrightarrow L^q(\bbR),\,q\in\left[2,\frac{2}{1-2\delta}\right],\,H^{0.5}(\bbR)\hookrightarrow L^q(\bbR),\,q\in\left.\left[2,\infty\right.\right),
$$
give us
\begin{equation}\label{app2}
|\Lambda^{0.5}(fg)|_0\leq C_\delta\|g\|_{0.5}\|f\|_{0.5+\delta}.
\end{equation}
Collecting the estimates \eqref{app1} and \eqref{app2}, we conclude the first statement. With the same ideas and the embedding
$$
H^{\delta}(\bbR^2)\hookrightarrow L^q(\bbR^2),\,q\in\left[2,\frac{2}{1-\delta}\right],\,H^{1}(\bbR^2)\hookrightarrow L^q(\bbR^2),\,q\in\left.\left[2,\infty\right.\right),
$$
we conclude the result.
\end{proof}
\subsection{The Hodge decomposition elliptic estimates}
\begin{proposition}\label{Hodge}Let $\Omega$ be a domain with boundary $\partial\Omega$ of Sobolev class $H^{k+0.5}$. 
Then for $v \in H^k(\Omega) $, 
\begin{equation*}
\|v\|_{H^k(\Omega)} \le C \Big[\|v\|_{L^2(\Omega)} + \|\curl v\|_{H^{k-1}(\Omega)} + \|\div v\|_{H^{k-1}(\Omega)} + \|v\cdot N\|_{H^{k-0.5}(\bdy\Omega)}\Big]\,,
\end{equation*}
where $N$ denotes the outward unit normal to $\bdy \Omega$.
\end{proposition}
\begin{proposition}\label{Hodge2}Let $\Omega$ be a domain with boundary $\partial\Omega$ of Sobolev class $H^{k+0.5}$. Let $\psi_0$ be a given smooth mapping and define
$$
\curl_{\psi_0} v=\curl(v\circ\psi_0)=(A_0)^j_1(v\circ\psi_0)^2_{,j}-(A_0)^j_2(v\circ\psi_0)^1_{,j},
$$
$$
\div_{\psi_0} v=\div(v\circ\psi_0)=(A_0)^i_j(v\circ\psi_0)^j_{,i},
$$
where $A_0=(\nabla\psi_0)^{-1}$. Then for $v \in H^k(\Omega) $, 
\begin{equation*}
\|v\|_{H^k(\Omega)} \le C \Big[\|v\|_{L^2(\Omega)} + \|\curl_{\psi_0} v\|_{H^{k-1}(\Omega)} + \|\div_{\psi_0} v\|_{H^{k-1}(\Omega)} + \|v\cdot N\|_{H^{k-0.5}(\bdy\Omega)}\Big]\,,
\end{equation*}
where $N$ denotes the outward unit normal to $\bdy \Omega$.
\end{proposition}
The proof of Propositions \ref{Hodge}  and \ref{Hodge2} are given in Cheng \& Shkoller \cite{ChSh2014}.

\begin{proposition}\label{normaltrace} Suppose that $v'\in L^2(\Omega)$ with $\text{div}v\in L^2(\Omega)$. Then $v'\cdot N \in H^ {-\frac{1}{2}} (\bdy\Omega)$
and
\begin{equation*}
\|v' \cdot N\|_{H^{-1/2}(\bdy\Omega)}\leq C\left(\|v' \|_{L^2(\Omega)}+\| \operatorname{div} v \|_ { L^2(\Omega) }  \right).
\end{equation*}
\end{proposition}

\subsection{A commutator estimate}
The following is Lemma 5.1 in Coutand \& Shkoller \cite{CoSh2010}:
\begin{proposition}\label{commutator}Let $\Omega$ be a domain and assume that its boundary, $\partial\Omega$, is smooth. Then 
\begin{equation*}
|\jeps(fg')-f\jeps g'|_0\leq C\|f\|_{W^{1,\infty}}|g|_0.
\end{equation*}
\end{proposition}
\subsection{An elliptic estimate}
Let's consider 
$$
\Omega=\bbT\times[-1,0],
$$
and the elliptic problem
\begin{subequations}\label{elliptic}
\begin{alignat}{2}
- \operatorname{div} (A \nabla u)&= f \qquad&&\text{in}\quad\Omega\,,\\
 u &= 0 &&\text{on}\quad \partial\Omega\,.
\end{alignat}
\end{subequations}
Then, we have the following elliptic estimate
\begin{lemma}\label{lemaa6} Suppose that the matrix $A\in H^{1.5}(\Omega)$ with  $A > 0$, and that $f\in H^{0.5}(\Omega)$. 
Then the solution to (\ref{elliptic}a-b) verifies
$$
\|\Lambda^{1.25}\nabla u\|_{L^2(\Omega) }\leq C\left(\|\Lambda^{0.25}f\|_{L^2(\Omega)}+\|\Lambda^{1.25}A\|_{L^2(\Omega)}\|\nabla u\|_{L^\infty(\Omega)}\right.\left.
+\|\Lambda^{0.5}\nabla u\|_{L^2(\Omega)}\|\Lambda^{0.25}\nabla A\|_{L^2(\Omega)}\right) \,,
$$
and
$$
\|\Lambda^{1.5}\nabla u\|_{L^2(\Omega) }\leq C\left(\|\Lambda^{0.5}f\|_{L^2(\Omega)}+\|\Lambda^{1.5}A\|_{L^2(\Omega)}\|\nabla u\|_{L^\infty(\Omega)}\right.\left.
+\|\Lambda^{0.75}\nabla u\|_{L^2(\Omega)}\|\Lambda^{0.25}\nabla A\|_{L^2(\Omega)}\right) \,.
$$
\end{lemma}
\begin{proof}
We proof only the first estimate, being the second one straightforward. We consider the approximate problem
\begin{subequations}\label{ellipticreg}
\begin{alignat}{2}
- (\tilde{A}^i_j \tilde{u}_{,j})_{,i}&= f \qquad&&\text{in}\quad\Omega\,,\\
 \tilde{u} &= 0 &&\text{on}\quad \partial\Omega\,.
\end{alignat}
\end{subequations}
where $\tilde{A}$ is a $C^\infty$ regularization of $A$. For a given $\phi\in H^1(\Omega),$ we consider the weak formulation of the problem (\ref{ellipticreg}a-b):
$$
\int_\Omega \tilde{A}^i_j \tilde{u}_{,j}\phi_{,i}dx=\int_\Omega f \phi dx.
$$
These problems have solutions $\tilde{u}$ which are smooth. We focus on high norm uniform estimate. To do that, we pick $\phi= \Lambda^3 \tilde{u}$, where $\widehat{\Lambda u} =|k|\hat{u}(k)$. Then, using the self-adjointness of the $\Lambda$ operator, the weak formulation reads
$$
\int_\Omega \Lambda^{1.5}\left(\tilde{A}^i_j \tilde{u}_{,j}\right)\Lambda^{1.5}\tilde{u}_{,i}dx=\int_\Omega \Lambda^{0.5}f \Lambda^{2.5}\tilde{u} dx.
$$
We write
\begin{eqnarray*}
I&=&\int_\Omega \Lambda^{1.5}\left(\tilde{A}^i_j \tilde{u}_{,j}\right)\Lambda^{1.5}\tilde{u}_{,i}dx\\
&=&\int_\Omega [\Lambda^{1.5},\tilde{A}^i_j]\tilde{u}_{,j}\Lambda^{1.5}\tilde{u}_{,i}dx+\int_\Omega \tilde{A}^i_j\Lambda^{1.5}\tilde{u}_{,j}\Lambda^{1.5}\tilde{u}_{,i}dx.
\end{eqnarray*}
Notice that the first term can be estimated by layers (\emph{i.e.} fixing $x_2\in[-1,0]$) using the Kenig-Ponce-Vega estimate (see \cite{kato1988commutator} and \cite{KenigPonceVega}) along the $x_1$ coordinate:
\begin{eqnarray*}
\|[\Lambda^{1.5},\tilde{A}^i_j]\tilde{u}_{,j}\|_{L^2(\bbT)}&\leq& C\left(\|\Lambda^{1.5}\tilde{A}^i_j\|_{L^2(\bbT)}\|\nabla\tilde{u}\|_{L^\infty(\bbT)}+\|\Lambda^{0.5}\nabla\tilde{u}\|_{L^4(\bbT)}\|\nabla A\|_{L^{4}(\bbT)}\right)\\
&\leq& C\left(\|\Lambda^{1.5}\tilde{A}^i_j\|_{L^2(\bbT)}\|\nabla\tilde{u}\|_{L^\infty(\bbT)}+\|\Lambda^{0.75}\nabla\tilde{u}\|_{L^2(\bbT)}\|\Lambda^{0.25}\nabla \tilde{A}\|_{L^{2}(\bbT)}\right)
\end{eqnarray*}
Using Tonelli's theorem, together with $\|\cdot\|_{L^2(\Omega)}^2=\int_{-1}^0\|\cdot\|^2_{L^2(\bbT)}dx_2$, we have
\begin{eqnarray*}
\|[\Lambda^{1.5},\tilde{A}^i_j]\tilde{u}_{,j}\|_{L^2(\Omega)}&\leq& C\left(\|\Lambda^{1.5}\tilde{A}^i_j\|_{L^2(\Omega)}\|\nabla\tilde{u}\|_{L^\infty(\Omega)}+\|\Lambda^{0.75}\nabla\tilde{u}\|_{L^2(\Omega)}\|\Lambda^{0.25}\nabla \tilde{A}\|_{L^2(\Omega)}\right).
\end{eqnarray*}
The second integral provides us with the estimate
$$
\|\Lambda^{1.5}\nabla \tilde{u}\|_{L^2}\leq C\left(\|\Lambda^{0.5}f\|_{L^2(\Omega)}+\|\Lambda^{1.5}\tilde{A}^i_j\|_{L^2(\Omega)}\|\nabla\tilde{u}\|_{L^\infty(\Omega)}\right.\left.
+\|\Lambda^{0.75}\nabla\tilde{u}\|_{L^2(\Omega)}\|\Lambda^{0.25}\nabla \tilde{A}\|_{L^2(\Omega)}\right).
$$
Passing to the limit $\tilde{A}\rightarrow A$, we conclude the desired uniform estimate for $\tilde{u}$.
\end{proof}


\begin{thebibliography}{10}

\bibitem{AmbroseST}
D.~{Ambrose}.
\newblock {The zero surface tension limit of two-dimensional interfacial Darcy
  flow.}
\newblock {\em J. Math. Fluid Mech.}, 16:105--143, 2014.

\bibitem{bear}
J.~Bear.
\newblock {\em {D}ynamics of fluids in porous media}.
\newblock Dover Publications, 1988.

\bibitem{BCG}
L.~Berselli, D.~C\'ordoba, and R.~Granero-Belinch\'on.
\newblock {L}ocal solvability and turning for the inhomogeneous {M}uskat
  problem.
\newblock {\em Interfaces and Free Boundaries}, 16(2):175--213, 2014.

\bibitem{castro2012breakdown}
A.~Castro, D.~Cordoba, C.~Fefferman, and F.~Gancedo.
\newblock {B}reakdown of smoothness for the {M}uskat problem.
\newblock {\em {A}rchive for {R}ational {M}echanics and {A}nalysis},
  208(3):805--909, 2013.

\bibitem{ccfgl}
A.~Castro, D.~Cordoba, C.~Fefferman, F.~Gancedo, and M.~Lopez-Fernandez.
\newblock {R}ayleigh-{T}aylor breakdown for the {M}uskat problem with
  applications to water waves.
\newblock {\em Annals of Math}, 175:909--948, 2012.

\bibitem{ccfgonephase}
{\'A}.~Castro, D.~C{\'o}rdoba, C.~Fefferman, F.~Gancedo, and R.~Orive.
\newblock {S}plash singularities for the one-phase {M}uskat problem in stable
  regimesincompressible flow in porous media with fractional diffusion.
\newblock {\em arXiv:1311.7653 [math.AP]}, 2013.

\bibitem{Castro-Cordoba-Gancedo:recent-results-muskat}
A.~Castro, D.~C{\'o}rdoba, and F.~Gancedo.
\newblock Some recent results on the {M}uskat problem.
\newblock {\em Journ\'ees Equations aux Derivees Partielles}, 2010.

\bibitem{CF}
M.~Cerminara and A.~Fasano.
\newblock {M}odelling the dynamics of a geothermal reservoir fed by gravity
  driven flow through overstanding saturated rocks.
\newblock {\em Journal of Volcanology and Geothermal Research}, 233:37--54,
  2012.

\bibitem{chen1993hele}
X.~Chen.
\newblock The hele-shaw problem and area-preserving curve-shortening motions.
\newblock {\em {A}rchive for {R}ational {M}echanics and {A}nalysis},
  123(2):117--151, 1993.

\bibitem{cheng2012global}
C.~H.~A. Cheng, D.~Coutand, and S.~Shkoller.
\newblock Global existence and decay for solutions of the {H}ele--{S}haw flow
  with injection.
\newblock {\em Interfaces Free Bound.}, 16(3):297--338, 2014.

\bibitem{ChSh2014}
C.~H.~A. {Cheng} and S.~{Shkoller}.
\newblock {Solvability and regularity for an elliptic system prescribing the
  curl, divergence, and partial trace of a vector field on Sobolev-class
  domains}.
\newblock {\em ArXiv e-prints}, Aug. 2014.

\bibitem{ccgs-13}
P.~Constantin, D.~Cordoba, F.~Gancedo, L.~Rodr\'iguez-Piazza, and R.~Strain.
\newblock On the {M}uskat problem: global in time results in 2d and 3d.
\newblock {\em arXiv preprint arXiv:1310.0953}, 2014.

\bibitem{ccgs-10}
P.~Constantin, D.~Cordoba, F.~Gancedo, and R.~Strain.
\newblock On the global existence for the {M}uskat problem.
\newblock {\em Journal of the European Mathematical Society}, 15:201--227,
  2013.

\bibitem{Peter}
P.~Constantin and M.~Pugh.
\newblock {G}lobal solutions for small data to the {H}ele-{S}haw problem.
\newblock {\em Nonlinearity}, 6:393--415, 1993.

\bibitem{c-c-g10}
A.~Cordoba, D.~C{\'o}rdoba, and F.~Gancedo.
\newblock {I}nterface evolution: the {H}ele-{S}haw and {M}uskat problems.
\newblock {\em Annals of Math}, 173, no. 1:477--542, 2011.

\bibitem{c-g07}
D.~C{\'o}rdoba and F.~Gancedo.
\newblock {C}ontour dynamics of incompressible 3-{D} fluids in a porous medium
  with different densities.
\newblock {\em Communications in Mathematical Physics}, 273(2):445--471, 2007.

\bibitem{c-g09}
D.~C{\'o}rdoba and F.~Gancedo.
\newblock {A} maximum principle for the {M}uskat problem for fluids with
  different densities.
\newblock {\em Communications in Mathematical Physics}, 286(2):681--696, 2009.

\bibitem{c-g10}
D.~C{\'o}rdoba and F.~Gancedo.
\newblock {A}bsence of squirt singularities for the multi-phase {M}uskat
  problem.
\newblock {\em Communications in Mathematical Physics}, pages 1--15, 2009.

\bibitem{CGO}
D.~C\'ordoba, R.~Granero-Belinch\'on, and R.~Orive.
\newblock {O}n the confined {M}uskat problem: differences with the deep water
  regime.
\newblock {\em Communications in Mathematical Sciences}, 12(3):423--455, 2014.

\bibitem{cponephase}
D.~Cordoba and T.~Pernas-Casta{\~n}o.
\newblock {N}on-splat singularity for the one-phase {M}uskat problem.
\newblock {\em arXiv:1409.2483 [math.AP]}, 2014.

\bibitem{CoHoSh2013}
D.~Coutand, J.~Hole, and S.~Shkoller.
\newblock Well-posedness of the free-boundary compressible 3-{D} {E}uler
  equations with surface tension and the zero surface tension limit.
\newblock {\em SIAM J. Math. Anal.}, 45(6):3690--3767, 2013.

\bibitem{CoSh2006}
D.~Coutand and S.~Shkoller.
\newblock The interaction between quasilinear elastodynamics and the
  {N}avier-{S}tokes equations.
\newblock {\em Arch. Ration. Mech. Anal.}, 179(3):303--352, 2006.

\bibitem{CoSh2007}
D.~Coutand and S.~Shkoller.
\newblock Well-posedness of the free-surface incompressible {E}uler equations
  with or without surface tension.
\newblock {\em J. Amer. Math. Soc.}, 20(3):829--930, 2007.

\bibitem{CoSh2010}
D.~Coutand and S.~Shkoller.
\newblock A simple proof of well-posedness for the free-surface incompressible
  {E}uler equations.
\newblock {\em Discrete Contin. Dyn. Syst. Ser. S}, 3(3):429--449, 2010.

\bibitem{CoSh2012}
D.~Coutand and S.~Shkoller.
\newblock Well-posedness in smooth function spaces for the moving-boundary
  three-dimensional compressible {E}uler equations in physical vacuum.
\newblock {\em Arch. Ration. Mech. Anal.}, 206(2):515--616, 2012.

\bibitem{CoSh2014}
D.~Coutand and S.~Shkoller.
\newblock On the finite-time splash and splat singularities for the 3-{D}
  free-surface {E}uler equations.
\newblock {\em Comm. Math. Phys.}, 325(1):143--183, 2014.

\bibitem{CoSh20142}
D.~{Coutand} and S.~{Shkoller}.
\newblock {On the impossibility of finite-time splash singularities for vortex
  sheets}.
\newblock {\em ArXiv e-prints}, July 2014.

\bibitem{Parseval-Pillai-Advani:model-variation-permeability}
Y.~De~Parseval, K.~Pillai, and S.~Advani.
\newblock A simple model for the variation of permeability due to partial
  saturation in dual scale porous media.
\newblock {\em Transport in porous media}, 27(3):243--264, 1997.

\bibitem{elliott1982weak}
C.~M. Elliott and J.~R. Ockendon.
\newblock {\em Weak and variational methods for moving boundary problems},
  volume~59.
\newblock Pitman Boston, 1982.

\bibitem{escher2011generalized}
J.~Escher, A.-V. Matioc, and B.-V. Matioc.
\newblock A generalized {R}ayleigh-{T}aylor condition for the {M}uskat problem.
\newblock {\em Nonlinearity}, 25(1):73--92, 2012.

\bibitem{e-m10}
J.~Escher and B.-V. Matioc.
\newblock On the parabolicity of the {M}uskat problem: Well-posedness,
  fingering, and stability results.
\newblock {\em Zeitschrift f{\"u}r Analysis und ihre Anwendungen},
  30(2):193--218, 2011.

\bibitem{ES}
J.~Escher and G.~Simonett.
\newblock {C}lassical solutions for {H}ele-{S}haw models with surface tension.
\newblock {\em Advances in Differential Equations}, 2(4):619--642, 1997.

\bibitem{escher1998center}
J.~Escher and G.~Simonett.
\newblock A center manifold analysis for the {M}ullins-{S}ekerka model.
\newblock {\em {J}ournal of {D}ifferential {E}quations}, 143(2):267--292, 1998.

\bibitem{FeIoLi2013}
C.~{Fefferman}, A.~D. {Ionescu}, and V.~{Lie}.
\newblock {On the absence of ''splash'' singularities in the case of two-fluid
  interfaces}.
\newblock {\em ArXiv e-prints}, Dec. 2013.

\bibitem{F}
A.~Friedman.
\newblock {F}ree boundary problems arising in tumor models.
\newblock {\em Atti Accad. Naz. Lincei Cl. Sci. Fis. Mat. Natur. Rend.
  Lincei,}, 9(3-4), 2004.

\bibitem{gancedo2013splasnqgmsukat}
F.~Gancedo and R.~Strain.
\newblock Absence of splash singularities for SQG sharp fronts and the Muskat
  problem.
\newblock {\em Arxiv preprint arXiv:1309.4023}, 2013.

\bibitem{GG}
J.~G\'omez-Serrano and R.~Granero-Belinch\'on.
\newblock On turning waves for the inhomogeneous Muskat problem: a
  computer-assisted proof.
\newblock {\em Nonlinearity}, 27(6):1471--1498., 2014.

\bibitem{G}
R.~Granero-Belinch{\'o}n.
\newblock Global existence for the confined muskat problem.
\newblock {\em SIAM Journal on Mathematical Analysis}, 46(2):1651--1680, 2014.

\bibitem{HaSh2014}
M.~Had\v{z}i\'{c} and S.~Shkoller.
\newblock Global stability and decay for the classical Stefan problem.
\newblock {\em Communications on Pure and Applied Mathematics}, pages n/a--n/a,
  2014.

\bibitem{hadzic5817well}
M.~{Hadzic} and S.~{Shkoller}.
\newblock {Well-posedness for the classical Stefan problem and the zero surface
  tension limit}.
\newblock {\em ArXiv e-prints}, Dec. 2011.

\bibitem{HeleShaw:motion-viscous-fluid-parallel-plates}
H.~S. Hele-Shaw.
\newblock On the motion of a viscous fluid between two parallel plates.
\newblock {\em Trans. Royal Inst. Nav. Archit.}, 40:218, 1898.

\bibitem{kato1988commutator}
T.~Kato and G.~Ponce.
\newblock Commutator estimates and the {E}uler and {N}avier-{S}tokes equations.
\newblock {\em Communications on Pure and Applied Mathematics}, 41(7):891--907,
  1988.

\bibitem{KenigPonceVega}
C.~E. Kenig, G.~Ponce, and L.~Vega.
\newblock Well-posedness and scattering results for the generalized
  {K}orteweg-{D}e {V}ries equation via the contraction principle.
\newblock {\em Communications on Pure and Applied Mathematics}, 46(4):527--620,
  1993.

\bibitem{kinderlehrer1977regularity}
D.~Kinderlehrer and L.~Nirenberg.
\newblock Regularity in free boundary problems.
\newblock {\em Annali della Scuola Normale Superiore di Pisa-Classe di
  Scienze}, 4(2):373--391, 1977.

\bibitem{kinderlehrer1978smoothness}
D.~Kinderlehrer and L.~Nirenberg.
\newblock The smoothness of the free boundary in the one phase Stefan problem.
\newblock {\em Communications on Pure and Applied Mathematics}, 31(3):257--282,
  1978.

\bibitem{meirmanov1992stefan}
A.~Meirmanov.
\newblock The Stefan problem, volume 3 of de gruyter expositions in
  mathematics, 1992.

\bibitem{Muskat}
M.~Muskat.
\newblock {T}he flow of homogeneous fluids through porous media.
\newblock {\em Soil Science}, 46(2):169, 1938.

\bibitem{NB}
D.~Nield and A.~Bejan.
\newblock {\em {C}onvection in porous media}.
\newblock Springer Verlag, 2006.

\bibitem{P}
C.~Pozrikidis.
\newblock {N}umerical simulation of blood and interstitial flow through a solid
  tumor.
\newblock {\em Journal of Mathematical Biology}, 60(1):75--94, 2010.

\end{thebibliography}
\end{document}